\documentclass[reqno]{amsart}

\usepackage{amsfonts}
\usepackage{amssymb}
\usepackage{amsmath}
\usepackage{enumitem}
\usepackage{xcolor}
\usepackage{todonotes}

\setenumerate{label={\rm (\alph{*})}}
\usepackage{bbm,bm,euscript,mathrsfs}
\usepackage{paper_diening}
\usepackage{graphicx}
\usepackage[top=1in, bottom=1.25in, left=1.10in, right=1.10in]{geometry}
\usepackage{cases, color,amsthm}
\usepackage{hyperref}
\hypersetup{linkcolor=blue, colorlinks=true,citecolor = red}

\allowdisplaybreaks

\numberwithin{equation}{section}

\usepackage{tikz-cd}
\usetikzlibrary{lindenmayersystems}
\usetikzlibrary{decorations.pathreplacing}

\DeclareMathOperator{\Bog}{Bog}

\DeclareMathOperator{\Div}{div}

\DeclareMathOperator{\diver}{div}
\newcommand{\R}{\mathbb R}
\newcommand{\N}{\mathbb N}
\newcommand{\dd}{\mathrm d}
\newcommand{\dt}{\,\mathrm{d} t}

\renewcommand{\bfB}{\mathbf{B}}

\newcommand{\Testzeta}{{\mathscr{F}}^{\Div}_{\eta}}

\newcommand{\bx}{\mathbf{x}}
\newcommand{\by}{\mathbf{y}}
\newcommand{\bz}{\mathbf{z}}

\newcommand{\bu}{\mathbf{v}}

\newcommand{\bv}{\mathbf{v}}
\newcommand{\bn}{\mathbf{n}}
\newcommand{\be}{\mathbf{e}}

\newcommand{\dy}{\, \mathrm{d}\mathbf{y}}

\newcommand{\dx}{\, \mathrm{d} \mathbf{x}}

\newcommand{\divx}{\mathrm{div} }

\newcommand{\nabx}{\nabla }
\newcommand{\naby}{\nabla_{\mathbf{y}}}

\newcommand{\Delx}{\Delta }
\newcommand{\Dely}{\Delta_{\mathbf{y}}}

\newcommand{\dxt}{\,\mathrm{d}\mathbf{x}\,\mathrm{d}t}
\newcommand{\ds}{\,\mathrm{d}s}

\newcommand{\Oeta}{\Omega_\eta}

\newtheorem{theorem}{Theorem}[section]
\newtheorem{lemma}[theorem]{Lemma}
\newtheorem{proposition}[theorem]{Proposition}
\newtheorem{corollary}[theorem]{Corollary}
\newtheorem{remark}[theorem]{Remark}

\theoremstyle{definition}
\newtheorem{definition}[theorem]{Definition}

\begin{document}

\title[FSI with perfectly elastic shells]{Well-posedness theorems in fluid-structure interaction: perfectly elastic shells}

\author{Dominic Breit}
\address{Faculty of Mathematics, University of Duisburg-Essen, Thea-Leymann-Stra{\ss}e 9, 45127 Essen, Germany}
\email{dominic.breit@uni-due.de}

\author{Prince Romeo Mensah}
\address{Faculty of Mathematics, University of Duisburg-Essen, Thea-Leymann-Stra{\ss}e 9, 45127 Essen, Germany}
\email{prince.mensah@uni-due.de}

\author{Sebastian Schwarzacher}
\address{Department of Mathematics,  
	Uppsala University, 
	L\"agerhyddsv\"agen 1,
	752 37 Uppsala, Sweden and Department of Mathematical Analysis,
	Faculty of Mathematics and Physics,
	Charles University,
	Sokolovská 83,
	186 75 Praha 8, Czech Republic}
\email{schwarz@karlin.mff.cuni.cz}

\author{Pei Su}
\address{Department of Mathematical Analysis,
	Faculty of Mathematics and Physics,
	Charles University,
	Sokolovská 83,
	186 75 Praha 8, Czech Republic}
\email{peisu@karlin.mff.cuni.cz}




\date{\today}



\begin{abstract}
%

In this work, we consider the interaction of a 3D incompressible fluid with a 2D flexible shell that occupies (a part of) the boundary of the fluid domain. We assume that the shell is perfectly elastic while the fluid is governed by the Navier--Stokes equations. Consequently, damping within the coupled system comes entirely from the parabolic fluid subsystem.
Our main result is the construction of a local-in-time unique strong solution to the system of PDEs.  Standard techniques from the literature do not apply here. They are restricted to visco-elastic structures, where the corresponding solid phase is parabolic. Our construction relies on a different method built upon a
 new estimate for the acceleration of the system.

In the case of a 2D viscous incompressible fluid interacting with a 1D perfectly elastic shell
 we can extend the local solution globally in time
 (until a possible self-intersection of the shell).

\end{abstract}

\maketitle

\section{Introduction}
The interaction of fluids with elastic structures has been the object of intensive research over the last decades. Significant achievements have been made by mathematicians, physicists and engineers working in the field and various research directions have evolved.
We are interested in the case where a flexible elastic structure occupies (a part of) the boundary of the fluid domain. This means that the fluid domain (or part of it) is
moving in time in accordance with the displacement of the structure. Due to its apparent relevance for applications (blood flow, oil flow, etc.), it is of particular interest in applied mathematics~\cite{sunny,kaltenbacher}.
We distinguish two cases visualised in Figure \ref{fig:2}:
\begin{itemize}
\item \textbf{Elastic plates.} The reference geometry for the fluid domain is flat and  the elastic plate is located on top of a rectangle. In this context, the transformation map between the reference geometry
and the moving domain is particularly simple.
\item \textbf{Elastic shells.} In the case of a general reference geometry, where the structure occupies
(a part of) the curved boundary, we speak about an 
elastic shell. Here, the transformation map is significantly more complicated than for plates.
\end{itemize}
Our main concern in this paper is the analysis of elastic shells but our main result for the three dimensional fluid model (where the structure is two-dimensional) is even new for plates. The reference model is the {\em linearised Koiter shell model}, 
where the elastic part of the equation for
 the solid becomes $\alpha\Dely^2\eta+B \eta $, where $B$ is a
 linear second-order differential operator.\footnote{The term $B$ will later be ignored since it has no impact on the analysis.}
Furthermore, we distinguish between dissipative and non-dissipative solid materials.
\begin{itemize}
\item \textbf{Visco-elastic solids.} The structure equation contains a dissipative term as
$-\gamma\partial_t\Dely\eta$ relating to the viscosity of the shell. 
The solid phase is thus parabolic.
\item \textbf{Perfectly elastic solids.} There is no dissipative term in the equation for the solid. The resulting solid phase is then hyperbolic and this results in regularity incompatibilities with the parabolic fluid phase.
\end{itemize}
All results we are aware of concerning the existence of strong solutions for the three-dimensional
fluid-structure interaction problems involving either  plates or shells consider visco-elastic solids. In this setting, the equation for just the solid is strongly dissipative, as it contains a second order damping term. The latter stabilises solid motions {\em independent of the weak dissipative effects coming from the fluid}. From a mathematical point of view, it leads to better a-priori bounds which are critical for many existence, regularity and uniqueness results.
In particular, short time existence for visco-elastic shells was shown recently by the authors~\cite{BMSS}.
However, it turns out that the developed theory on visco-elastic solids {\em methodologically excludes the treatment of perfectly elastic materials}, which are used for the majority of real-world applications of fluid-structure interactions. 


Our aim, therefore, is
to prove the local-in-time existence of strong solutions of the coupled systems of partial 
differential equations describing the interaction of a viscous incompressible fluid with a perfectly elastic linear shell in the physical case of three space dimensions. 
In the two-dimensional case, we can further prove that this solution exists globally in time 
(until a possible self-intersection of the shell).

\subsection{Bibliographical review}
In the following, we give an overview of recent works on
fluid-structure interactions involving (visco)-elastic structures interacting with an unsteady viscous incompressible fluid.
For a more exhaustive list, we refer to \cite{sunny} for an overview on the setting considered in this paper and to~\cite{kaltenbacher} for a more general overview on mathematical efforts in fluid-structure interactions. 

The theory on the existence of weak solutions is by now essentially equivalent when we speak of shells and plates or visco-elastic solids and perfectly elastic solids. The most relevant research to this current work is the result from \cite{LeRu}, which 
shows existence of a weak solution for a linearised elastic Koiter shell interacting with the incompressible Navier--Stokes equations; namely, for the coupled system we consider in this paper (see Section
\ref{subsec:PDE}  below). We also refer to the earlier work~\cite{CanMuh13} for the simplified model 
of a cylindrical linearised Koiter shell. Finally, in \cite{MuSc} the existence of weak solutions was shown
for the fully nonlinear model as suggested in the original literature by Koiter.

For analytical properties of the coupled system beyond the existence of weak solutions the difference between 3D models (with 2D structure) and 2D models
 (with 1D structure) becomes critical.

\textbf{3D models.} In the three dimensional case, one can 
expect the existence of more regular solutions and uniqueness
 in a short time horizon. An unrestricted global-in-time well-posedeness result is still out of reach -- 
even for 3D Navier--Stokes equation in fixed domains -- and
one may argue whether it can even be expected.
For short times on the other hand there exists various results that consider bulk elastic bodies immersed
in viscous incompressible fluids~\cite{coutand2005motion,coutand2006interaction,Kuk,
raymond2014fluid}. 
 Concerning thin  solids we refer to~\cite{cheng2007navier, CS}, where the interaction between a viscous incompressible fluid and a nonlinear elastic shell is studied and the existence of a local strong solution is shown.
However, inertia are ignored in both works in the 3D case leading to a static equation for the dynamics
of the structure. Thus the results differ substantially to ours.
More related to our work are the papers \cite{DRR} and \cite{BMSS}.
In \cite{DRR} the existence of a local strong solution to the interaction problem
with a linear visco-elastic plate is proved. This is extended to shells in \cite{BMSS}
(see also \cite{BM} and \cite{mensah2025strong} for higher topologies).
 Both \cite{DRR} and \cite{BMSS} rely on a similar strategy that builds on the methodology developed in~\cite{GraHil}, which heavily
benefits from the dissipation of the structure. The case of a perfectly elastic 
structure interacting with a three dimensional fluid is open for both plates and shells and needs a different approach that we introduce in this work.

\textbf{2D models.}
If the fluid motion is planar, one may expect a longer existence time for strong solutions,
as regularity and uniqueness was shown for the Navier--Stokes equations in fixed domains in the seminal work of Ladyshenskaya \cite{La}.
Indeed, a corresponding result for visco-elastic plates/beams was shown in the pioneering paper \cite{GraHil}. More precisely, the existence
of a global strong solution to the interaction problem with a linear 
visco-elastic plate is shown. This result has been extended by the first author
in \cite{Br} to the case of a visco-elastic shells (up to the point of self intersection).
Both papers rely on an acceleration estimate which uses heavily, the parabolic character
of the solid phase.
More recently, the third and fourth authors in \cite{ScSu} developed a different strategy to show strong solutions for perfectly elastic plates for arbitrary times, which was not accessible by previous methods.
The strategy developed in \cite{ScSu} builds the starting point of the investigations here, which include long time existence results for the previously untouched regime of perfectly elastic shells.

\subsection{The system of PDEs} \label{subsec:PDE}
We are interested in the interaction of an incompressible fluid with a flexible
 shell where the shell  reacts to the surface forces induced by the fluid and 
deforms the spatial reference domain $\Omega \subset \mathbb{R}^n$ to 
$\Omega_{\eta(t)}$ with respect to a coordinate transform
 $\bfvarphi_{\eta(t)}$ (see Figure \ref{fig:2} for the typical situation
 and  Section~\ref{ssec:geom} for the precise set-up). The displacement is proportional to
the normal $\bfn$ at the boundary of the reference geometry
with amplitude 
$\eta$.
The function $\eta:I \times \omega \mapsto   \mathbb{R}$
with $I=(0,T)$ for some $T>0$ solves
the equation\footnote{We identify the boundary $\partial\Omega$ with the flat torus $\omega$
here.}
 \begin{equation}\label{1}
\left\{\begin{aligned}
& \varrho_s\partial_t^2\eta + \alpha\Dely^2\eta=g-\bn^\intercal\bm{\tau}\circ\bm{\varphi}_\eta\bn_\eta
 \vert \mathrm{det}(\naby \bm{\varphi}_{\eta})\vert
&\text{ for all }  (t,\by)\in I\times\omega
 ,\\
&\eta(0,\by)=\eta_0(\by), \quad (\partial_t\eta)(0, \by)=\eta_*(\by)
&\text{ for all } \by\in\omega.
 	\end{aligned}\right.
 \end{equation}
 The vector $\bfn_\eta$ is the normal vector
of the deformed boundary and 
$\bfvarphi_{\eta(t)}:\omega\rightarrow\partial\Omega_{\eta(t)}$ is
a parametrisation of the moving boundary.
By $\bftau$ we denote the Cauchy stress of the fluid
 given by Newton's rheological law, {\em i.e.},
$$\bftau=\mu\big(\nabx\bu+(\nabx\bu)^\intercal\big)-\pi\mathbb I_{n\times n},$$
where $\mu>0$ is the viscocity of the fluid and $n=2,3$. Here
$\bu:I\times\Omega_\eta\rightarrow\R^n$ and 
$\pi:I\times\Omega_\eta\rightarrow\R$ are the velocity field and hydrodynamical pressure,
respectively.
The parameters $\varrho_s$ and $\alpha$ are positive constants and the function 
$g: I \times \omega \mapsto   \mathbb{R}$ is a given
 forcing term. The motion of the fluid is governed by the Navier--Stokes equations
 \begin{equation}\label{2}
\left\{\begin{aligned}
 &\varrho_f\big(\partial_t \bu  + (\mathbf{v}\cdot \nabx)\mathbf{v} \big)
 = 
 \mu\Delx \bu -\nabx\pi+ \bff &\text{ for all }(t,\bx)\in I\times\Omega_\eta,\\
 &\Div \bu=0&\text{ for all }(t,\bx)\in I \times\Omega_\eta,\\
 &\bu(0,\bx)=\bu_0(\bx) &\text{ for all } \bx\in \Omega_{\eta_0},
 \end{aligned}\right.
 \end{equation}
in the moving space-time cylinder $I\times\Omega_\eta$.
Here $\varrho_f$ is a positive constant representing the density of the fluid
 and the function $\bff:I \times \Oeta \mapsto\mathbb{R}^n$ is a given volume force. The equations \eqref{1} and \eqref{2} are coupled through the kinematic boundary condition
\begin{align}
\label{interfaceCond}
\bu\circ \bfvarphi_\eta=\partial_t\eta\bfn \quad\text{ for all } (t,\by)\in I\times \omega.
\end{align} 
%
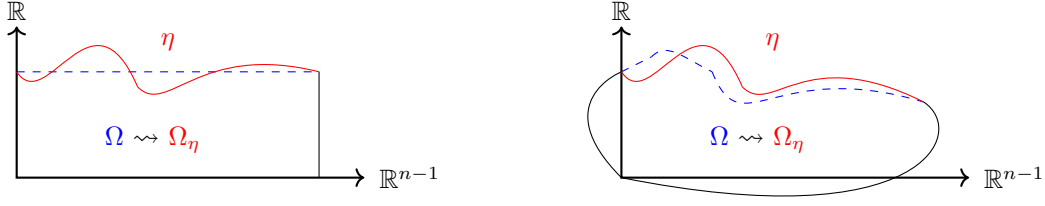
\begin{figure}
\begin{center}
\begin{tikzpicture}[scale=2]
  \begin{scope}   
     \draw [thick, <->] (0.5,0.5) -- (0.5,-0.5) -- (2.8,-0.5);
        \node at (0.5,0.6) {$\R$};
       \node at (3.1,-0.5) {$\R^{n-1}$};
       \draw [blue,dashed] (0.5,0.2) -- (2.5,0.2);
                \draw (2.5,0.2) -- (2.5,-0.5);
                \draw [red] (0.5,0.2) .. controls (0.7,-0.1) and (1,0.8) .. (1.3,0.1);
                \draw [red] (1.3,0.1) .. controls (1.5,-0.1) and  (1.7,0.4) .. (2.5,0.2);
                        \node at (1.4,-0.25) {\textcolor{blue}{$\Omega$} $\rightsquigarrow$ \textcolor{red}{$\Omega_\eta$}};
                         \node [red] at (1.5,0.4) {$\eta$};
        \draw [thick, <->] (4.5,0.5) -- (4.5,-0.5) -- (6.8,-0.5);
        \node at (4.5,0.6) {$\R$};
       \node at (7.1,-0.5) {$\R^{n-1}$};
                     \draw (4.5,-0.5) .. controls (4,0) and (4.5,0.2) .. (4.5,0.2);
       \draw (6.5,0) .. controls (6.8,-0.2) and (6.5,-0.9) .. (4.5,-0.5);
        \draw [blue,dashed] (4.5,0.2) .. controls (4.7,0.3) and (4.7,0.3) .. (4.7,0.3);
        \draw [blue,dashed] (4.7,0.3) .. controls (4.8,0.4) and (4.9,0.3) .. (5.1,0.2);
       \draw [blue,dashed] (5.1,0.2) .. controls (5.3,-0.3) and (5.5,0.3) .. (6.5,0);
                \draw [red] (4.5,0.2) .. controls (4.7,-0.1) and (5,0.8) .. (5.3,0.1);
                \draw [red] (5.3,0.1) .. controls (5.5,-0.1) and  (5.7,0.4) .. (6.5,0);
                        \node at (5.4,-0.25) {\textcolor{blue}{$\Omega$} $\rightsquigarrow$ \textcolor{red}{$\Omega_\eta$}};
                         \node [red] at (5.5,0.4) {$\eta$};
  \end{scope}
\end{tikzpicture}
\caption{Domain transformation for elastic plates (left) and shells (right).}\label{fig:2}
\end{center}
\end{figure}
A weak solutions $(\eta,\bu)$ to \eqref{1}--\eqref{interfaceCond} 
can be constructed to satisfy the energy inequality and thus
\begin{equation*}
\sup_{I}\|\partial_t\eta\|_{L^2_\by}^2+\sup_{I}\|\Dely\eta\|_{L^2_\by}^2+\sup_{I}\|\bu\|_{L^2_\bx}^2+\int_{I}\|\nabx\bu\|_{L^2_\bx}^2\dt<\infty.
\end{equation*}
We speak about a strong solution if all quantities in 
\eqref{1} and \eqref{2} belong to the class $L^\infty_t(L^2)$.
The precise definitions can be found in Section \ref{sec:main}. In a simplified version,
 our main result reads as follows (see Theorems \ref{thm:NS3} and \ref{thm:NS2} for the precise statements).
\begin{theorem}\label{thm:mainsimple}
\begin{enumerate}
\item
{\em (Local strong solution 3D)} Under natural assumptions on the data  (see Theorem \ref{thm:NS3}), there exists $T_{\tt max}>0$ and a unique strong 
solution to \eqref{1}--\eqref{interfaceCond} which exists in $(0,T_{\tt max})$.
\item {\em(Global strong solution 2D)} If $n=2$ then this solution exists
until the self-intersection of the shell (see Theorem \ref{thm:NS2}). If there is no self-intersection 
it exists in $(0,\infty)$.
\end{enumerate}
\end{theorem}

\subsection{Strategy of the proof}
At the heart of our analysis is an $L^\infty_t$-estimate for the acceleration
$\frac{1}{2}\|\partial_t\bu\|_{L^2_\bx}^2+\frac{1}{2}\|\partial_t^2\eta\|_{L^2_\by}^2$.
It can be formally achieved by differentiating the system (transformed to the reference
geometry) and testing it by $(\partial_t^2\eta,\partial_t\bu)$. This strategy was successfully used in \cite{ScSu} for perfectly elastic plates in 2D. It yields
substantially more regularity than the $L^2_t$-acceleration estimate previously used in \cite{GraHil} and
\cite{Br} for the visco-elastic case, which was not even sufficient in the case of elastic plates in 2D.  It stands to wonder if related arguments could be used to estimate the nonlinear terms appearing in 3D.
At first glance, the method from \cite{ScSu} has a severe disadvantage: It can only be performed
after transforming the fluid equation to the fixed-in-time reference geometry. This transformation
leads to a perturbed Navier--Stokes type system with coefficients depending on
the structure displacement. 
In the two-dimensional case the embeddings at hand
are sufficiently strong to close
the estimate. Note, in particular, that one-dimensional embeddings show that the structure function
is a priori Lipschitz continuous in space uniformly in time. This is not the case for the 3D problem. Here higher order a-priori estimates are only possible locally in time. For that we follow the standard strategy to construct strong short time solutions via contraction principles. This means that the analysis is performed on the {\em decoupled problem}, where the geometry of the fluid domain is decoupled from the interaction problem. We then proceed as follows.
\begin{enumerate}
\item[1)] We start with a geometric prescription given by the mollified coordinate map $\zeta_\varepsilon$ describing the geometry. An a priori $L^\infty_t$-acceleration estimate is obtained
independently of $\varepsilon$, with constants depending only on
certain well-chosen quantities associated with $\zeta$. In order to do this rigorously we additionally include artificial dissipation in the structure equation.

Unfortunately, the resulting system (which is a variant of that from \cite{BMSS}) still does not possess enough regularity to justify in-time differentiation, which is essential for the estimate. To overcome this, we approximate the time derivative by temporal difference quotients. The resulting estimate is by far the biggest effort of the paper and can be found in Proposition~\ref{thm:transformedSystemLinear}. It is uniform with respect to the regularisation and thus allow to construct the solution map $F:\zeta\to F(\zeta)=\eta$, where $\eta$ is the structure displacement of the decoupled interaction problem.
\item[2)] We show that there exists a normed space and a large enough ball $B_R$ in this space, such that one can choose a small enough time $T_{\tt max}$ for which $F(\zeta)\in B_R$. This bound involves a non-standard Gr\"{o}nwall lemma including integrals of arbitrary high powers of the acceleration on the right-hand side and the use of elliptic regularity theory for the Stokes system with coefficients c.f.~\cite{Br}. 
\item[3)] We then show that this mapping is a contraction, such that a unique fixed point (a unique solution) can be found. Here it turns out
to be beneficial to work with the fluid equation in the moving domain
in order to avoid the appearance of the pressure function.
The idea is then to prove a stability estimate for two solutions arising from two 
different geometries. The main difficulty is to compare the two velocity
fields which are defined on different (moving) domains. It can be overcome
by transforming one of the solutions to the geometry of the second one. 
This destroys its solenoidal character but can be corrected by means
of a version of the Bogovskij operator in moving domains, c.f.~\cite{SaaSch21}. This part builds partially on \cite{BMSS}, where a weak-strong uniqueness result for
the visco-elastic problem was shown.
\end{enumerate}
In the 2D case we close the $L^\infty_t$-acceleration estimate directly for a strong solution, {\em i.e.}, no additional approximation layer is needed. However, in contrast to the plate case~\cite{ScSu} we first have to prove (fractional) spatial differentiability of $\eta$, see Lemma \ref{lem:4.2}. Combined with the $L^\infty_t$-acceleration estimate, the argument can then be closed by a standard Gr\"{o}nwall inequality. 
After proving  the $L^\infty_t$-acceleration estimate we can then bound $\nabla\pi$ and $\nabla^2\bv$ by elliptic regularity theory, 
c.f.~\cite{Br}. Eventually, this yields additional spatial regularity for the structure.

\section{Preliminaries \& main results}
\label{sec:prelim}
\subsection{Conventions}
For simplicity, we set all physical constants in \eqref{1}--\eqref{interfaceCond} to 1. The analysis is not affected as long as they are strictly positive.
For two non-negative quantities $f$ and $g$, we write $f\lesssim g$  if there is a $c>0$ such that $f\leq\,c g$. Here $c$ is a generic constant which does not depend on the crucial quantities. If necessary, we specify particular dependencies. We write $f\approx g$ if both $f\lesssim g$ and $g\lesssim f$ hold. In the notation for function spaces (see next subsection), we do not distinguish between scalar- and vector-valued functions. However, vector-valued functions will usually be denoted in bold case.
For simplicity, we supplement \eqref{1} with periodic boundary conditions and identify $\omega$ (which represents the complete boundary of $\Omega$) with  $(0,1)^2$ (or $(0,1)$ if $n=2$). We consider periodic function spaces for zero-average functions.
It is only a technical matter to consider \eqref{1} on a nontrivial subset of $\partial\Omega$ together with zero boundary conditions for $\eta$ and $\naby\eta$ instead of considering, see e.g. \cite{LeRu} or \cite{BrSc} for the corresponding geometrical set-up.
We shorten the time interval $(0,T)$ by $I$.

\subsection{Classical function spaces}
Let $\mathcal O\subset\R^n$, $n=2,3$, be open.
The function spaces of continuous or $\alpha$-H\"older-continuous functions, $\alpha\in(0,1)$,
 are denoted by $C(\overline{\mathcal O})$ or $C^{0,\alpha}(\overline{\mathcal O})$ respectively, where $\overline{\mathcal O}$ is the closure of $\mathcal O$. Similarly, we write $C^1(\overline{\mathcal O})$ and $C^{1,\alpha}(\overline{\mathcal O})$.
We denote  by $L^p(\mathcal O)$ and $W^{k,p}(\mathcal O)$ for $p\in[1,\infty]$ and $k\in\mathbb N$, the usual Lebesgue and Sobolev spaces over $\mathcal O$. 
For a bounded domain $\mathcal O\subset\R^n$,  the notation $(f)_{\mathcal O}:=\dashint_{\mathcal O}f\dx:=\mathcal L^n(\mathcal O)^{-1}\int_{\mathcal O}f\dx$ represents the mean or average value of $f\in L^p(\mathcal O)$, where $\mathcal L^n$ denotes the $n$-dimensional Lebesgue measure.
 We denote by $W^{k,p}_0(\mathcal O)$, the closure of the smooth and compactly supported functions in $W^{k,p}(\mathcal O)$. If $\partial\mathcal O$ is regular enough, this coincides with the functions vanishing $\mathcal H^{n-1}$ -a.e. on $\partial\mathcal O$ with the $(n-1)$-dimensional Hausdorff measure $\mathcal H^{n-1}$. 
 We also denote by $W^{-k,p'}(\mathcal O)$ the dual of $W^{k,p}_0(\mathcal O)$ for $p,p'\in[1,\infty)$, with $\frac{1}{p}+\frac{1}{p'}=1$.
  Finally, we consider subspaces
$W^{1,p}_{\Div}(\mathcal O)$ and $W^{1,p}_{0,\Div}(\mathcal O)$ of divergence-free vector fields which are defined accordingly. The space $L^p_{\Div}(\mathcal O)$ is defined as the closure of the set of smooth and compactly supported solenoidal functions in $L^p(\mathcal O).$ We will use the shorthand notations $L^p_\bx$ and $W^{k,p}_\bx$ in the case of $n$-dimensional domains (typically spaces defined over $\Omega\subset\R^n$ or $\Omega_\eta\subset\R^n$) and   
$L^p_\by$ and $W^{k,p}_\by$ for $(n-1)$- dimensional sets (typcially spaces of periodic functions defined over $\omega\subset\R^{n-1}$). 
For any pair of separable Banach spaces $(X,\|\cdot\|_X)$ and $(Y,\|\cdot\|_Y)$ with $X\subset Y$, we write $X\hookrightarrow Y$ if $X$ is continuously embedded in $Y$, that is $\Vert\cdot\Vert_Y\lesssim \Vert \cdot\Vert_X$.
Since we only consider functions on $\omega$ with periodic boundary conditions and zero mean values, we have the following equivalences
\begin{align*}
\|\cdot\|_{W^{1,2}_\by}\approx \|\nabla_\by\cdot\|_{L^2_\by},\quad \|\cdot\|_{W^{2,2}_\by}\approx \|\Dely\cdot\|_{L^2_\by},\quad \|\cdot\|_{W^{4,2}_\by}\approx \|\Dely^2\cdot\|_{L^2_\by}.
\end{align*}
Similarly, it holds for $s>0$
\begin{align*}
\|\cdot\|_{W^{2s,2}_\by}\approx \|\Dely^{s}\cdot\|_{L^2_\by},
\end{align*}
where the operator $\Dely^{s}$ is defined as follows. 
For a separable Banach space $(X,\|\cdot\|_X)$, we denote by $L^p(I;X)$, the set of (Bochner-) measurable functions $u:I\rightarrow X$ such that the mapping $t\mapsto \|u(t)\|_{X}$ belongs to $L^p(I)$. 
The set $C(\overline{I};X)$ denotes the space of functions $u:\overline{I}\rightarrow X$ which are continuous with respect to the norm topology on $(X,\|\cdot\|_X)$. For $\alpha\in(0,1]$ we write
$C^{0,\alpha}(\overline{I};X)$ for the space of H\"older-continuous functions with values in $X$. The space $W^{1,p}(I;X)$ consists of those functions from $L^p(I;X)$ for which the distributional time derivative belongs to $L^p(I;X)$ as well. The space $W^{k,p}(I;X)$ is defined accordingly.
We use the shorthand $L^p_tX$ for $L^p(I;X)$. For instance, we write $L^p_tW^{1,p}_\bx$ for $L^p(I;W^{1,p}(\mathcal O))$. Similarly, $W^{k,p}_tX$ stands for $W^{k,p}(I;X)$.
For $p\in[1,\infty)$, the fractional Sobolev space (Sobolev-Slobodeckij space) with differentiability $s>0$ with $s\notin\mathbb N$ will be denoted by $W^{s,p}(\mathcal O)$. For $s>0$, we write $s=\lfloor s\rfloor+\lbrace s\rbrace$ with $\lfloor s\rfloor\in\N_0$ and $\lbrace s\rbrace\in(0,1)$.
 We denote by $W^{s,p}_0(\mathcal O)$, the closure of the smooth and compactly supported functions in $W^{s,p}(\mathcal O)$. For $s>\frac{1}{p}$ this coincides with the functions vanishing $\mathcal H^{n-1}$ -a.e. on $\partial\mathcal O$ provided that $\partial\mathcal O$ is regular enough. We also denote by $W^{-s,p'}(\mathcal O)$, for $s>0$ and $p,p'\in[1,\infty)$, with $\frac{1}{p}+\frac{1}{p'}=1$, the dual of $W^{s,p}_0(\mathcal O)$. Similar to the case of unbroken differentiabilities above, we use the shorthand notations $W^{s,p}_\bx$  and $W^{s,p}_\by$. 
We will denote by $B^s_{p,q}(\R^n)$, the standard Besov spaces on $\R^n$ with differentiability $s>0$, integrability $p\in[1,\infty]$ and fine index $q\in[1,\infty]$. They can be defined (for instance) via Littlewood-Paley decomposition leading to the norm $\|\cdot\|_{B^s_{p,q}(\R^n)}$. 
 We refer to \cite{RuSi} and \cite{Tr,Tr2} for an extensive description. 
For a bounded domain $\mathcal O\subset\R^n$, the Besov spaces $B^s_{p,q}(\mathcal O)$ are defined as the restriction of functions from $B^s_{p,q}(\R^n)$, that is
 \begin{align*}
 B^s_{p,q}(\mathcal O)&:=\{f|_{\mathcal O}:\,f\in B^s_{p,q}(\R^n)\},\\
 \|g\|_{B^s_{p,q}(\mathcal O)}&:=\inf\{ \|f\|_{B^s_{p,q}(\R^n)}:\,f|_{\mathcal O}=g\}.
 \end{align*}
 If $s\notin\mathbb N$ and $p\in(1,\infty)$ we have $B^s_{p,p}(\mathcal O)=W^{s,p}(\mathcal O)$.

\subsection{Function spaces on variable domains}
\label{ssec:geom}
 The spatial domain $\Omega$ is assumed to be an open bounded subset of $\mathbb{R}^n$, $n=2,3$ with smooth boundary $\partial\Omega$ and an outer unit normal ${\bfn}$. We assume that
 $\partial\Omega$ can be parametrised by an injective mapping ${\bfvarphi}\in C^k(\omega;\R^n)$ for some sufficiently large $k\in\N$. If $n=3$ we suppose for all points $\by=(y_1,y_2)\in \omega$ that the pair of vectors  
$\partial_i {\bfvarphi}(\by)$, $i=1,2,$ are linearly independent, if $n=2$ we must have $\partial_y\bfvarphi(\by)\neq0$.
 For a point $\bx$ in the neighbourhood
of $\partial\Omega$, we define the functions $\by$ and $s$ by  
\begin{align*}
 \by(\bx)=\arg\min_{\by\in\omega}|\bx-\bfvarphi(\by)|,\quad s(\bx)=(\bx-\by(\bx))\cdot\bfn(\by(\bx)).
 \end{align*}
Moreover, we define the projection $\bfp(\bx)=\bfvarphi(\by(\bx))$. We define $L>0$ to be the largest number such that $s,\by$ and $\bfp$ are well-defined on $S_L$, where
\begin{align}\label{eq:boundary1}
S_L=\{\bx\in\R^n:\,\mathrm{dist}(\bx,\partial\Omega)<L\}.
\end{align}
	Accordingly, the mapping
\begin{align}\label{eq:Lambda}
\Lambda:\partial\Omega\times(-L,L)\times S_L,\quad (\bfp,s)\mapsto \bfp+s\bfn(\bfvarphi^{-1}(\bfp)),
\end{align}
is a $C^k$-diffeomorphism.
Due to the smoothness of $\partial\Omega$ for $L$ small enough we have $\abs{s(\bx)}=\min_{\by\in\omega}|\bx-\bfvarphi(\by)|$ for all $\bx\in S_L$. This implies that $S_L=\{s\bfn(\by)+\by:(s,\by)\in (-L,L)\times \omega\}$.
For a given function $\eta : I \times \omega \rightarrow\R$ we parametrise the deformed boundary by
\begin{align*}
{\bfvarphi}_\eta(t,\by)={\bfvarphi}(\by) + \eta(t,\by){\bfn}(\by), \quad \,\by \in \omega,\,t\in I.
\end{align*}
The domain $\Omega_{\eta(t)}$ is defined through
\begin{equation}\label{eq:2612}
\partial\Omega_{\eta(t)}=\set{{\bfvarphi}(\by) + \eta(t,\by){\bfn}(\by):\by\in \omega}.
\end{equation}
With the abuse of notation we define deformed space-time cylinder $I\times\Omega_\eta=\bigcup_{t\in I}\set{t}\times\Omega_{\eta(t)}\subset\R^{n+1}$.
The corresponding function spaces for variable domains are defined as follows.
\begin{definition}{(Function spaces)}
For $I=(0,T)$, $T>0$, and $\eta\in C(\overline{I}\times\omega)$ with $\|\eta\|_{L^\infty(I\times\omega)}< L$ we define for $1\leq p,r\leq\infty$
\begin{align*}
L^p(I;L^r(\Omega_\eta))&:=\big\{v\in L^1(I\times\Omega_\eta):\,\,v(t,\cdot)\in L^r(\Omega_{\eta(t)})\,\,\text{for a.e. }t,\,\,\|v(t,\cdot)\|_{L^r(\Omega_{\eta(t)})}\in L^p(I)\big\},\\
L^p(I;W^{1,r}(\Omega_\eta))&:=\big\{v\in L^p(I;L^r(\Omega_\eta)):\,\,\nabla v\in L^p(I;L^r(\Omega_\eta))\big\}.
\end{align*}
\end{definition}
\noindent 
In order to establish a relationship between the 
fixed domain and the time-dependent domain, we introduce the Hanzawa transform $\bm{\Psi}_\eta : \Omega \rightarrow\Omega_\eta$ defined by
\begin{equation}
\label{map}
\bfPsi_\eta(\bx)
=
 \left\{
  \begin{array}{lr}
    \mathbf{p}(\bx)+\big(s(\bx)+\eta(\by(\bx))\phi(s(\bx))\big)\bn(\by(\bx)) &\text{if dist}(\bx,\partial\Omega)<L,\\
    \bx &\text{elsewhere},
  \end{array}
\right.
\end{equation}
for any $\eta:\omega\rightarrow (-L,L)$. Here $\phi\in C^\infty(\mathbb R)$ is such that 
$\phi\equiv 0$ in a neighbourhood of $-L$ and $\phi\equiv 1$ in a neighbourhood of $0$. The other variables  $\mathbf{p}$, $s$ and $\bn$ are as defined above.
Due to the size of $L$, we find that $\bfPsi_\eta$ is a homomorphism such that $\bfPsi_\eta|_{\Omega\setminus S_L}$ is the identity. Furthermore, $\bm{\Psi}_\eta$ together with its inverse\footnote{It exists provided that we choose $\phi$ such that $|\phi'|<L/\alpha$ and $\|\eta\|_{L^\infty_\by}\leq\alpha<L$.} 
 $\bm{\Psi}_\eta^{-1} : \Omega_\eta \rightarrow\Omega$  possesses the following properties, see \cite{Br} for details. If for some $\alpha>0$, we assume that
\begin{align*}
\Vert\eta\Vert_{L^\infty_\by}
+
\Vert\zeta\Vert_{L^\infty_\by}
< \alpha <L 
\end{align*}
holds, then for any  $s>0$, $\varrho,p\in[1,\infty]$ and for any $\eta,\zeta \in B^{s}_{\varrho,p}(\omega)\cap W^{1,\infty}(\omega)$, we have that
\begin{align}
\label{210and212}
&\Vert \bm{\Psi}_\eta \Vert_{B^s_{\varrho,p}(\Omega\cup S_\alpha)}
+
\Vert \bm{\Psi}_\eta^{-1} \Vert_{B^s_{\varrho,p}(\Omega\cup S_\alpha)}
 \lesssim
1+ \Vert \eta \Vert_{B^s_{\varrho,p}(\omega)},
\\
\label{211and213}
&\Vert \bm{\Psi}_\eta - \bm{\Psi}_\zeta  \Vert_{B^s_{\varrho,p}(\Omega\cup S_\alpha)} 
+
\Vert \bm{\Psi}_\eta^{-1} - \bm{\Psi}_\zeta^{-1}  \Vert_{B^s_{\varrho,p}(\Omega\cup S_\alpha)} 
\lesssim
 \Vert \eta - \zeta \Vert_{B^s_{\varrho,p}(\omega)},
\end{align}
and
\begin{align}
\label{218}
&\Vert \partial_t\bm{\Psi}_\eta \Vert_{B^s_{\varrho,p}(\Omega\cup S_\alpha)}
\lesssim
 \Vert \partial_t\eta \Vert_{B^{s}_{ \varrho,p}(\omega)},
\qquad
\eta
\in W^{1,1}(I;B^{s}_{\varrho,p}(\omega)),
\end{align}
holds uniformly in time with the hidden constants depending only on the reference geometry, on $L-\alpha$ and $R$. One easily computes the determinant of $\nabla\bfPsi_\eta$ seeing that
\begin{align}\label{eq:detPsi}
\mathrm{det}(\nabla\bfPsi_\eta)\approx 1+\eta.
\end{align}
Indeed, employing the mapping $\Lambda$ from \eqref{eq:Lambda} we have
\begin{align*}
\mathrm{det}(\nabla\bfPsi_\eta)&=\mathrm{det}(\nabla(\bfPsi_\eta\circ\Lambda^{-1})\circ\Lambda)\mathrm{det}(\nabla\Lambda),
\end{align*}
where $\nabla(\bfPsi_\eta\circ\Lambda^{-1})$ has column vectors (in the coordinates $(y,s)$ with $y=\bfvarphi^ {-1}(\bfp)$)
\begin{align*}
\partial_1\bfvarphi+(s+\eta)\partial_1\bfn+\partial_1\eta\bfn,\,\,\,\partial_2\bfvarphi+(s+\eta)\partial_2\bfn+\partial_2\eta\bfn,\,\,\,\bfn,
\end{align*}
and thus the determinant only depends on
\begin{align*}
\partial_1\bfvarphi+(s+\eta)\partial_1\bfn,\,\,\,\partial_2\bfvarphi+(s+\eta)\partial_2\bfn,\,\,\,\bfn.
\end{align*}
Since $\mathrm{det}(\nabla\Lambda^{-1})$ is a smooth and strictly positive function we arrive at \eqref{eq:detPsi}. Similarly to the above we obtain
\begin{align*}
\partial_{\bfn}(\bfPsi_\eta\circ\Lambda)=\bfn,
\end{align*} 
which implies
\begin{align}\label{eq:dnPsi}
\partial_{\bfn}\bfPsi_\eta=\nabla(\bfPsi_\eta\circ\Lambda)\circ\Lambda^{-1}\partial_{\bfn}\Lambda^{-1}=\nabla(\bfPsi_\eta\circ\Lambda)\circ\Lambda^{-1}\bfn=\bfn.
\end{align}

Finally, we recall a result from \cite[Prop. 3.3]{MuSc}
concerning solenoidal extension.
\begin{lemma}
\label{prop:musc}
For a given $\eta\in L^\infty(I;W^{1,2}( \omega ))$ with $\|\eta\|_{L^\infty_{t,\by}}<\alpha<L$, there are linear operators
\begin{align*}
\mathscr K_\eta:L^1( \omega )\rightarrow\mathbb R,\quad
\Testzeta:\{\xi\in L^1(I;W^{1,1}( \omega )):\,\mathscr K_\eta(\xi)=0\}\rightarrow L^1(I;W^{1,1}_{\Div}(\Omega\cup S_{\alpha} )),
\end{align*}
such that the tuple $(\Testzeta(\xi-\mathscr K_\eta(\xi)),\xi-\mathscr K_\eta(\xi))$ satisfies
\begin{align*}
\Testzeta(\xi-\mathscr K_\eta(\xi))&\in L^\infty(I;L^2(\Omega_\eta))\cap L^2(I;W^{1,2}_{\Div}(\Omega_\eta)),\\
\xi-\mathscr K_\eta(\xi)&\in L^\infty(I;W^{2,2}( \omega ))\cap  W^{1,\infty}(I;L^{2}( \omega )),\\
\mathrm{tr}_\eta (\Testzeta&(\xi-\mathscr K_\eta(\xi))=\xi-\mathscr K_\eta(\xi),\\
\Testzeta(\xi-\mathscr K_\eta&(\xi))(t,\bx)=0 \text{ for } (t,\bx)\in I \times (\Omega \setminus S_{\alpha})
\end{align*}
provided that we have $\xi\in L^\infty(I;W^{2,2}( \omega ))\cap  W^{1,\infty}(I;L^{2}(\omega))$.
In particular, we have the estimates
\begin{align*}
\|\Testzeta(\xi-\mathscr K_\eta(\xi))\|_{L^q(I;W^{1,p}(\Omega \cup S_{\alpha}  ))}
&\lesssim \|\xi\|_{L^q(I;W^{1,p}( \omega ))}+\|\xi\naby \eta\|_{L^q(I;L^{p}( \omega ))},\\
\|\partial_t\Testzeta(\xi-\mathscr K_\eta(\xi))\|_{L^q(I;L^{p}( \Omega\cup S_{\alpha}))}
&\lesssim \|\partial_t\xi\|_{L^q(I;L^{p}( \omega ))}+\|\xi\partial_t \eta\|_{L^q(I;L^{p}( \omega ))},
\\
\|\Testzeta(\xi-\mathscr K_\eta(\xi))\|_{L^q(I;W^{2,p}(\Omega \cup S_{\alpha}  ))}
&\lesssim \|\xi\|_{L^q(I;W^{2,p}( \omega ))}+\|\xi\naby^2 \eta\|_{L^q(I;L^{p}( \omega ))}
\\
&\quad+\|\abs{\naby \xi}\abs{\naby \eta}\|_{L^q(I;L^{p}( \omega ))}+\|\abs{\xi}\abs{\naby \eta}^2\|_{L^q(I;L^{p}( \omega ))}\\
&\quad+\|\xi\naby\eta\|_{L^q(I;L^{p}( \omega ))},
\\
\|\partial_t\Testzeta(\xi-\mathscr K_\eta(\xi))\|_{L^q(I;W^{1,p}( \Omega\cup S_{\alpha}))}
&\lesssim \|\partial_t\xi\|_{L^q(I;W^{1,p}( \omega ))}+\|\xi\partial_t \naby \eta\|_{L^q(I;L^{p}( \omega ))}
\\
&\quad+\|\abs{\partial_t \xi}\abs{\naby \eta}\|_{L^q(I;L^{p}( \omega ))}+\|\abs{\naby \xi}\abs{\partial_t \eta}\|_{L^q(I;L^{p}( \omega ))}
\\
&\quad+\|\xi\abs{\partial_t\eta}\abs{\naby \eta}\|_{L^q(I;L^{p}( \omega ))},
\end{align*}
for any $p\in (1,\infty),q\in(1,\infty]$.
\end{lemma}
For the corrector $\mathscr K_\eta(\xi)$ from above the following result is proved in \cite[Corollary 3.2]{MuSc}.
\begin{corollary}[Corrector]
\label{cor:cor}
For a given $\eta\in L^\infty(I;W^{1,2}( \omega ))$ with $\|\eta\|_{L^\infty_{t,\by}}<\alpha<L$, the corrector map  
\[
\mathscr K_\eta:L^1(\omega)\to \R,\, \quad\xi\mapsto \mathscr K_\eta(\xi)=(\tilde{\xi})_{{\lambda_{\eta}}}=\frac{\int_{ {S_\alpha}} \tilde{\xi}(\bfp(\bx)) {{\lambda_{\eta}}}({t},\bx)\dx}{\int_{ {S_\alpha}} {{\lambda_{\eta}}}({t},\bx)\dx},
\]
satisfies the following estimates for $q\in [1,\infty]$:
\begin{align*}
\|\mathscr K_\eta(\xi)\|_{L^{q}(0,T)}&\leq C \|\xi\|_{L^{q}(0,T;L^1(\omega))},
\\
\|\partial_t \mathscr K_\eta(\xi)\|_{L^{q}(0,T)}
&\leq C
\Big(\|\partial_t \xi\|_{L^{q}(0,T;L^1(\omega))}
+\| \xi \partial_t\eta\|_{L^{q}(0,T;L^1(\omega))}\Big),
\end{align*}
whenever the right hand side is finite. Here $C$ depends only on $L-\alpha$.

\end{corollary}

\subsection{The main results}\label{sec:main}
In this subsection, we introduce the notions of a strong solution to \eqref{1}--\eqref{interfaceCond} that are under consideration. 
The existence of a weak solution can be shown as in \cite{LeRu}.
Note that the weak formulation uses a pressure-free formulation (that is, with test-function satisfying additionally divergence-free constraint). If the solution possesses more regularity,
the pressure can be recovered by solving
\begin{align}
\label{eq:press}
\begin{aligned}
-\Delta \pi&= \diver \big(\partial_t \bu +(\bu\cdot\nabla)\bu-\Delta\bu -\bff\big)\qquad&\text{ in }I\times \Omega_\eta,
\\
\pi  &= \big(\vert \mathrm{det}(\naby \bm{\varphi}_{\eta})\vert^{-1}(\partial_t^2\eta +\Dely^2\eta-g)\bn^\intercal\bn_\eta\big)\circ\bm{\varphi}_{\eta}^{-1}
&
\\&\quad+\bn\circ\bm{\varphi}_{\eta}^{-1}\big(\nabx\bu+(\nabx\bu)^\intercal \big)\big)\bn_\eta\circ\bm{\varphi}_{\eta}^{-1} &\text{ on }I\times \partial\Omega_\eta.
\end{aligned}
\end{align}
 Setting
$\pi(t)=\pi_0(t)+c_\pi(t)$, where $\int_{\Omega_{\eta(t)}}\pi_0(t)\dx=0$ and $c_\pi=\pi-\pi_0$ is constant in space and 
testing the structure equation with 1 we obtain
\begin{equation}
\begin{aligned}\label{eq:pressure}
c_\pi(t)\int_{\omega}\bfn\cdot\bfn_\eta|\det(\naby\bfvarphi_\eta)|\dy
&=
\int_{\omega}\bfn\big(\nabla\bu+(\nabla\bu)^\intercal-\pi_0\mathbb I_{n\times n}\big)\circ\bfvarphi_\eta\bfn_\eta|\det(\naby\bfvarphi_\eta)|\dy
\\&\quad+\int_\omega\partial_t^2\eta\dy-\int_\omega g\dy.
\end{aligned}
\end{equation}
Since $\Omega_\eta$ is $C^1$ uniformly in time, the operator $\Delta$ has the usual regularity and uniqueness properties for $C^1$ domains. In particular, it allows for a unique solution in $L^2$, if the right hand side is in $W^{-2,2}$ and the boundary value in $W^{-\frac{1}{2},2}$ or for a unique solution in $W^{1,2}$, provided that its boundary value is in $W^{\frac{1}{2},2}$ and the right hand side is in $W^{-1,2}$. Moreover, in this particular case, the solution of \eqref{eq:press} satisfies 
\begin{align*}
-\nabla \pi=\partial_t \bu +(\bu\cdot\nabla)\bu-\Delta\bu-\bff,
\end{align*}
which implies that
\begin{align*}
\int_I\int_{\Omega_\eta}|\nabla\pi|^2\dx\dt&\lesssim \int_I\int_{\Omega_\eta}\big(|\partial_t\bv\vert^2+|(\bu\cdot\nabla)\bu)|^2+|\Delta\bu|^2+\vert\bff\vert^2\big)\dx\dt
\\&\lesssim
\int_I\|\partial_t\bu\|^2_{L^{2}(\Omega_\eta)}\dt
+
\bigg( \int_I\|\bu\|^6_{L^6(\Omega_\eta)}\dt\bigg)^{\frac{1}{3}}\bigg(\int_I\|\nabla\bu\|^3_{L^{3}(\Omega_\eta)}\dt\bigg)^{\frac{2}{3}}
\\&\quad
+\int_I\|\nabla^2\bu\|^2_{L^{2}(\Omega_\eta)}\dt
+\int_I\|\bff\|^2_{L^{2}(\Omega_\eta)}\dt,
\end{align*}
whenever the right hand side is finite, independent of the boundary value of $\pi$ in \eqref{eq:press}.
This is the case for a strong solution defined as follows.

\begin{definition}[Strong solution] \label{def:strongSolution}
We call the triple $(\eta,\bu,\pi)$ a strong solution to \eqref{1}--\eqref{interfaceCond} provided that $(\eta,\bu,\pi)$ verifies the equations \eqref{1}--\eqref{interfaceCond} almost everywhere and satisfies
$$
\eta \in W^{2,\infty}(0,T;L^2(\omega))\cap W^{1,\infty}(0, T;W^{2,2}(\omega))\cap L^\infty(0, T; W^{4,2}(\omega)),$$
$$ \bu \in W^{1,\infty} \big(I; L^{2}(\Omega_{\eta}) \big)\cap  L^\infty \big(I; W^{2,2}(\Omega_{\eta}) \big),\quad\pi\in  L^\infty \big(I; W^{1,2}(\Omega_{\eta}) \big).
$$
\end{definition}

For a strong solution $(\eta,\bu,\pi)$ the momentum equation holds in the strong sense, that is we have
 \begin{align*} 
 \partial_t \bu+(\bu\cdot\nabla)\bu&=\Delta\bu-\nabla \pi+\bff&
 \end{align*}
a.e. in $I\times\Omega_\eta$. The shell equation together with the regularity properties above yield $\eta\in L^2(I;W^{4,2}(\omega))$.  Hence the shell equation holds in the strong sense as well, that is, we have
 \begin{align*}
\ \partial_t^2\eta+\Dely^2\eta=g-\bn^\intercal\bm{\tau}\circ\bm{\varphi}_\eta\bn_\eta
\vert \mathrm{det}(\naby \bm{\varphi}_\eta)\vert
 \end{align*}
a.e. in $I\times\omega$. 
Note that for a strong solution, the Cauchy stress $\bftau=\nabla\bu+(\nabla\bu)^\intercal-\pi\mathbb I_{n\times n}$ possesses enough regularity to be evaluated at the moving boundary (this is due to the trace theorem and the uniform Lipschitz continuity of $\Omega_\eta$).
Our main result is as follows.
\begin{theorem}\label{thm:NS3}
Let $n=3$.
Suppose that the dataset $(\bff, g, \eta_*, \eta_0,  \bv_0)$
satisfies 
\begin{equation}
\begin{aligned}
\label{dataset'}
\bff\in L^2&\big(I; W^{1,2}(\R^3)\big)\cap W^{1,2}\big(I; L^2(\R^3)\big),\quad 
g \in L^2\big(I; W^{1,2}(\omega)\big)\cap W^{1,2}\big(I; L^2(\omega)\big),  
\\
&\eta_0 \in W^{4,2}(\omega) , \quad \eta_* \in W^{2,2}(\omega),\quad 
\bv_0\in W^{2,2}_{\mathrm{\Div}}(\Omega_{\eta_0}),
\end{aligned}
\end{equation}
which verifies $\bu_0\circ\bfvarphi_{\eta_0} =\eta_* \bfn$ on $\omega$.
Then there exists $T_{\tt max}>0$ such that there exists is a unique strong solution to \eqref{1}--\eqref{2} in the sense of Definition \ref{def:strongSolution} in $(0,T_{\tt max})$. 
\end{theorem}
\begin{theorem}\label{thm:NS2}
Let $n=2$.
Suppose that the dataset $(\bff, g, \eta_*, \eta_0,  \bv_0)$
satisfies \eqref{dataset'}.
Then there exists a unique strong solution to \eqref{1}--\eqref{2} in the sense of Definition \ref{def:strongSolution}. The interval of existence is of the form $I = (0, t)$, where $t < T$ only in case $\Omega_{\eta(s)}$ approaches a self-intersection when $s\rightarrow t$ 
\footnote{Self-intersection and degeneracy are excluded if $\|\eta\|_{L^\infty_{t,\by}}<L$, cf. \eqref{eq:boundary1}.}.
\end{theorem}

\subsection{Stokes equations in rough domains}
In this section, we present the necessary framework to parametrise
the boundary of a domain $\mathcal O\subset\R^n$ by local maps of a certain regularity. This yields, in particular, a rigorous definition of a  $B^{s}_{\varrho,p}$-boundary. We follow the presentation from \cite{Br} (see also \cite{Br2}).

We assume that $\partial\mathcal O$ can be covered by a finite
number of open sets $\mathcal U^1,\dots,\mathcal U^\ell$ for some $\ell\in\mathbb N$, such that
the following holds. For each $j\in\{1,\dots,\ell\}$ there is a reference point
$\by^j\in\R^n$ and a local coordinate system $\{\be^j_1,\dots,\be_n^j\}$ (which we assume
to be orthonormal and set $\mathcal Q_j=(\be_1^j|\dots |\be_n^j)\in\mathbb R^{n\times n}$), a function
$\varphi_j:\mathbb R^{n-1}\rightarrow\mathbb R$
and $r_j>0$
with the following properties:
\begin{enumerate}[label={\bf (A\arabic{*})}]
\item\label{A1} There is $h_j>0$ such that
$$\mathcal U^j=\{\bx=\mathcal Q_j\bz+\by^j\in\mathbb R^n:\,\bz=(\bz',z_n)\in\R^n,\,|\bz'|<r_j,\,
|z_n-\varphi_j(\bz')|<h_j\}.$$
\item\label{A2} For $\bx\in\mathcal U^j$ we have with $\bz=\mathcal Q_j^\intercal(\bx-\by^j)$
\begin{itemize}
\item $\bx\in\partial\mathcal O$ if and only if $z_n=\varphi_j(\bz')$;
\item $\bx\in\mathcal O$ if and only if $0<z_n-\varphi_j(\bz')<h_j$;
\item $\bx\notin\mathcal O$ if and only if $0>z_n-\varphi_j(\bz')>-h_j$.
\end{itemize}
\item\label{A3} We have that
$$\partial\mathcal O\subset \bigcup_{j=1}^\ell\mathcal U^j.$$
\end{enumerate}
In other words, for any $\bx_0\in\partial\mathcal O$ there is a neighbourhood $U$ of $\bx_0$ and a function $\phi:\mathbb R^{n-1}\rightarrow\mathbb R$ such that after translation and rotation\footnote{By translation via $\by^j$ and rotation via $\mathcal Q_j$ we can assume that $\bx_0=\bm{0}$ and that the outer normal at~$\bx_0$ is pointing in the negative $x_n$-direction.}
 \begin{equation*}\label{eq:3009}
 U \cap \mathcal O = U \cap G,\quad G = \set{(\bx',x_n)\in \R^n \,:\, \bx' \in \R^{n-1}, x_n > \phi(\bx')}.
 \end{equation*}
 The regularity of $\partial\Omega$ will be described by means of local coordinates as just described.
 \begin{definition}\label{def:besovboundary}
 Let ${\mathcal{O}}\subset\R^n$ be a bounded domain, $s>0$ and $1\leq \rho,q\leq\infty$. We say that $\partial{\mathcal{O}}$ belongs to the class $B^s_{\rho,q}$ if there is $\ell\in\mathbb N$ and functions $\varphi_1,\dots,\varphi_\ell\in B^s_{\rho,q}(\mathbb R^{n-1})$ satisfying \ref{A1}--\ref{A3}.
 \end{definition}
Clearly, a similar definition applies for a Lipschitz boundary (or a $C^{1,\alpha}$-boundary with $\alpha\in(0,1)$) by requiring that $\varphi_1,\dots,\varphi_\ell\in W^{1,\infty}(\mathbb R^{n-1})$ (or $\varphi_1,\dots,\varphi_\ell\in C^{1,\alpha}(\mathbb R^{n-1})$). We say that the local Lipschitz constant of $\partial{\mathcal{O}}$, denoted by $\mathrm{Lip}(\partial{\mathcal{O}})$, is (smaller or) equal to some number $L>0$ provided that the Lipschitz constants of $\varphi_1,\dots,\varphi_\ell$ are not exceeding $L$.

After these preparations let us consider the steady Stokes system
\begin{equation}\label{eq:Stokes}
\left\{\begin{aligned}
&\Delta \bu-\nabla\pi=-\bff,	\\
&\Div\bu=0,\\
&\bu|_{\partial{\mathcal{O}}}=\bu_{\partial},
\end{aligned}\right.
\end{equation}
in a domain ${\mathcal{O}}\subset\R^n$ with unit normal $\bfn$. The result given in the following theorem is a maximal regularity estimate for the solution of \eqref{eq:Stokes} in terms of the right-hand side. The boundary data under minimal assumption on the regularity of $\partial\mathcal O$ is obtained in \cite[Theorem 3.1]{Br}. 
\begin{theorem}\label{thm:stokessteady}
Let $p\in(1,\infty)$, $s\geq 1+\frac{1}{p}$ and 
\begin{align*}
\varrho\geq p\quad\text{if}\quad p(s-1)\geq n,\quad \varrho\geq \tfrac{n}{s-1}\quad\text{if}\quad p(s-1)< n,
\end{align*}
 such that 
$n\big(\frac{1}{p}-\frac{1}{2}\big)+1\leq  s$.
 Suppose that ${\mathcal{O}}$ is a $B^{\theta}_{\varrho,p}$-domain for some $\theta>s-1/p$ with locally small Lipschitz constant, $\bff\in W^{s-2,p}({\mathcal{O}})$ and $\bu_{\partial}\in W^{s-1/p,p}(\partial{\mathcal{O}})$ with $\int_{\partial{\mathcal{O}}}\bu_\partial\cdot\bfn\,\dd\mathcal H^{n-1}=0$. Then there is a unique solution to \eqref{eq:Stokes} and we have
\begin{align*}
\|\bu\|_{W^{s,p}({\mathcal{O}})}+\|\pi\|_{W^{s-1,p}({\mathcal{O}})}\lesssim\|\bff\|_{W^{s-2,p}({\mathcal{O}})}+\|\bu_{\partial}\|_{W^{s-1/p,p}(\partial{\mathcal{O}})}.
\end{align*}
\end{theorem}
\begin{remark}\label{rem:stokes}
As observed in \cite[Remark 2.10]{BMSS} Theorem \ref{thm:stokessteady} can be applied in $\Omega_{\eta(t)}$ (for fixed $t$ and uniformly in time) with $s=p=2$ under the assumption
$\eta\in L^\infty(I;W^{2,2}\cap C^{1}(\omega))$.
\end{remark}

\subsection{Universal Bogovskij}

Bogovskij operators are natural to be considered in star-shaped domains. As Lipschitz domains are unions of star-shaped domains, for some time Bogovskij operators are available on Lipschitz domains. Recently the concept of universal Bogovskij operators was introduced in \cite{KamSchSpe20}. Observe that the same Bogovskij operator actually can be used for a family of domains, as long as the Lipschitz constant is controlled. This allows to use a (locally) steady operator to correct the divergence in the time-changing domains.

More precisely in \cite[Corollary 3.4]{KamSchSpe20} the following statement was shown:
\begin{theorem} \label{thm:ndBog1}
	Let $\Sigma \subset \R^{3}$ be a three-dimensional Lipschitz-manifold, $M > \gamma > 0$, $C_L>0$, $b\in C_c^\infty(\Sigma \times [0,\gamma])$ with unit integral. Then there exists a linear, universal Bogovskij operator $\Bog: C_c^\infty(\Sigma \times [0,M]) \to C_c^\infty(\Sigma \times [0,M])$ such that for any $C_L$-Lipschitz function (i.e., with Lipschitz constant $C_L$) $\eta:\Sigma \to [\gamma,M]$ and $\Sigma_\eta := \{(\bx',x_n) \in \Sigma \times [0,M]: 0< x_n < \eta(\bx')\}$ the operator $\Bog$ maps $C_c^\infty(\Omega_\eta)$ to $C_c^\infty(\Omega_\eta)$ with $\Div\Bog f = f - b \int f \dx$. 
	In addition,
	\[\norm{\Bog(f)}_{W^{s+1,p}( \Sigma_\eta)} \leq C_B^{s,p}\norm{f}_{W^{s,p}( \Sigma_\eta)},\]
	for all $ 1<p<\infty $ and $ s \ge 0$
	with $C_B^{s,p}$ only depending on $s$, $p$, $\operatorname{diam}(\Sigma),C_L,\gamma $ and the Lipschitz properties of $\Sigma$. %
\end{theorem}
In order to make this operator admissible for our needs we introduce the following version. 
\begin{corollary} \label{thm:ndBog}
	There is a universal Bogovskij operator, such that for all $\eta:\omega\rightarrow (-L,L)$ with $\norm{\naby \eta}_{L^\infty(\omega)}\leq C_L$ and $b\in C_c^\infty(\Omega\setminus S_L)$ (where $\Omega_\eta$ is defined by \eqref{eq:2612}) with unit integral we have
	\[
	\Bog: C_c^\infty(\Omega_\eta) \to C_c^\infty(\Omega_\eta)\text{ with }\Div\Bog f = f - b \int f \dx.
	\]
	In addition, it holds
	\begin{align*}
\norm{\Bog(f)}_{W^{s+1,p}( \Omega_\eta)} &\leq C_B^{s,p}\norm{f}_{W^{s,p}( \Omega_\eta)},\\
\norm{\Bog(\Div \bff)}_{W^{s,p}( \Omega_\eta)} &\leq C_B^{s,p}\norm{\bff}_{W^{s,p}( \Omega_\eta)},
\end{align*}
	for all $ 1<p<\infty $ and $ s \ge 0$
	with $C_B^{s,p}$ only depending on $L,\varphi,C_L$.
\end{corollary}

\begin{proof}
	The proof is by now standard. One covers the domain $S_L$ with balls of finite overlap, such that on each ball all possible functions $\eta$ can be written as a graph. On these sets one may apply Theorem~\ref{thm:ndBog1}. Hence the partition of unity argument introduced in~\cite[Section 3.1]{SaaSch21} allows to construct the desired operator.
\end{proof} 

\begin{remark}[Time-derivative of Bogovskij operator]\label{time-bog}
	{\rm Carefully note that the same operator can be applied to a time-changing function with time-changing support. The operator then automatically has zero trace on the variable support of the function. In particular for $\sup_t\norm{\naby \eta(t)}_{L^\infty_\by}\leq C_L$, $\sup_t\norm{\eta(t)}_{L^\infty_\by}\leq L$, we find that
		$\partial_t \mathrm{Bog}(f\chi_{\Omega_\eta})=\mathrm{Bog} (\partial_tf\chi_{\Omega_\eta})$ and $\mathrm{Bog}(\partial_t f\chi_{\Omega_\eta})=0$ on $\partial\Omega_\eta$.}
\end{remark}

\subsection{Gr\"onwall lemma}
We need the following version of a local Gr\"{o}nwall's lemma based on \cite[Lemma 24]{DiRu2005}.
\begin{lemma}
\label{lem:gron}
Let $T>0$ and $p\geq1$. Assume that for $f\in L^p((0,T))$ and $h \in L^1((0, T))$ with $f,h\geq0$ a.a. we have
\[
f(t)\leq c_0+\int_0^th(s)\ds+c_1\int_0^tf^p(s)\ds\quad\text{for a.a. }t\in[0,T],
\]
for some non-negative constants $c_0,c_1$.
Then there exist some $\tilde{T}\in(0, T]$ that depends on $h, c_0, c_1$ and $p$, such that for a.a. $t\in [0,\tilde{T}]$ we have
\[
f(t)\leq 2f(0)+2\int_0^th(s)\ds.\]
Moreover, it holds
\[\limsup_{t\to 0}f(t)\leq c_0.
\]
\end{lemma}
\begin{proof}

At first we consider the solution to the integral equation (which is nothing but a simple ordinary differential equation with unique local solution),
\[
g(t)=g(0)+\int_0^t h(s)\ds+c_1\int_0^tg^p(s)\ds\text{ with }g(0)=c_0.
\]
We find that $g$ is a continuous solution to an ODE for short times. Moreover by \cite[Lemma 24]{DiRu2005} that there exists a $\tilde{T}>0$ depending on $h, c_0, c_1$ and $p$, such that
\[
g(t)\leq 2g(0)+2\int_0^th(s)\ds\quad\text{for a.a. }t\in[0,\tilde T].
\]
Now we find that by our assumption on $f$ that
\[
f(t)-g(t)\leq c_1\int_0^t(f^p-g^p)\ds\quad\text{for a.a. }t\in[0,\tilde T].
\] 
It follows that
\[
f(t)-g(t)\leq c_1\int_0^t\int_0^1(\theta \abs{f(s)}+(1-\theta)\abs{g(s)})^{p-1}\, \dd\theta(f(s)-g(s))\ds\quad\text{for a.a. }t\in[0,\tilde T].
\] 
This implies for
$
a(t):=(f(t)-g(t))_{+},
$
\[
a(t)\leq c_1\int_0^t\bigg(\int_0^1(\theta \abs{f(s)}+(1-\theta)\abs{g(s)})^{p-1}\, \dd\theta\bigg) a(s)\ds\quad\text{and}\quad a(0)=0.
\]
We obtain $a(t)\equiv0$ by the standard Gronwall lemma as $\int_0^1(\theta \abs{f}+(1-\theta)\abs{g})^{p-1}\, d\theta \in L^1(0,\tilde T)$.
Hence both claims follows as $f(t)\leq g(t)$ by the above and the continuity of $g$ at $0$.
\end{proof}

\section{3D Navier--Stokes equations}
\label{sec:loc}
Our goal in this section is to construct a local-in-time strong solution of \eqref{1}--\eqref{interfaceCond}, i.e. proving Theorem \ref{thm:NS3}. To do this, it is useful to transform this system onto the fixed reference domain.
Thus, for a solution $( \eta, \bu,  \pi )$  of \eqref{1}--\eqref{interfaceCond}, we set $\overline{\pi}=\pi\circ \bfPsi_\eta$ and 
$\overline{\bu}=\bu\circ \bfPsi_\eta$
and define
\begin{equation}\label{matrices}
\begin{aligned}
\mathbf{A}_\eta=J_\eta\big( \nabx \bfPsi_\eta^{-1}\circ \bfPsi_\eta \big)^\intercal\nabx \bfPsi_\eta^{-1}\circ \bfPsi_\eta,&\\
\mathbf{B}_\eta=J_\eta \left(\nabx \bfPsi_\eta^{-1}\circ \bfPsi_\eta\right)^\intercal,&\\
\mathbf{H}_\eta(\overline{\bu}, \overline{\pi})
=-\mathbf{A}_\eta \nabx \overline{\bu}
+
\mathbf{B}_\eta \overline{\pi},&
\\
\mathbf{h}_\eta(\overline{\bu})
=
-
J_\eta \nabx\overline{\bu}\cdot \partial_t \bfPsi_\eta^{-1}\circ \bfPsi_\eta -\mathbf{B}_\eta\nabx\overline{\bfv}~\overline{\bu}
+
J_\eta  \overline{\bff},
\end{aligned}
\end{equation}
where $J_\eta=\mathrm{det}(\nabla\bfPsi_\eta)$ and $\overline{\bff}=\bff \circ \bfPsi_\eta$. Exactly as in the two-dimensional visco-elastic case considered in  \cite[Lemma 4.2]{Br} we obtain the following result.
\begin{theorem}
\label{thm:transformedSystem}
Suppose that the dataset
$(\bff, g, \eta_0, \eta_*, \bu_0)$
satisfies \eqref{dataset'}.
Then $( \eta, \bu,  \pi )$ is a strong solution to \eqref{1}--\eqref{interfaceCond} in the sense of Definition \ref{def:strongSolution}, if and only if $( \eta, \overline{\bu},  \overline{\pi} )$ is a strong solution of
\begin{align}
\label{contEqAloneBar}
\mathbf{B}_{\eta}:\nabx \overline{\bu}= 0,
\\
\partial_t^2\eta  + \Dely^2\eta
=
g+\bn^\intercal \big[\mathbf{H}_\eta(\overline{\bu}, \overline{\pi})\big]\circ\bm{\varphi} \bn ,
\label{shellEqAloneBar}
\\
J_{\eta}\partial_t \overline{\bu}  -\divx(\mathbf{A}_{\eta}  \nabx\overline{\bu}) 
 +\divx(\mathbf{B}_{\eta}\overline{\pi}) 
 = 
\mathbf{h}_\eta(\overline{\bu}),
\label{momEqAloneBar}
\end{align}
with  $\overline{\bu}  \circ \bm{\varphi}  =(\partial_t\eta)\bn$ on $I\times \omega$.
\end{theorem}
The idea now is to replace the unknown geometry, i.e. $\eta$ in $J_\eta, \mathbf A_\eta,\mathbf B_\eta$, $\mathbf h_\eta$ and $\mathbf H_\eta$, by a given function $\zeta$ and consider the problem of finding
$( \eta, \overline{\bu},  \overline{\pi} )$ as a strong solution of
\begin{align}
\label{contEqAloneBar'}
\mathbf{B}_{\zeta}:\nabx \overline{\bu}= 0,
\\
\partial_t^2\eta  + \Dely^2\eta
=
g+\bn^\intercal \big[\mathbf{H}_\zeta(\overline{\bu}, \overline{\pi})\big]\circ\bm{\varphi} \bn ,
\label{shellEqAloneBar'}
\\
J_{\zeta}\partial_t \overline{\bu}  -\divx(\mathbf{A}_{\zeta}  \nabx\overline{\bu}) 
 +\divx(\mathbf{B}_{\zeta}\overline{\pi}) 
 = 
\mathbf{h}_\zeta(\overline{\bu}),
\label{momEqAloneBar'}
\end{align}
with  $\overline{\bu}  \circ \bm{\varphi}  =(\partial_t\eta)\bn$ on $I\times \omega$.
We introduce the following spaces where we eventually will seek a solution within. For  $0<s\leq 1/4$, we define
\begin{align}\label{reguzeta}
\begin{aligned}
X_{\tt shell}(0,T)&:=W^{2,\infty}(0,T;L^2(\omega))\cap W^{1,\infty}(0, T;W^{2,2}(\omega))\cap L^\infty(0, T; W^{4,2}(\omega))\\
&\quad\cap W^{2,2}(0,T;W^{1/2,2}(\omega))\cap W^{1,2}(0,T;W^{s+2,2}(\omega))\cap L^{2}(0,T;W^{9/2,2}(\omega)),
\end{aligned}
\end{align}
and
\begin{align*}
 X_{\tt fluid}(0,T):= W^{1,\infty}(0,T;L^2(\Omega))\cap L^\infty(0,T;W^{2,2}(\Omega))\times L^\infty(0,T;W^{1,2}(\Omega)).
\end{align*}
We equip the spaces $X_{\tt shell}(0,T)$ and $X_{\tt fluid}(0,T)$,  with their natural norms
$\|\cdot\|_{X_{\tt shell}(0,T)}$ and $\|\cdot\|_{X_{\tt fluid}(0,T)}$, respectively. 
In Section \ref{sec:3.2} we show a-priori estimates, which imply that
the mapping $\zeta\mapsto\eta $ maps a ball $B_R$ in $X_{\tt shell}$ to itself provided that we choose $R$ large enough and $T:=T_{\tt max}$ small enough and that  $(\overline{\bu},\overline{\pi})$ belongs to a ball $B_R$ in $X_{\tt fluid}$.
In Section~4 below, we then prove that the mapping $\zeta\mapsto\eta $ is a contraction in the topology of the energy space
\begin{align*}
W^{1,\infty}(0,T;L^2(\omega))\cap L^{\infty}(I;W^{2,2}(\omega)),
\end{align*}
see Section \ref{sec:3.3}.

In the following assume that $\zeta\in  X_{\tt shell}(0,T)$, with 
\[
\|\zeta\|_{X_{\tt shell}(0,T)}\leq R.
\]
We do not solve
 \eqref{contEqAloneBar'}--\eqref{momEqAloneBar'}  directly but add an additional regularisation terms, such that all calculations in the following are admissible. 
Thus we consider for $0<\varepsilon\ll1$, the function ${\zeta_\varepsilon} $
as the mollification of $\zeta$ in space and time with a radius $\varepsilon$
and solve
\begin{equation}\label{eq:dissstructureTrans}
	\left\{\begin{aligned}
		& \partial_t^2\eta-\varepsilon\partial_t\Dely\eta +\Dely^2\eta=g
		+
		\bn^\intercal \big[\mathbf{H}_{\zeta_\varepsilon}(\overline{\bu}, \overline{\pi})\big]\circ\bm{\varphi} \bn
		&\text{ for all }  (t,\by)\in I\times\omega
		,\\
		&J_{\zeta_\varepsilon}\partial_t \overline{\bu}  -\divx(\mathbf{A}_{\zeta_\varepsilon}  \nabx\overline{\bu}) 
		+\divx(\mathbf{B}_{\zeta_\varepsilon}\overline{\pi}) 
		= 
		\mathbf{h}_{\zeta_\varepsilon}(\overline{\bu}) &\text{ for all }(t,\bx)\in I\times\Omega,\\
		&\mathbf{B}_{\zeta_\varepsilon}:\nabx \overline{\bu}=0&\text{ for all }(t,\bx)\in I\times\Omega.
	\end{aligned}\right.
\end{equation}
with  $ \overline \bu   \circ \bm{\varphi} =(\partial_t\eta)\bn$ on $I\times \omega$.
The latter reads in the weak formulation
\begin{equation*}
\begin{aligned}
	&\int_I  \frac{\mathrm{d}}{\dt}\bigg(\int_\omega \partial_t \eta \, \phi  J_{\zeta_\varepsilon}\dy
	+
	\int_{\Omega}\overline\bu  \cdot {\bfphi}\dx
	\bigg)\dt 
	\\
	&=\int_I  \int_{\Omega}\big(  \partial_t J_{\zeta_\varepsilon} \,\overline\bu \cdot{\bfphi}+J_{\zeta_\varepsilon}\,\overline\bu\cdot \partial_t  {\bfphi} 
	-  
	\bfA_{\zeta_\varepsilon}\nabla \overline\bu:\nabla {\bfphi} +\bfh_{\zeta_\varepsilon}(\overline \bfv)\cdot{\bfphi} \big) \dx\dt
	\\
	&\quad+
	\int_I \int_\omega \big(\partial_t \eta\, \partial_t\phi-\epsilon\partial_t\nabla_\by\eta\cdot \nabla_\by\phi+
	g\, \phi-\Dely\eta\,\Dely \phi \big)\dy\dt,
\end{aligned}
\end{equation*}
for all  $(\phi, {\bfphi}) \in C^\infty(\overline{I}\times\omega) \times C^\infty(\overline{I} ;C^\infty_{\Div}(\R^3))$ with $\phi(T,\cdot)=0$, ${\bfphi}(T,\cdot)=0$ and $\bfphi\circ\bfvarphi =\phi {\bfn}$ on $I\times\omega$.
%
Moreover we work with a regularised set of initial conditions $(\eta_{\varepsilon,0},\eta_{\varepsilon,*},\bv_{\varepsilon,0})$ satisfying
\begin{equation}\label{initialconv}
\begin{aligned}
&\eta_{\varepsilon,0} \rightarrow \eta_0 \qquad\text{in }\quad {{W^{4,2}(\omega)}},\\
&\eta_{\varepsilon,*} \rightarrow \eta_* \qquad\text{in }\quad {{W^{2,2}(\omega),}}\\
&\overline \bv_{\varepsilon,0} \rightarrow \overline \bv_0 \qquad\text{in } \quad {{W^{2,2}(\Omega).}}
\end{aligned} 
\end{equation}
We construct a solution of \eqref{eq:dissstructureTrans} with smooth initial data via a refined Galerkin approximation that allows us to prove higher temporal regularity (depending on $\varepsilon$). This is the content of the next subsection. 

\subsection{Galerkin approximation}
Firstly by solving the eigenvalue problems of the
Laplace operator we construct a smooth orthogonal basis $(\hat \bfX_k)_{k\in\N}$
of $W^{1,2}_{0}(\Omega)$. More precisely, we collect all the eigenfunctions of the following Stokes problem:
\begin{equation*}
	\left\{\begin{aligned}
		&-\Delta\hat \bfX_k+\nabla\hat{\mathfrak q}_k=\lambda_k\hat \bfX_k \quad 
		&\text{in}\quad \Omega,\\
		&\Div \hat \bfX_k=0 \quad &\text{in}\quad \Omega,\\
		&\hat \bfX_k=0 \quad &\text{on}\quad \partial\Omega,
	\end{aligned}\right.
\end{equation*}
where $\lambda_k$ is the corresponding eigenvalue of the eigenfunction 
$\hat \bfX_k$ for every $k\in\mathbb{N}$. This is orthonormal in $W^{1,2}_{0}(\Omega)$ and orthogonal in $L^2(\Omega)$, respectively.

Next we construct $(\check \bfX_k)_{k\in\mathbb{N}}$ by extending $(\check X_k)_{k\in 
	\mathbb{N}}$, the basis of ${W}^{2,2}(\omega):=\{\,f\in W^{2,2}(\omega): \int_\omega f\, \dd\mathcal{H}^2=0\}$ 
by solving the Stokes problem
\begin{equation*}
	\left\{\begin{aligned}
		&-\Delta\check  \bfX_k+\nabla \check  {\mathfrak q}_k=0&\text{in}\quad \Omega,\\
		&\Div \check  \bfX_k=0 &\text{in}\quad \Omega,\\
		& \check   \bfX_k= \check X_k\bfn\quad &\text{on}\quad\partial\Omega,
	\end{aligned}
	\right.
\end{equation*}
These functions are smooth in space and linearly independent. 

Now we define $(X_k, \bfX_k)_{k\in\mathbb{N}}$ in an enumeration way as follows 
\begin{equation*}
	\left(X_k, \bfX_k\right):=
	\left\{\begin{aligned}
		&\left(\check X_{\frac{k+1}{2}}, \check \bfX_{\frac{k+1}{2}}\right)
		\quad &k\,\,\, \text{odd}\\
		&\left(0, \hat \bfX_{\frac{k}{2}}\right)\quad &k\,\,\,\text{even}
	\end{aligned}
	\right.\qquad\text{ for all }k\in\mathbb{N}.
\end{equation*}
The $\text{span} \left\{(X_k, \, \bfX_k)\,\left|\,k\in \mathbb{N}\right.
\right\}$
consists of all couples
$(X, \bfX)\in {W}^{2,2}(\omega) \times W^{1,2}(\Omega) $ with $\bfX\circ\bfvarphi=X\,\bfn$ for every $\by\in \omega$. 

We observe further that the space $\text{span} \{( J_{\zeta_\varepsilon}^{-1}\circ\bfvarphi X_k, \bfB_{\zeta_\varepsilon}
^{-\intercal}\bfX_k)| k\in\mathbb{N}\}$ consists of the test 
function space:
$$(\psi, \Psi)\in\left\{(\psi, \Psi)\in 
{W}^{2,2}(\omega)\times W^{1,2}(\Omega)|\,\, \Psi\circ\bfvarphi=\psi \,\bfn\right\}, $$ 
which has the same regularity as ${\zeta_\varepsilon}$, being smooth in time. Note that $\Div (\bfB_{\zeta_\varepsilon}^{\intercal}\Psi)=0$ in $\Omega$ provided that
$\Div(\Psi\circ\bfPsi_{\zeta_\varepsilon}^{-1})=0$ in $\Omega_{\zeta_\varepsilon}$.
On that space we introduce the corresponding projection
operators by
\begin{equation*}
	\begin{aligned}
		&\mathscr{P}_n^{\tt shell}\phi:=J_{\zeta_\varepsilon}^{-1}\circ\bfvarphi\sum_{k=1}^n P^k(\phi) X_k:=J_{\zeta_\varepsilon}^{-1}\circ\bfvarphi\sum_{k=1}^n\langle J_{\zeta_\varepsilon}\circ\bfvarphi\,\phi, X_k\rangle_{{W}^{2,2}(\omega)}X_k,
		\\ 
		&\mathscr{P}_n^{\tt fluid}\Phi:=\bfB_{\zeta_\varepsilon}^{-\intercal}\sum_{k=1}^nP^k(\Phi)\bfX_k
		=\bfB_{\zeta_\varepsilon}^{-\intercal}\Big(\sum_{k\leq n,k\text{ odd}}P^k(\phi)\bfX_k+\sum_{k\leq n,k\text{ even}}\langle\bfX_k,\bfB_{\zeta_\varepsilon}^{\intercal}\Phi\rangle_{W^{1,2}_{0}(\Omega)}\bfX_k\Big).
	\end{aligned}
\end{equation*}

We make the following ansatz, for every fixed $n\in\mathbb{N}$:
\begin{equation}\label{ansatz}
	\begin{aligned}
		\overline \bu_n(t,\bx):=\bfB_{\zeta_\varepsilon}^{-\intercal}\sum_{k=1}^n
	\alpha^k_n(t)\bfX_k(\bx) \quad\text{ for all }t\in(0, T),\\
		\eta_n(t,\by):=\int_0^t(J_{\zeta_\varepsilon}^{-1}\circ\bfvarphi)(\tau)\sum_{k=1}^n\alpha_n^k(\tau) X_k(\by)\dd\tau+\mathscr{P}_n^{\tt shell}\eta_0\quad\text{ for all }t\in(0, T).
	\end{aligned}
\end{equation}

From the construction \eqref{ansatz}, we see that $\overline \bu_n\circ\bfvarphi=\partial_t\eta_n\,\bn$ for $(t, \by)\in (0, T)\times \omega$.

Now we go back to the weak formulation and take the test function ${\bfphi}:=\bfB_{\zeta_\varepsilon}^{-\intercal}\bfX_k$ for $k=1,\dots,n$. We seek for $(\eta_n, \overline{\bv}_n)$ (or rather the $\alpha_ n^k$'s from \eqref{ansatz}) by solving:
\begin{align}
	&\int_\Omega\partial_t\overline\bu_n\cdot \mathbf{B}_{\zeta_\varepsilon}^{-\intercal}\bfX_k J_{\zeta_\varepsilon}\dd\bx+\int_\Omega\nabla\overline\bu_n\,\partial_t\bfPsi_{\zeta_\varepsilon}^{-1}\circ\bfPsi_{\zeta_\varepsilon} \, \cdot \bfB_{\zeta_\varepsilon}^{-\intercal}\bfX_k J_{\zeta_\varepsilon}\dd\bx\nonumber
	\\
	&+\frac{1}{2}\int_\omega\partial_t{\zeta_\varepsilon} \partial_t\eta_n\,X_k \dy+\frac{1}{2}\int_\Omega(\overline\bu_n\,\cdot\mathbf{B}_{\zeta_\varepsilon}\nabla)\overline\bu_n\cdot\mathbf{B}_{\zeta_\varepsilon}^{-\intercal}\bfX_k\dd\bx-\frac{1}{2}\int_\Omega(\overline\bu_n\cdot\bfB_{\zeta_\varepsilon}\nabla)(\bfB_{\zeta_\varepsilon}^{-\intercal}\bfX_k)\cdot\overline\bu_n\,\dd\bx\nonumber
	\\ &+\int_\Omega(\mathbf{A}_{\zeta_\varepsilon}\nabla)\overline\bu_n:\nabla(\mathbf{B}_{\zeta_\varepsilon}^{-\intercal}\bfX_k)\dd\bx
	-\int_\Omega \bff\circ\bfPsi_{\zeta_\varepsilon}\cdot \,\bfB_{\zeta_\varepsilon}^{-\intercal}\bfX_k J_{\zeta_\varepsilon}\dd\bx-\int_\omega g\,J_{\zeta_\varepsilon}^{-1}\circ\bfvarphi X_k\dy\nonumber
	\\
	&
	+\int_\omega J_{\zeta_\varepsilon}^{-1}\circ\bfvarphi X_k\,\partial_t^2\eta_n \dy+\int_\omega\Dely\eta_n\Dely(J_{\zeta_\varepsilon}^{-1}\circ\bfvarphi  X_k) \dy+\varepsilon\int_\omega\partial_t\nabla \eta_n\cdot \nabla(J_{\zeta_\varepsilon}^{-1}\circ\bfvarphi X_k)\dy=0,\label{weakreference}
\end{align}
where  $(X_k, \bfX_k)\in {W}^{2,2}(\omega )\times W^{1,2}_{\operatorname{div}}(\Omega)$ with $\bfX_k \circ\bm{\varphi}= X_k\,\bn$ on $\omega$.  Note that in the above formulation we decouple the geometry and the solution pair $(\eta_n, \overline\bu_n)$.

Using the ansatz \eqref{ansatz}, we substitute the formulas of $\partial_t\overline\bu_n$ and $\partial_t^2\eta_n$ in \eqref{weakreference} and obtain the first order nonlinear ODE for $\boldsymbol{\alpha}_n=(\alpha_n^k)_{k=1}^n$:
\begin{equation}\label{ODE}
\mathcal{A}(t)\boldsymbol{\alpha}'_n(t)=\mathcal{F}(\partial_t\nabla\zeta_\varepsilon, \partial_t\zeta_\varepsilon, \nabla^2\zeta_{\varepsilon}, \nabla\zeta_{\varepsilon}, \boldsymbol{\alpha}_n),
\end{equation}
where the coefficient matrix $$\mathcal{A}(t)=\int_\Omega\bfB_{\zeta_\varepsilon}^{-\intercal}\bfX_k\cdot \bfB_{\zeta_\varepsilon}^{-\intercal}\bfX_l\, J_{\zeta_{\varepsilon}}\dx +\int_\omega J_{\zeta_\varepsilon}\circ\bfvarphi\, X_k\cdot J_{\zeta_{\varepsilon}}\circ\bfvarphi\, X_l \dy $$
is symmetric and coercive. The smoothness of $\zeta_\varepsilon$, together with the ODE theory, ensures the existence and uniqueness of $\boldsymbol{\alpha}_n$, which is in fact smooth.

We can now multiply  \eqref{weakreference} by $\alpha_n^k$ and sum over $k$. This corresponds to the test by $(\partial_t\eta_n,\overline{\bu}_n)$ and we obtain the basic energy estimate:
\begin{align*}
&\sup_I
\int_\omega\big(\vert \partial_t\eta_n\vert^2
+
\vert  \Dely \eta_n\vert^2 \big)\dy
 +
 \sup_I
 \int_\Omega\vert  \overline{\bu}_n\vert^2\dx
+\varepsilon\int_I\int_\omega|\partial_t\naby\eta_n|^2\dy\dt+
 \int_I\int_\Omega\vert  \nabx\overline{\bu}_n\vert^2\dx\dt
\\
&\lesssim
\int_\omega\big(\vert \eta_*\vert^2
+
\vert \Dely \eta_0\vert^2 \big)\dy
 +
 \int_I\int_\omega|g|^2\dy\dt
  +
   \int_{\Omega_{\zeta_\varepsilon(0)}}\vert {\bu}_0\vert^2\dx
+
 \int_I\int_{\Omega_{\zeta_\varepsilon}}\vert {\bff}\vert^2 
 \dx\dt,
\end{align*}
where the hidden constant is independent of $n,T$ and $\varepsilon$. This guarantees already the existence of a limit object $(\eta_\varepsilon,\overline\bu_\varepsilon)$, $n\to \infty$,  which is a weak solution to \eqref{eq:dissstructureTrans}.
Now we derive a higher order estimate on the Galerkin level, where we allow $\varepsilon$-dependence thus benefiting from the regularity of the regularised geometry. In comparison, the $\varepsilon$-independent estimate will be derived in the next subsection \ref{sec:3.2}. This $\epsilon$-dependent a-priori estimate is unavoidable even though we use finite differences later. The reason is that without only with an epsilon dependent bound the difference quotients are bounded uniformly leading to an $\epsilon$ independent bound. 
We introduce the following higher order in time energy:
\begin{align*}
\mathcal E_n(t)&=\|\partial_t^2\eta_n(t)\|_{L^2_\by}^2+\|\partial_t\Dely\eta_n(t)\|_{L^2_\by}^2+\|\partial_t\overline\bu_n(t)\|_{L^2_\bx}^2+\|\nabla\overline\bu_n(t)\|_{L^2_\bx}^2,\\
H(t)&=1+\|\overline{\bff}(t)\|^2_{W^{1,2}_\bx}+\|\partial_t\overline{\bff}(t)\|^2_{L^2_\bx}+\| g(t)\|^2_{W^{1,2}_\by}+\|\partial_tg(t)\|^2_{L^2_\by}.
\end{align*}
and obtain the following result.

\begin{proposition}
\label{thm:accestn}
Let $\varepsilon>0$ be given.
There is $T_{\tt max}^\varepsilon>0$ depending on $\varepsilon$ such that the following holds.
Suppose that the dataset
$( \bff,g, \eta_0, \eta_*, \overline{\bu}_0)$
satisfies \eqref{dataset'}.
Further assume that $\zeta\in X_{\tt shell}(0,T_{\tt max}^\varepsilon)$ in \eqref{reguzeta}.
Then the solution $( \eta_n, \overline{\bu}_n)$ of \eqref{weakreference} satisfies 
\begin{align}
\mathcal E_n(t)+\int_0^t\|\partial_t\nabla\overline\bu_n\|^2_{L^2_\bx}\ds+\varepsilon\int_0^t\|\partial_t^2\eta_n\|_{W^{1,2}_\by}^2\ds\lesssim \mathcal E_n(0)+1+\int_0^t H(s)\ds+\int_0^t \mathcal E_n^{4}(s)\ds,\label{gronwallpreparen}
\end{align}
for all $t\in (0,T_{\tt max}^\varepsilon)$ and $c$ independent of $n$.
\end{proposition}
%
\begin{proof}
Based on the analysis around \eqref{ODE}, the system \eqref{weakreference} is smooth with respect to the time variable. Thereby we differentiate \eqref{weakreference} in time, multiply the resulting expression by ${\alpha_n^k}'(t)$ and sum with respect to $k$ (this corresponds to a test by $(\partial_t^2\eta_n,\partial_t\overline{\bu}_n)$ in the time-differentiated system). 
Note that we have the following identities
\begin{align*}
\sum_{k=1}^n{\alpha_n^k}'(t)\bfX_k&=\partial_t(\bfB_{\zeta_\varepsilon}^\intercal\overline\bu_n)=\bfB_{\zeta_\varepsilon}^\intercal\partial_t\overline\bu_n+\partial_t\bfB_{\zeta_{\varepsilon}}^\intercal\overline\bu_n,
\\
\sum_{k=1}^n{\alpha_n^k}'(t)X_k&=\partial_t(J_{\zeta_{\varepsilon}}\circ\bfvarphi\,\partial_t\eta_n)=\partial_t J_{\varepsilon}\circ\bfvarphi\,\partial_t\eta_n+J_{\varepsilon}\circ\bfvarphi\,\partial_t^2\eta_n,
\\
\sum_{k=1}^n{\alpha_n^k}'(t)\nabla(\bfB_{\zeta_{\varepsilon}}^{-\intercal}\bfX_k)&=\nabla(\bfB_{\zeta_\varepsilon}^{-\intercal}\partial_t(\bfB_{\zeta_\varepsilon}^\intercal\overline\bu_n)),
\\
\sum_{k=1}^n{\alpha_n^k}'(t)\nabla(\partial_t\bfB_{\zeta_\varepsilon}^{-\intercal}\bfX_k)&=\nabla(\partial_t\bfB_{\zeta_\varepsilon}^{-\intercal}\partial_t(\bfB_{\zeta_\varepsilon}^\intercal\overline\bu_n)),
\\
\sum_{k=1}^n{\alpha_n^k}'(t)\Delta(J_{\zeta_{\varepsilon}}^{-1}\circ\bfvarphi\, X_k)&=\Delta(J_{\zeta_\varepsilon}^{-1}\circ\bfvarphi\partial_t(J_{\varepsilon}\circ\bfvarphi\,\partial_t\eta_n)),
\\
\sum_{k=1}^n{\alpha_n^k}'(t)\Delta(\partial_t J_{\varepsilon}^{-1}\circ\bfvarphi\, X_k)&=\Delta(\partial_t J_{\varepsilon}^{-1}\circ\bfvarphi\, \partial_t(J_{\varepsilon}\circ\bfvarphi\,\partial_t\eta_n)),
\\
\sum_{k=1}^n{\alpha_n^k}'(t)\nabla(J_{\varepsilon}^{-1}\circ\bfvarphi\, X_k)&=\nabla(J_{\varepsilon}^{-1}\circ\bfvarphi\,\partial_t(J_{\varepsilon}\circ\bfvarphi\partial_t\eta_n)),
\\
\sum_{k=1}^n{\alpha_n^k}'(t)\nabla(\partial_tJ_{\varepsilon}^{-1}\circ\bfvarphi\, X_k)&=\nabla(\partial_t J_{\varepsilon}^{-1}\circ\bfvarphi\,\partial_t(J_{\varepsilon}\circ\bfvarphi\,\partial_t\eta_n)).
\end{align*}
Further note that 
\[
\|\nabla\overline\bu_n(t)\|_{L^2_\bx}^2\lesssim  \int_0^t\|\partial_t\nabla\overline\bu_n\|^2_{L^2_\bx}\ds + \|\nabla\overline\bu_n(0)\|_{L^2_\bx}^2.
\]
Hence we obtain after integrating over time that
\begin{align*}
	\mathcal E_n(t)+\int_0^t\|\partial_t\nabla\overline\bu_n\|^2_{L^2_\bx}\ds+\varepsilon\int_0^t\|\partial_t^2\eta_n\|_{W^{1,2}_\by}^2\ds=\mathcal E_n(0)+({\tt I})_n+\cdots+ ({\tt XIII})_n,
\end{align*}
where the terms $({\tt I})_n,\dots, ({\tt XIII})_n$ are defined and estimated in the following. 
Since $\bfB_{\zeta_\varepsilon}$, $J_{\zeta_\varepsilon}$ and $\bf\Psi_{\zeta_\varepsilon}$ are all smooth in time and space, it suffices to estimate the critical terms in the remaining part. 

We first have by integration by parts that
\begin{align*}
({\tt I})_n:&=\int_0^t\int_\Omega\partial_t^2\overline\bu_n\partial_t\bfB_{\zeta_\varepsilon}^{-\intercal}
\bfB_{\zeta_\varepsilon}^\intercal\overline\bu_n J_{\zeta_\varepsilon}\dx\dd s
\\
&=\int_\Omega\partial_t\overline\bu_n\partial_t\bfB_{\zeta_\varepsilon}^{-\intercal}\bfB_{\zeta_\varepsilon}^\intercal\overline\bu_n J_{\zeta_\varepsilon}\dx\big|_0^t-\int_0^t\int_\Omega\partial_t\overline\bu_n\partial_t(\partial_t\bfB_{\zeta_\varepsilon}^{-\intercal}\bfB_{\zeta_\varepsilon}^\intercal\overline\bu_n J_{\zeta_\varepsilon})\dx\dd s\\
&\leq \kappa\|\partial_t\overline\bu_n\|_{L^2_\bx}^2+c(\kappa) C_0(1+t)+\int_0^t\|\partial_t\overline\bu_n\|_{L^2_\bx}^2\dd s.
\end{align*}
Then we also see that
\begin{align*}
({\tt II})_n:&=\int_0^t\int_\Omega\partial_t\overline\bu_n\partial_t\bfB_{\zeta_\varepsilon}^{-\intercal}\partial_t(\bfB_{\zeta_\varepsilon}^\intercal\overline\bu_n) J_{\zeta_\varepsilon}\dx\dd s+\int_0^t\int_\Omega\partial_t\overline\bu_n\bfB_{\varepsilon}^{-\intercal}\partial_t(\bfB_{\zeta_\varepsilon}^\intercal\overline\bu_n)\partial_t J_{\zeta_\varepsilon}\dx\dd s\\
&\lesssim \int_0^t\|\partial_t\overline\bu_n\|_{L^2_\bx}(\|\partial_t\overline\bu_n\|_{L^2_\bx}+\|\overline\bu_n\|_{L^2_\bx})\dd s\\
&\leq \int_0^t\|\partial_t\overline \bu_n\|_{L^2_\bx}^2\dd s+C_0 t.
\end{align*}
\begin{align*}
({\tt III})_n:&=\int_0^t\int_\Omega\partial_t\nabla\overline\bu_n\partial_t{\bf\Psi}_{\zeta_\varepsilon}^{-1}\circ{\bf\Psi}_{\zeta_\varepsilon} \bfB_{\zeta_\varepsilon}^{-\intercal}\partial_t(\bfB_{\zeta_\varepsilon}^\intercal \overline\bu_n) J_{\zeta_\varepsilon}\dx\dd s\\
&\leq \kappa\int_0^t\|\partial_t\nabla\overline\bu_n\|_{L^2_\bx}^2\dd s+c(\kappa)\int_0^t\|\partial_t\overline\bu_n\|_{L^2_\bx}^2\dd s+C_0 t.
\end{align*}
\begin{align*}
({\tt IV})_n:&=\int_0^t\int_\Omega\nabla\overline\bu_n\partial_t\left(\partial_t{\bf\Psi}_{\zeta_\varepsilon}^{-1}\circ{\bf\Psi}_{\zeta_\varepsilon}\,\bfB_{\zeta_\varepsilon}^{-\intercal} J_{\zeta_\varepsilon}\right)\partial_t(\bfB_{\zeta_\varepsilon}^\intercal\overline\bu_n)\dx\dd s\\
&\lesssim \int_0^t\|\nabla\overline\bu_n\|_{L^2_\bx}^2\dd s+\int_0^t(\|\partial_t\overline\bu_n\|_{L^2_\bx}^2+\|\overline\bu_n\|_{L^2_\bx}^2)\dd s.
\end{align*}
\begin{align*}
({\tt V})_n:&=\frac{1}{2}\int_0^t\int_\omega\partial_t^2\zeta_{\varepsilon}\partial_t\eta_n\,\partial_t(J_{\zeta_{\varepsilon}}\circ\bfvarphi\partial_t\eta_n)\dy\dd s+\frac{1}{2}\int_0^t\int_\omega\partial_t\zeta_{\varepsilon}\partial_t^2\eta_n\,\partial_t(J_{\zeta_{\varepsilon}}\circ\bfvarphi\partial_t\eta_n)\dy\dd s\\
&\lesssim \int_0^t\|\partial_t^2\eta_n\|_{L^2_\by}^2\dd s+\int_0^t\|\partial_t\eta_n\|_{L^2_\by}^2\dd s.
\end{align*}
\begin{align*}
({\tt VI})_n:&=\frac{1}{2}\int_0^t\int_\Omega\partial_t\overline\bu_n\bfB_{\zeta_\varepsilon}\nabla\overline\bu_n\bfB_{\zeta_\varepsilon}^{-\intercal}\partial_t(\bfB_{\zeta_\varepsilon}^\intercal\overline\bu_n)\dx\dd s+\frac{1}{2}\int_0^t\int_\Omega\overline\bu_n\partial_t\bfB_{\zeta_\varepsilon}\nabla\overline\bu_n\bfB_{\zeta_\varepsilon}^{-\intercal}\partial_t(\bfB_{\zeta_\varepsilon}^\intercal\overline\bu_n)\dx\dd s\\
&\quad +\frac{1}{2}\int_0^t\int_\Omega\overline\bu_n\bfB_{\zeta_\varepsilon}\nabla\overline\bu_n\partial_t\bfB_{\zeta_\varepsilon}^{-\intercal}\partial_t(\bfB_{\zeta_\varepsilon}\overline\bu_n)\dx\dd s+\frac{1}{2}\int_0^t\int_\Omega\overline\bu_n\bfB_{\zeta_\varepsilon}\partial_t\nabla\overline\bu_n\bfB_{\zeta_\varepsilon}^{-\intercal}\partial_t(\bfB_{\zeta_\varepsilon}^\intercal\overline\bu_n)\dx\dd s\\
&\lesssim \int_0^t(\|\nabla\overline\bu_n\|_{L^2_\bx}+\|\partial_t\nabla\overline\bu_n\|_{L^2_\bx})(\|\overline\bu_n\|_{L^4_\bx}\|\partial_t\overline\bu_n\|_{L^4_\bx}+\|\overline\bu_n\|_{L^4_\bx}^2)\dd s+\int_0^t\|\nabla\overline\bu_n\|_{L^2_\bx}\|\partial_t\overline\bu_n\|_{L^4_\bx}^2\dd s\\
&\lesssim \kappa\int_0^t\|\partial_t\overline{\bu}_n\|_{W^{1,2}_\bx}^2\ds+c(\kappa) \int_0^t\Big(\|\overline{\bu}_n\|^{8}_{W^{1,2}_\bx}+\|\partial_t\overline{\bu}_n\|_{L^2_\bx}^{8}+1\Big)\ds.
\end{align*}
The integral from the contribution of
$$\frac{1}{2}\int_0^t\int_\Omega(\overline\bu_n\cdot\bfB_{\zeta_\varepsilon}\nabla)(\bfB_{\zeta_\varepsilon}^{-\intercal}\bfX_k)\cdot\overline\bu_n\,\dd\bx\ds,$$
denoted by $({\tt VII})_n$, can be treated similarly as $({\tt VI})_n$. 

We continue the estimate:
\begin{align*}
({\tt VIII})_n:&=\int_0^t\int_\Omega\partial_t\bfA_{\zeta_\varepsilon}\nabla\overline\bu_n\,\nabla(\bfB_{\zeta_\varepsilon}^{-\intercal}\partial_t(\bfB_{\zeta_\varepsilon}^\intercal\overline\bu_n))\dx\dd s+\int_0^t\int_\Omega\bfA_{\zeta_\varepsilon}\nabla\overline\bu_n\,\nabla(\partial_t\bfB_{\zeta_\varepsilon}^{-\intercal}\partial_t(\bfB_{\zeta_\varepsilon}^\intercal\overline\bu_n))\dx\dd s\\
&\quad +\int_0^t\int_\Omega\bfA_{\zeta_\varepsilon}\partial_t\nabla\overline\bu_n\left(\nabla(\bfB_{\zeta_\varepsilon}^{-\intercal}\partial_t(\bfB_{\zeta_\varepsilon}\overline\bu_n))-\partial_t\nabla\overline\bu_n\right)\dx\dd s\\
&\lesssim \int_0^t\int_\Omega (\partial_t\nabla\overline\bu_n+\nabla\overline\bu_n)(\overline\bu_n+\partial_t\overline\bu_n+\nabla\overline\bu_n)\dx\dd s\\
&\leq \kappa\int_0^t\|\partial_t\nabla\overline\bu_n\|_{L^2_\bx}^2\dd s+c(\kappa)\int_0^t(\|\nabla\overline\bu_n\|_{L^2_\bx}^2+\|\partial_t\overline\bu_n\|_{L^2_\bx}^2+\|\overline\bu_n\|_{L^2_\bx}^2)\dd s.
\end{align*}
\begin{align*}
({\tt IX})_n:&=\int_0^t\int_\Omega\partial_t(\bff\circ{\bf\Psi}_{\zeta_\varepsilon}\,\bfB_{\zeta_\varepsilon}^{-\intercal} J_{\zeta_\varepsilon})\partial_t(\bfB_{\zeta_\varepsilon}^\intercal\overline\bu_n)\dx\dd s\\
&\lesssim \int_0^t\int_\Omega(\partial_t\overline\bff+\partial_t\nabla\overline\bff)(\partial_t\overline\bu_n+\overline\bu_n)\dx\dd s\\
&\leq \int_0^tH(s)\dd s+\int_0^t(\|\partial_t\overline\bu_n\|_{L^2_\bx}^2+\|\overline\bu_n\|_{L^2_\bx}^2)\dd s.
\end{align*}
\begin{align*}
({\tt X})_n:&=\int_0^t\int_\omega\partial_t(g\, J_{\zeta_\varepsilon}^{-1}\circ\bfvarphi)\partial_t( J_{\zeta_\varepsilon}\circ\bfvarphi \partial_t\eta_n)\dt\dd s\\
&\lesssim \int_0^t\int_\omega(g+\partial_tg)(\partial_t\eta_n+\partial_t^2\eta_n)\dy\dd s\\
&\leq \int_0^t H(s)\dd s+\int_0^t(\|\partial_t\eta_n\|_{L^2_\by}^2+\|\partial_t^2\eta_n\|_{L^2_\by}^2)\dd s.
\end{align*}
Now we take again the integration by parts with respect to time and obtain that
\begin{align*}
({\tt XI})_n:&=\int_0^t\int_\omega\partial_t^3\eta_n J_{\zeta_\varepsilon}^{-1}\circ\bfvarphi\partial_t J_{\zeta_\varepsilon}\circ\bfvarphi\partial_t\eta_n\dy\dd s+\int_0^t\int_\omega\partial_t^2\eta_n\partial_t J_{\zeta_\varepsilon}^{-1}\circ\bfvarphi\partial_t(J_{\zeta_\varepsilon}\circ\bfvarphi\partial_t\eta_n)\dy\dd s\\
&=\int_\omega\partial_t^2\eta_n  J_{\zeta_\varepsilon}^{-1}\circ\bfvarphi\partial_t J_{\zeta_\varepsilon}\circ\bfvarphi\partial_t\eta_n\dy\big|_0^t
-\int_0^t\partial_t^2\eta_n\partial_t(J_{\zeta_\varepsilon}^{-1}\circ\bfvarphi\partial_t J_{\zeta_\varepsilon}\circ\bfvarphi\partial_t\eta_n)\dy\dd s\\
&\quad +\int_0^t\int_\omega\partial_t^2\eta_n\partial_t J_{\zeta_\varepsilon}^{-1}\circ\bfvarphi\partial_t(J_{\zeta_\varepsilon}\circ\bfvarphi\partial_t\eta_n)\dy\dd s\\
&\leq\kappa\|\partial_t^2\eta_n\|_{L^2_\by}^2+c(\kappa)\|\partial_t\eta_n\|_{L^2_\by}^2+\int_0^t(\|\partial_t^2\eta_n\|_{L^2_\by}^2+\|\partial_t\eta_n\|_{L^2_\by}^2)\dd s.
\end{align*}
We also see that
\begin{align*}
({\tt XII})_n:&=\int_0^t\int_\omega\partial_t\Delta\eta_n\Delta(J_{\zeta_\varepsilon}^{-1}\circ\bfvarphi\partial_t J_{\zeta_\varepsilon}\circ\bfvarphi\partial_t\eta_n)\dy\dd s\\
&\quad +\int_0^t\int_\omega\Delta\eta_n\Delta(\partial_t J_{\zeta_\varepsilon}^{-1}\circ\bfvarphi\partial_t(J_{\zeta_\varepsilon}\circ\bfvarphi\partial_t\eta_n))\dy\dd s\\
&=\int_0^t\int_\omega\partial_t\Delta\eta_n\Delta(J_{\zeta_\varepsilon}^{-1}\circ\bfvarphi\partial_t J_{\zeta_\varepsilon}\circ\bfvarphi\partial_t\eta_n)\dy\dd s+\int_0^t\Delta\eta_n\partial_t\Delta\eta_n\dy\dd s\\
&\quad +\int_\omega\Delta\eta_n\partial_t J_{\zeta_\varepsilon}^{-1}\circ\bfvarphi J_{\zeta_\varepsilon}\circ\bfvarphi\partial_t\eta_n\dy\big|_0^t
-\int_0^t\int_\omega\partial_t\Delta\eta_n\partial_t J_{\zeta_\varepsilon}^{-1}\circ\bfvarphi J_{\zeta_\varepsilon}\circ\bfvarphi\partial_t\eta_n\dy\dd s\\
&\quad -\int_0^t\int_\omega\Delta\eta_n\partial_t(\partial_t J_{\zeta_\varepsilon}^{-1}\circ\bfvarphi J_{\zeta_\varepsilon}\circ\bfvarphi)\partial_t\eta_n\dy\dd s\\
&\lesssim \|\Delta\eta_n\|_{L^2_\by}^2+\|\partial_t\eta_n\|_{L^2_\by}^2+\int_0^t(\|\partial_t\Delta\eta_n\|_{L^2_\by}^2+\|\Delta\eta_n\|_{L^2_\by}^2+\|\partial_t\eta_n\|_{L^2_\by}^2)\dd s.
\end{align*}

Finally by taking an integration by parts in space we deal with the integral with $\varepsilon$ as below:
\begin{align*}
({\tt XIII})_n:&=\varepsilon\int_0^t\int_\omega\partial_t^2\nabla\eta_n\,\nabla(J_{\zeta_\varepsilon}^{-1}\circ\bfvarphi\partial_t J_{\zeta_\varepsilon}\circ\bfvarphi\partial_t\eta_n)\dy\dd s\\
&\quad +\varepsilon\int_0^t\int_\omega\partial_t\nabla\eta_n\,\nabla(\partial_t J_{\zeta_\varepsilon}^{-1}\circ\bfvarphi\partial_t(J_{\zeta_\varepsilon}\circ\bfvarphi\partial_t\eta_n))\dy\dd s\\
&=-\varepsilon\int_0^t\int_\omega\partial_t^2\eta_n\Delta(J_{\zeta_\varepsilon}^{-1}\circ\bfvarphi\partial_t J_{\zeta_\varepsilon}\circ\bfvarphi\partial_t\eta_n)\dy\dd s\\
&\quad -\varepsilon\int_0^t\int_\omega\partial_t\Delta\eta_n\,\partial_t J_{\zeta_\varepsilon}^{-1}\circ\bfvarphi\partial_t(J_{\zeta_\varepsilon}\circ\bfvarphi\partial_t\eta_n)\dy\dd s\\
&\lesssim \int_0^t\int_\omega\partial_t^2\eta_n(\partial_t\eta_n+\partial_t\nabla\eta_n+\partial_t\Delta\eta_n)\dy\dd s+\int_0^t\int_\omega\partial_t\Delta\eta_n\partial_t\eta_n\dy\dd s\\
&\leq \int_0^t\|\partial_t^2\eta_n\|_{L^2_\by}^2+\|\partial_t\eta_n\|_{L^2_\by}^2+\|\partial_t\Delta\eta_n\|_{L^2_\by}^2\dd s.
\end{align*}
Collecting all estimates and choosing $\kappa$ small enough yields the claim. 
\end{proof}
As already mentioned, the appropriate estimate of the initial values is a very sensitive point in the proof of the higher order in time estimate for the coupled system. It seems that it can only be rigorously justified using a Galerkin approximation, see~\cite{ScSu}. 
On the Galerkin level we find that estimate \eqref{gronwallpreparen} implies using Lemma~\ref{lem:gron} (assuming that $T_{\tt max}^\varepsilon(\mathcal E_n(0)+1)\ll1$),
that
\begin{align}\label{eq:16.03}
\mathcal E_n(t)+\int_0^t\|\partial_t\nabla\overline\bu_n\|^2_{L^2_\bx}\ds+\varepsilon\int_0^t\|\partial_t^2\eta_n\|_{W^{1,2}_\by}^2\ds\lesssim\mathcal E_n(0)+1+\int_0^t H(s)\ds,
\end{align}
for all $t\in (0,T_{\tt max}^\varepsilon)$ uniformly in $n$
and thereby we derive further that
\begin{align}
\label{eq:cont}
\limsup_{t\to 0} \mathcal{E}_{n}(t)\lesssim \mathcal{E}_{n}(0)+1.
\end{align}

Now we employ the argument of \cite[Section 4.2 and Prop. 4.2]{ScSu} to find that
\begin{equation*}
	\|\partial_t\overline\bu_n(0)\|_{L^2_\bx}+\|\partial_t^2\eta_n(0)\|_{L^2_\by} \leq c\left(\|\overline\bu_0\|_{W^{2,2}_\bx}+\|\eta_0\|_{W^{4,2}_\by}+\|\eta_*\|_{W^{2,2}_\by}+\|\overline\bff_0\|_{L^2_\bx}+\|g_0\|_{L^2_\by}\right), 
\end{equation*}
by using the system \eqref{weakreference} at the initial time. 
Hence the right-hand side of \eqref{eq:16.03} and \eqref{eq:cont} are bounded uniformly in $n$.

\subsection{The self-mapping property}\label{sec:3.2}
Compared with Proposition \ref{thm:accestn}, we derive in this subsection the $\varepsilon$-independent estimate, which eventually gives us the Proposition below. 
Similarly, we set
\begin{align*}
\mathcal E_\eta(t)&=\|\partial_t^2\eta(t)\|_{L^2_\by}^2+\|\partial_t\Dely\eta(t)\|_{L^2_\by}^2+\|\partial_t\overline\bu(t)\|_{L^2_\bx}^2+\|\nabla\overline\bu(t)\|_{L^2_\bx}^2,\\
H(t)&=1+\|\overline{\bff}(t)\|^2_{W^{1,2}_\bx}+\|\partial_t\overline{\bff}(t)\|^2_{L^2_\bx}+\| g(t)\|^2_{W^{1,2}_\by}+\|\partial_tg(t)\|^2_{L^2_\by}.
\end{align*}
At the heart of this section lies the following estimate for the time-differentiated system. Eventually we conclude with a Gr\"{o}nwall argument and the elliptic estimate from Theorem \ref{thm:stokessteady}.

\begin{proposition}
\label{thm:transformedSystemLinear}
There is $T_{\tt max}>0$ such that the following holds.
Suppose that the dataset
$( \bff,g, \eta_0, \eta_*, \overline{\bu}_0)$
satisfies \eqref{dataset'}.
Further assume that $\zeta\in X_{\tt shell}(0,T_{\tt max})$ in \eqref{reguzeta}, such that $\|{\zeta_\varepsilon}\|_{X_{\tt shell}(0,T_{\tt max})}\leq R$, for all $0\leq \epsilon\ll 1$.
Then the local strong solution $( \eta, \overline{\bu})$ of \eqref{contEqAloneBar'}--\eqref{momEqAloneBar'} satisfies 
\begin{align*}
\mathcal E_\eta(t)+\int_0^t\|\partial_t\nabla\overline\bu\|^2_{L^2_\bx}\ds+\int_0^t\|\partial_t\eta\|_{W^{2+s,2}_\by}^2\ds\leq\,c_0 \mathcal E_\eta(0)+c_1\int_0^t H(s)\ds+c_2\int_0^t \mathcal E_\eta^{5}(s)\ds,
\end{align*}
for all $t\in (0,T_{\tt max})$.
Here $c_0$ is independent of the data, $c_1$ and $c_2$ may depend on $T_{\tt max}$, $R$ and the initial data.
\end{proposition}
Before we prove the proposition we need the following two technical lemmas. The first is a uniform pressure estimate. The second is to improve the spatial regularity for solid.
\begin{lemma}[Uniform pressure estimate]
\label{lem:pres}
Under the assumptions of Proposition~\ref{thm:transformedSystemLinear}, we have 
\begin{align}\label{eq:1905} 
\int_I\Big(\|\overline\bfv_\varepsilon\|_{W^{2,2}_\bx(\Omega)}^2+\|\overline\pi_\varepsilon\|_{W_\bx^{1,2}(\Omega)}^2\Big)\dt&\lesssim \int_I\big(\mathcal E_\eta^{5}+1\big)\dt+\int_I\|\bff\circ\bfPsi_{\zeta_\varepsilon}\|_{L^2_\bx}^2\dt,
\end{align}
with constant only depending on $R$ but independent of $\varepsilon$.
\end{lemma}
\begin{proof}
In order to control the pressure we write similarly to \eqref{eq:pressure}
\begin{align*}
\overline\pi_\varepsilon=\overline\pi_{\varepsilon,0}+c_{\overline\pi_\varepsilon},
\end{align*}
where $(\overline\pi_{\varepsilon,0})_\Omega:=\int_\Omega\overline \pi_{\varepsilon,0}\dx=0$ and $c_{\overline\pi_\varepsilon}$ is a function of time only. The latter satisfies
\begin{align*}
c_{\overline\pi_\varepsilon}(t)\int_{\omega}\bfn\mathbf{B}_{{\zeta_\varepsilon}}\circ\bfvarphi\bfn\dy=\int_{\omega}\bfn\big(\bfA_{{\zeta_\varepsilon}}\nabla\overline\bfv_\varepsilon-\mathbf{B}_{{\zeta_\varepsilon}}\overline\pi_{\varepsilon,0}\big)\circ\bfvarphi\bfn\dy+\int_\omega\partial_t^2\eta_\varepsilon\dy-\int_\omega g\dy,
\end{align*}
due to the structure equation. Noticing that $\bfB_{{\zeta_\varepsilon}}$ is uniformly elliptic we infer from Poincar\'e's inequality
\begin{align*}
\int_{I}\|\overline\pi_\varepsilon\|^2_{W^{1,2}_\bx}\dt&\lesssim \int_{I}\|\nabla\overline\pi_\varepsilon\|^2_{L^{2}_\bx}\dt+\int_{I}|c_{\overline\pi_\varepsilon}|^2\dt\\
&\lesssim \int_{I}\|\nabla\overline\pi_\varepsilon\|^2_{L^{2}_\bx}\dt+\int_{I}\int_\omega|\partial_t^2\eta_\varepsilon|^2\dy\dt+\int_{I}\int_\omega|g|^2\dy\dt\\
&\quad +\int_{I}\|\overline\pi_{\varepsilon,0}\|_{L^{2}(\partial\Omega)}^2\dt+\int_{I}\|\nabla\overline\bfv_\varepsilon\|_{L^{2}(\partial\Omega)}^2\dt.
\end{align*}
By the trace theorem and $(\overline\pi_{\varepsilon,0})_\Omega=0$
we obtain for $I=[0,T]$ that
\begin{align*}
\int_{I}\|\overline\pi_{\varepsilon,0}\|_{L^{2}(\partial\Omega)}^2\dt
&\lesssim\int_{I}\|\overline\pi_{\varepsilon,0}\|_{W^{1,2}_\bx}^{2}\dt\lesssim \int_{I}\|\nabla\overline\pi_{\varepsilon,0}\|_{L^{2}_\bx}^{2}\dt=\int_{I}\|\nabla\overline\pi_\varepsilon\|_{L^{2}_\bx}^{2}\dt.
\end{align*}
We can use estimates for the steady Stokes system (see Theorem \ref{thm:stokessteady} and Remark \ref{rem:stokes}) with $s=2$, $p=2$ and $n=3$ to get
\begin{align*}
&\int_I\int_\Omega|\nabla^2\overline\bfv_\varepsilon|^{2}\dxt+\int_I\int_\Omega|\nabla\overline\pi_\varepsilon|^{2}\dxt\\
&\lesssim \int_I\int_\Omega\big(|\bfB_{\zeta_\varepsilon}\nabla\overline\bu_\varepsilon\,\overline\bu_\varepsilon|^{2}+|J_{\zeta_\varepsilon}\partial_t\overline\bfv_\varepsilon|^{2}+|J_{\zeta_\varepsilon}\nabla\overline \bu_\varepsilon\partial_t\bfPsi_{\zeta_\varepsilon}^{-1}\circ\bfPsi_{\zeta_\varepsilon}|^2+|J_{\zeta_\varepsilon}\bff \circ \bfPsi_{\zeta_\varepsilon}|^{2}\big)\dxt\\
&\quad +\int_I\|\partial_t\eta_\varepsilon\|^{2}_{W^{\frac{3}{2}, 2}(\omega)}\dt.
\end{align*}
Except for the first term, all terms on the right-hand side can be controlled by $\int_I\mathcal E_\eta\dt$. As for the convective term we have
 \begin{align*}
& \int_I\int_\Omega|\bfB_{\zeta_\varepsilon}\nabla\overline\bu_\varepsilon\,\overline\bu_\varepsilon|^{2}\dxt
 \lesssim
  \int_{I }\|\overline\bu_\varepsilon\|^2_{L^4_\bx}\|\nabla\overline\bu_\varepsilon\|_{L^{4}_\bx}^2\dt\\
&\leq\,c\int_{I }\|\overline\bu_\varepsilon\|_{L^4_\bx}^2\|\nabx\overline\bu_\varepsilon\|^{\frac{1}{2}}_{L^{2}_\bx}\|\nabla\overline\bu_\varepsilon\|_{W^{1,2}_\bx}^{\frac{3}{2}}\dt\\
&\leq\,c(\kappa)\int_{I }\|\overline\bu_\varepsilon\|_{L^4_\bx}^{8}\|\nabx\overline\bu_\varepsilon\|^{2}_{L^{2}_\bx}\dt+\kappa\int_{I }\|\nabla\overline\bu_\varepsilon\|_{W^{1,2}_\bx}^2\dt\\
&\leq \,c(\kappa)\int_{I }\|\overline\bu_\varepsilon\|^{10}_{W^{1,2}_\bx}\dt+\kappa\int_{I }\|\overline\bu_\varepsilon\|_{W^{2,2}_\bx}^2\, dt.
\end{align*}
Choosing $\kappa$ small enough we obtain the result.
\end{proof}

Regarding spatial regularity we show it for the discrete difference quotient in time that will eventually lead to an a-priori estimate of $\nabla \partial_t \eta$. We introduce here difference quotients:
\[
\Delta_\tau(f)(t,x)=\frac{f(t+\tau,x)-f(t,x)}{\tau}\text{ and }\Delta_{h,i}^s(f)(t,x)=\frac{f(t,x+he_i)-f(t,x)}{h^s},
\]
where in the following the index $i\in \{1,2\}$ is omitted to not overload the notation.

\begin{lemma}[Improved spatial regularity]
Under the assumption of Proposition~\ref{thm:transformedSystemLinear} we find for $I=[0,T]$ and all $s\in [0,\frac14)$ that
\begin{align}
&\nonumber
\int_{I}\|\Delta_\tau\eta_\varepsilon\|_{W^{s+2,2}_\by}^2\dt+\varepsilon\sup_I\|\Delta_\tau\eta_\varepsilon\|_{W^{s+1,2}_\by}^2\\
&\nonumber\lesssim\int_{I^\tau}\|\Delta_\tau\overline{\bfv}_\varepsilon\|_{W^{1,2}_\bx}^2\dt+\int_{I^\tau}\big(1+\|\partial_t\eta_\varepsilon\|_{W^{2,2}_\by}^{4}+\|\overline{\bv}_\varepsilon\|_{W^{1,2}_\bx}^{4}+\|\pi_\varepsilon\|_{W^{1,2}_\bx}^2\big)\dt
\\
&\quad+\sup_I\Big(\|\Delta_\tau\overline\bu_\varepsilon\|_{L^{2}_\bx}^2 +\|\Delta_\tau\Dely\eta_\varepsilon\|_{L^2_\by}^2 +\|\partial_t\Delta_\tau\eta_\varepsilon\|_{L^2_\by}^2\Big)+\varepsilon\|\Delta_\tau\eta_\varepsilon(0)\|_{W^{1+s,2}_\by}^2\label{eq:1508}
\end{align}
with constant depending on $R$ but uniformly in $\varepsilon$.
\end{lemma} 
\begin{proof}
For that we take the finite difference of \eqref{eq:dissstructureTrans} in time and test by
\begin{align*}   
(\phi_\varepsilon,\overline{\bm{\varphi}}_\varepsilon)=\left(\Delta_{-h}^s\Delta_h^s \Delta_\tau\eta_\varepsilon-\mathscr K_{{\zeta_\varepsilon}}(\Delta_{-h}^s\Delta_h^s \Delta_\tau\eta_\varepsilon, \,\,\mathscr F_{\zeta_\varepsilon}^{\mathrm{div}}(\Delta_{-h}^s\Delta_h^s \Delta_\tau\eta_\varepsilon-\mathscr K_{{\zeta_\varepsilon}}(\Delta_{-h}^s\Delta_h^s \Delta_\tau\eta_\varepsilon))\circ\bfPsi_{\zeta_\varepsilon})\right).
\end{align*}
Here for $0<s\leq 1/4$,  $\Delta_h^sv(\by)=h^{-s}(v(\by+h\bm{e}_\alpha)-v(\by))$ is the fractional difference quotient for space variable in direction $\bm{e}_\alpha$ with $\alpha\in\{1,2\}$. $\mathscr F_{\zeta_\varepsilon}^{\mathrm{div}}$ is the solenoidal extension, please see for instance Proposition \ref{prop:musc}. It follows that
\begin{align}
\nonumber
&\int_I \int_{ \omega }|\Delta_h^s\Dely\Delta_\tau\eta_\varepsilon|^2\dy\dt+\varepsilon\sup_I\int_\omega|\Delta_h^s\naby\Delta_\tau\eta_\varepsilon|^2\dy\\
\nonumber&
=\varepsilon \int_\omega|\Delta_h^s\naby\Delta_\tau\eta_\varepsilon(0)|^2\dy-\int_I\int_\Omega J_{\zeta_\varepsilon}^\tau\Delta_\tau\overline{\bv}_\varepsilon\cdot\partial_t\overline{\bm{\varphi}}_\varepsilon\dx\dt\\
&\quad +\int_I\int_\Omega\Delta_\tau\bfA_{\zeta_\varepsilon}\nabla\overline{\bv}_\varepsilon:\nabla\overline{\bm{\varphi}}_\varepsilon \dx\dt+\int_I\int_\Omega\bfA_{\zeta_\varepsilon}^\tau\nabla\Delta_\tau\overline{\bv}_\varepsilon:\nabla\overline{\bm{\varphi}}_\varepsilon\dx\dt\nonumber\\
&\quad +\int_I\frac{\dd}{\dt}\left(\int_\Omega J_{\zeta_\varepsilon}^\tau\Delta_\tau\overline{\bv}\cdot\overline{\bm{\varphi}}_\varepsilon\dx+\int_\omega\partial_t\Delta_\tau\eta_\varepsilon\cdot\phi_\varepsilon\dy\right)\dt-\int_I\int_\omega(\partial_t\Delta_\tau\eta_\varepsilon \partial_t\phi_\varepsilon+\Delta_\tau g \,\phi_\varepsilon)\dy\dt\nonumber\\
&\quad -\int_I\int_\Omega\pi_\varepsilon^\tau\Div(\Delta_\tau\bfB_{\zeta_\varepsilon}^\intercal \overline{\bm{\varphi}}_\varepsilon)\dx\dt-\int_I\int_\Omega\Delta_\tau\bfh_{\zeta_\varepsilon}\cdot \overline{\bm{\varphi}}_\varepsilon\dx\dt 
\nonumber\\
&\nonumber\quad +\int_I\int_\Omega\left(-\Delta_\tau J_{\zeta_\varepsilon} \overline \bu_\varepsilon\cdot \partial_t\overline{\bm{\varphi}}_\varepsilon+\frac{\dd}{\dt}\left(\Delta_\tau J_{\zeta_\varepsilon}\overline \bu_\varepsilon\cdot \overline{\bm{\varphi}}_\varepsilon\right)-\Delta_\tau(\partial_t J_{\zeta_\varepsilon}\overline \bu_\varepsilon)\cdot \overline{\bm{\varphi}}_\varepsilon\right)\dx\dt\\
&=:(O)^\varepsilon+(I)^\varepsilon+...+(VII)^\varepsilon+(VIII)^\varepsilon.\label{mom:reg}
\end{align}
By Lemma \ref{prop:musc}, ${\zeta_\varepsilon}\in X_{\tt shell}(0,T)$, \eqref{210and212} and \eqref{218}
\begin{align*}
(I)^\varepsilon&\lesssim \int_I\|\Delta_\tau\overline\bu_\varepsilon\|_{L^{2}_\bx}(\| \partial_t\Delta_\tau\Delta_{-h}^s\Delta_h^s\eta_\varepsilon\|_{L^{2}_\by}+\|\Delta_\tau\Delta_{-h}^s\Delta_h^s\eta_\varepsilon\|_{L^2_\by}+\| \naby\Delta_\tau\Delta_{-h}^s\Delta_h^s\eta_\varepsilon\|_{L^{2}_\by})\dt\\
&\lesssim \int_I\|\Delta_\tau\overline\bu_\varepsilon\|_{L^{2}_\bx}(\| \partial_t\Delta_\tau\eta_\varepsilon\|_{W^{1/2,2}_\by}+\|\Delta_\tau\eta_\varepsilon\|_{W^{3/2,2}_\by})\dt\\
&\lesssim \int_I\|\Delta_\tau\overline\bu_\varepsilon\|_{W^{1,2}_\bx}^2\dt+\int_I\|\Delta_\tau\eta_\varepsilon\|_{W^{2,2}_\by}^2\dt,
\end{align*}
using also the trace theorem (together with $\Delta_\tau\overline\bu_\varepsilon\circ\bfvarphi=\Delta_\tau\partial_t\eta_\varepsilon\bn$ in $\omega$).
Note that we can express $\mathbf{A}_{\zeta}=a(x,\zeta,\nabla \zeta)$ hence we find that
\begin{align*}
\tau\abs{\Delta_\tau \mathbf{A}_{\zeta}}&=\abs{a(x,\zeta(t+\tau,x),\nabla\zeta(t+\tau,x))-a(x,\zeta(t,x),\nabla\zeta(t,x))}
\\
&\leq \abs{a(x,\zeta(t+\tau,x),\nabla\zeta(t,x))-a(x,\zeta(t,x),\nabla\zeta(t,x))}
\\
&\quad +
\abs{a(x,\zeta(t,x),\nabla\zeta(t+\tau,x))-a(x,\zeta(t,x),\nabla\zeta(t,x))},
\end{align*}
and so checking the definition of $\mathbf{A}_{\zeta_\epsilon}$ we find
\[
\abs{\Delta_\tau \mathbf{A}_{\zeta_\epsilon}}
\lesssim (1+\abs{\nabla \zeta_\epsilon})(\abs{ \Delta_\tau \nabla \zeta_\epsilon}+\abs{\Delta_\tau  \zeta_\epsilon}),
\]
where the constant depends on $\Omega$ and $\norm{\zeta_\epsilon}_\infty$. Similarly we can estimate the other terms.
Using again Lemma \ref{prop:musc}, \eqref{210and212}--\eqref{218} and ${\zeta_\varepsilon}\in X_{\tt shell}(0,T)$ we estimate further that
\begin{align*}
(II)^\varepsilon&\lesssim \int_I\|\Delta_\tau\naby{\zeta_\varepsilon}\|_{L^4_\by}\|\nabla\overline\bu_\varepsilon\|_{L^2_\bx}\|\nabla\overline\bfvarphi_\varepsilon\|_{L^4_\bx}\dx+\int_I\|\Delta_\tau{\zeta_\varepsilon}\|_{L^\infty_\by}\|\nabla\overline\bu_\varepsilon\|_{L^2_\bx}\|\nabla\overline\bfvarphi_\varepsilon\|_{L^2_\bx}\dt
\\
&\lesssim \int_I\|\Delta_\tau\naby{\zeta_\varepsilon}\|_{W^{1,2}_\by}\|\nabla\overline\bu_\varepsilon\|_{L^2_\bx}\|\Delta_\tau\Delta_{-h}^s\Delta_h^s\eta_\varepsilon\|_{W^{1,4}_\by}\dt+\int_I\|\nabla\overline\bu_\varepsilon\|_{L^2_\bx}\|\Delta_\tau\Delta_{-h}^s\Delta_h^s\eta_\varepsilon\|_{W^{1,2}_\by}\dt
\\
&\lesssim \int_I\|\nabla\overline\bu_\varepsilon\|_{L^2_\bx}\|\Delta_\tau\eta_\varepsilon\|_{W^{2,2}_\by}\dt\lesssim \int_I\Big(\|\nabla\overline\bu_\varepsilon\|_{L^2_\bx}^2+\|\Delta_\tau\Dely\eta_\varepsilon\|_{L^{2}_\by}^2\Big)\dt,\\
(III)_\varepsilon&\lesssim \int_I\big(1+\|\naby{\zeta_\varepsilon}\|_{L^\infty_\by}\big)\|\Delta_\tau\nabla\overline\bu_\varepsilon\|_{L^2_\bx}\|\Delta_\tau\Delta_{-h}^s\Delta_h^s\eta_\varepsilon\|_{W^{1,2}_\by}\dt\\
&\lesssim \int_I\|\Delta_\tau\nabla\overline\bu_\varepsilon\|_{L^2_\bx}\|\Delta_\tau\eta_\varepsilon\|_{W^{2,2}_\by}\dt\\
&\lesssim \int_I\Big(\|\Delta_\tau \nabx\overline\bu_\varepsilon\|_{L^{2}_\bx}^2+\|\Delta_\tau\Dely\eta_\varepsilon\|_{L^{2}_\by}^2\Big)\dt,
\end{align*}
where we used the fact that $\partial_t\nabla_\by\zeta_\varepsilon\in L^\infty_t(W^{1,2}_\by)$. 
We also have 
\begin{align*}
(IV)^\varepsilon&
\lesssim \sup_I\|\Delta_\tau\overline \bu_\varepsilon\|_{L^2_\bx}\|\overline{\bm{\varphi}}_{\varepsilon}\|_{L^2_\bx}+\sup_I\|\partial_t\Delta_\tau\eta_{\varepsilon}\|_{L^2_\by}\|\phi_\varepsilon\|_{L^2_\by}\\
&\lesssim \sup_I\Big(\|\Delta_\tau\overline{\bu}_\varepsilon\|_{L^2_\bx}+\|\partial_t\Delta_\tau\eta_\varepsilon\|_{L^2_\by}\Big)\|\Delta_\tau\eta_\varepsilon\|_{W^{2s,2}_\by}\\
&\lesssim\sup_I\Big(\|\Delta_\tau\overline{\bu}_\varepsilon\|_{L^2_\bx}^2+\|\partial_t\Delta_\tau\eta_\varepsilon\|_{L^2_\by}^2+\| \Delta_\tau\Dely\eta_\varepsilon\|_{L^{2}_\by}^2\Big),
\end{align*}
and
\begin{align*}
(V)^\varepsilon 
&\lesssim \int_I\|\partial_t\Delta_\tau \eta_\varepsilon\|_{L^2_\by}\left(\|\partial_t\Delta_\tau \eta_\varepsilon\|_{W^{2s,2}_\by}
+
\Vert
\partial_t\zeta_\varepsilon\Delta_\tau\Delta_{-h}^s\Delta_h^s\eta_\varepsilon\|_{L^1_\by}
\right)\dt+\int_I\|\Delta_\tau g\|_{L^2_\by}\| \Delta_\tau\eta_\varepsilon\|_{W^{2s,2}_\by}\dt\\ 
&\lesssim \int_I\Big(\|\partial_t\Delta_\tau\eta_\varepsilon\|_{W^{1/2,2}_\by}^2+\|\Delta_\tau g\|_{L^2_\by}^2+\|\Delta_\tau\eta_\varepsilon\|_{W^{2,2}_\by}^2\Big)\dt\\
&\lesssim \int_I\Big(\|\Delta_\tau\overline{\bu}_\varepsilon\|_{W^{1,2}_\bx}^2+\|\Delta_\tau g\|_{L^2_\by}^2+\|\Delta_\tau\Dely\eta_\varepsilon\|_{L^{2}_\by}^2\Big)\dt,
\end{align*}
again by the Trace theorem. 
For the pressure term we have
\begin{align*}
(VI)^\varepsilon&\lesssim \int_I\|\pi_\varepsilon\|_{L^6_\bx}\|\Delta_\tau\naby^2{\zeta_\varepsilon}\|_{L^2_\by}\|\overline{\bfvarphi}_\varepsilon\|_{L^6_\bx}\dt\\
&\quad + \int_I\|\pi\|_{L^4_\bx}\|\Delta_\tau\naby{\zeta_\varepsilon}\|_{L^4_\by}(1+\|\naby{\zeta_\varepsilon}\|_{L^\infty_\by})\|\nabla\overline\bfvarphi_\varepsilon\|_{L^2_\bx}\dt\\
&\lesssim \int_I\big(\|\pi_\varepsilon\|_{W^{1,2}_\bx}^2+\|\overline\bfvarphi_\varepsilon\|_{W^{1,2}_\bx}^2\big)\dt\lesssim \int_I\big(\|\pi_\varepsilon\|_{W^{1,2}_\bx}^2+\|\Delta_\tau\eta_\varepsilon\|_{W^{1+2s}_\by}^2\big)\dt\\
&\lesssim \int_I\big(\|\pi_\varepsilon\|_{W^{1,2}_\bx}^2+\|\Delta_\tau\Dely\eta_\varepsilon\|_{L^2_\by}^2\big)\dt.
\end{align*}
We can now decompose $(VII)^\varepsilon$ into the sum of
\begin{align*}
(VII)^\varepsilon_1&:= \int_I\int_\Omega
\Delta_\tau\Big(J_{\zeta_\varepsilon} \nabx\overline{\bu}_\varepsilon\cdot \partial_t \bfPsi_{\zeta_\varepsilon}^{-1}\circ \bfPsi_{\zeta_\varepsilon}\Big)\cdot\overline{\bm{\varphi}}_\varepsilon\dx\dt,
\\(VII)^\varepsilon_2&:= \int_I\int_\Omega\Delta_\tau\big(\mathbf{B}_{\zeta_\varepsilon}\nabx\overline{\bfv}_\varepsilon~\overline{\bu}_\varepsilon\big)\cdot\overline{\bm{\varphi}}_\varepsilon\dx\dt,\\
(VII)^\varepsilon_3&:=- \int_I\int_\Omega\Delta_\tau\big(
J_{\zeta_\varepsilon}  \bff\circ \bfPsi_{\zeta_\varepsilon}\big)\cdot\overline{\bm{\varphi}}_\varepsilon\dx\dt,
\end{align*}
while further decomposing $(VII)^\varepsilon_2$ into the sum of $(VII)^\varepsilon_{2,1}$, $(VII)^{\varepsilon}_{2,2}$ and $(VII)^{\varepsilon}_{2,3}$ (defining them in an obvious manner differentiating in time by product rule).
We estimate by Lemma \ref{prop:musc}, \eqref{210and212}, \eqref{218} and ${\zeta_\varepsilon}\in X_{\tt shell}(0,T)$,
\begin{align*}
(VII)^\varepsilon_1&\lesssim\int_I
\|\nabla\overline\bu_\varepsilon^\tau\|_{L^2_\bx}\|\partial_t\nabla\zeta_\epsilon\|_{L^4_\by}\|\overline{\bm{\varphi}}_\epsilon\|_{L^4_\bx}\dt+\int_I\|\nabla\overline\bu_\varepsilon^\tau\|_{L^2_\bx}\|\partial_t\Delta_\tau\zeta_\epsilon\|_{L^3_\by}\|\overline\bfvarphi_\varepsilon\|_{L^6_\bx}\dt\\
&\quad +\int_I(\|\Delta_\tau\nabla\overline \bu_\epsilon\|_{L^2_\bx}+\|\nabla\overline\bu_\epsilon\|_{L^2_\bx})\|\overline{\bm{\varphi}}_{\epsilon}\|_{L^2_\bx}\dt\\
&\lesssim\int_I\|\nabla\overline\bu_\varepsilon^\tau\|_{L^2_\bx}(1+\|\partial_t^2\zeta_\varepsilon\|_{L^3_\by})\|\overline{\bm{\varphi}}_\epsilon\|_{W^{1,2}_\bx}\dt+\int_I\|\Delta_\tau\nabla\overline\bu_\epsilon\|_{L^2_\bx}\|\overline{\bm{\varphi}}_\epsilon\|_{L^2_\bx}\dt\\
& \lesssim \int_I\|\nabla\overline\bu_\varepsilon^\tau\|_{L^{2}_\bx}(1+\|\partial_t^2\zeta_\varepsilon\|_{L^3_\by})(\|\Delta_\tau\Delta_{-h}^s\Delta_h^s\eta_\varepsilon\|_{W^{1,2}_\by}+\|\Delta_\tau\Delta_{-h}^s\Delta_h^s\eta_\varepsilon\, \nabla{\zeta_\varepsilon}\|_{L^2_\by})\dt\\
&\quad +\int_{I^\tau}\|\Delta_\tau\nabla\overline\bu_\varepsilon\|_{L^2_\bx}\|\Delta_{-h}^s\Delta_h^s\partial_t\eta_\varepsilon\|_{L^2_\bx}\dt\\
&\lesssim\int_{I}\|\nabla \overline\bu_\varepsilon^\tau\|_{L^2_\bx}^4\dt+\int_I\|\partial_t^2\zeta_\varepsilon\|_{L^3_\by}^4\dt+\int_I\|\Delta_\tau\eta_\epsilon\|_{W^{2,2}_\by}^2\dt+\int_I\|\Delta_\tau\overline\bu_\epsilon\|_{W^{1,2}_\bx}^2\dt,
\end{align*}
where we use the embedding 
\begin{equation}\label{1embed}
	\zeta_\varepsilon\in W^{2,\infty}(I; L^2(\omega))\cap L^\infty(I;W^{4,2}(\omega))\hookrightarrow W^{\frac{7}{3},2}(I; L^3(\omega))\hookrightarrow W^{2,4}(I; L^3(\omega)),
\end{equation}
Furthermore,
\begin{align*}
(VII)^\varepsilon_{2,1}&\lesssim \int_I\|\Delta_\tau\mathbf B_{\zeta_\varepsilon}\|_{L^{6}_\bx}\|\nabx\overline{\bfv}_\varepsilon\|_{L^2_\bx}\|\overline{\bu}_\varepsilon\|_{L^6_\bx}\|\overline{\bm{\varphi}}_\varepsilon\|_{L^6_\bx}\dt\\
&\lesssim \int_I\|\Delta_\tau\naby{\zeta_\varepsilon}\|_{W^{1,2}_\by}\|\nabla\overline\bu_\varepsilon\|_{L^2_\bx}^2(\|\Delta_\tau\Delta_{-h}^s\Delta_h^s\eta_\varepsilon\|_{W^{1,2}_\by}+\|\Delta_\tau\Delta_{-h}^s\Delta_h^s\eta_\varepsilon\, \naby{\zeta_\varepsilon}\|_{L^2_\by})\dt\\
&\lesssim \int_I\|\nabla\overline\bu_\varepsilon\|_{L^2_\bx}^2\|\Delta_\tau\eta_\varepsilon\|_{W^{2,2}_\by}\dt\\
&\leq\,\int_I\big(\|\Delta_\tau\Dely\eta_\varepsilon\|_{L^2_\by}^{2}+\|\nabx\overline{\bv}_\varepsilon\|_{L^{2}_\bx}^{4}\big)\dt,
\end{align*}
\begin{align*}
(VII)^\varepsilon_{2,2}&\lesssim \int_I\|\mathbf B_{\zeta_\varepsilon}^\tau\|_{L^{\infty}_\bx}\|\Delta_\tau\nabx\overline{\bfv}_\varepsilon\|_{L^2_\bx}\|\overline{\bu}_\varepsilon^\tau\|_{L^4_\bx}\|\overline{\bm{\varphi}}_\varepsilon\|_{L^4_\bx}\dt \\
&\lesssim \int_I(1+\|\naby{\zeta_\varepsilon}^\tau\|_{L^{\infty}_\by})\|\Delta_\tau\overline{\bu}_\varepsilon\|_{W^{1,2}_\bx}\|\overline{\bu}_\varepsilon^\tau\|_{W^{1,2}_\bx}\|\Delta_\tau\eta_\varepsilon\|_{W^{2,2}_\by}\dt\\
&\lesssim\,\int_I\|\Delta_\tau\overline{\bfv}_\varepsilon\|_{W^{1,2}_\bx}^2\dt+\int_{I }\big(\|\Delta_\tau\eta_\varepsilon\|_{W^{2,2}_\by}^{4}+\|\overline{\bv}_\varepsilon\|_{W^{1,2}_\bx}^{4}\big)\dt,
\end{align*}
\begin{align*}
(VII)^\varepsilon_{2,3}&\lesssim \int_I\|\mathbf B_{\zeta_\varepsilon}^\tau\|_{L^{\infty}_\bx}\|\nabx\overline{\bfv}_\varepsilon\|_{L^2_\bx}\|\Delta_\tau\overline{\bu}_\varepsilon\|_{L^4_\bx}\|\overline{\bfvarphi}_\varepsilon\|_{L^4_\bx}\dt\\
&\lesssim \int_I(1+\|\naby{\zeta_\varepsilon}^\tau\|_{L^{\infty}_\by})\|\overline{\bu}_\varepsilon\|_{W^{1,2}_\bx}\|\Delta_\tau\overline{\bu}_\varepsilon\|_{W^{1,2}_\bx}\|\Delta_\tau\eta_\varepsilon\|_{W^{2,2}_\by}\dt\\
&\lesssim\,\int_I\|\Delta_\tau\overline{\bfv}_\varepsilon\|_{W^{1,2}_\bx}^2\dt+\int_{I }\big(\|\Delta_\tau\eta_\varepsilon\|_{W^{2,2}_\by}^{4}+\|\overline{\bv}_\varepsilon\|_{W^{1,2}_\bx}^{4}\big)\dt.
\end{align*}
Finally we see that
\begin{align*}
	(VIII)^\varepsilon&\lesssim \int_I\|\overline\bu_\varepsilon\|_{L^{2}_\bx}\big(\| \partial_t\Delta_\tau\Delta_{-h}^s\Delta_h^s\eta_\varepsilon\|_{L^{2}_\by}+\|\Delta_\tau\Delta_{-h}^s\Delta_h^s\eta\, \partial_t{\zeta_\varepsilon}\|_{L^2_\by})\dt\\&\quad +\sup_I\|\overline\bu_\varepsilon\|_{L^{2}_\bx}\|\Delta_{-h}^s\Delta_h^s\Delta_\tau\eta\|_{L^2_\bx}\\
	&\quad +\int_I\Big(\|\partial_t\Delta_\tau\zeta_\varepsilon\|_{L^3_\by}\|\overline\bu_\varepsilon\|_{L^6_\bx}+\|\partial_t\zeta_\varepsilon^\tau\|_{L^\infty_\by}\|\Delta_\tau\overline\bu_\varepsilon\|_{L^2_\bx}\Big)\|\Delta_\tau\eta_\varepsilon\|_{W^{2s,2}_\by}\dt\\
	&\lesssim 1+ \int_I\|\Delta_\tau\overline\bu_\varepsilon\|_{W^{1,2}_\bx}^2\dt+\int_I\|\Delta_\tau\eta_\varepsilon\|_{W^{2,2}_\by}^2\dt+\sup_I\|\Delta_\tau\Dely\eta_\varepsilon\|_{L^2_\by}^2,
\end{align*}
As all estimates are uniformly with respect to $h$ we have finished the proof.
\end{proof}
Now we are in a position to prove Proposition \ref{thm:transformedSystemLinear}. 

\begin{proof}[Proof of Proposition~\ref{thm:transformedSystemLinear}] 
We start working with system \eqref{eq:dissstructureTrans}. Observe that $( \eta_\varepsilon, \overline {\bu}_\varepsilon)$, as solutions of \eqref{eq:dissstructureTrans}, has sufficient regularity such that one can use $(\partial_t \eta_\varepsilon, \overline {\bu}_\varepsilon)$ as a test-function. Hence in particular the following energy inequality is satisfied:
\begin{align}
\nonumber
&\sup_I
\int_\omega\big(\vert \partial_t\eta_\varepsilon\vert^2
+
\vert  \Dely \eta_\varepsilon\vert^2 \big)\dy
 +
 \sup_I
 \int_\Omega\vert  \overline{\bu}_\varepsilon\vert^2\dx
+\varepsilon\int_I\int_\omega|\partial_t\naby\eta_\varepsilon|^2\dy\dt+
 \int_I\int_\Omega\vert  \nabx\overline{\bu}_\varepsilon\vert^2\dx\dt
\\
&\lesssim
\int_\omega\big(\vert \eta_*\vert^2
+
\vert \Dely \eta_0\vert^2 \big)\dy
 +
 \int_I\int_\omega|g|^2\dy\dt
  +
   \int_\Omega\vert \overline{\bu}_0\vert^2\dx
+
 \int_I\int_\Omega\vert \overline{\bff}\vert^2 
 \dx\dt,
\label{3.81}
\end{align}
where we used \eqref{initialconv} and the hidden constant is independent of $T$ and $\varepsilon$. In particular we find that $\norm{\eta_\epsilon}_\infty\leq C_0$, where $C_0$ is independent of $\varepsilon$ but may depend on the initial data via \eqref{3.81}.

Now we want to differentiate the equation in time and test with
$\partial_t\overline\bu_\varepsilon+\bfG$ in the momentum equation (the corresponding test-function for the structure equation is then $\partial_t^2\eta_\varepsilon+\bn^\intercal\bfG|_{\partial\Omega}\circ\bfvarphi$), where $\bfG=\bfB_{\zeta_\varepsilon}^{-\intercal}\partial_t\bfB_{\zeta_\varepsilon}^\intercal\overline\bu_\varepsilon$.
However, since this procedure is not well-defined at the current stage, we instead replace time derivatives by temporal difference quotients with parameter $0<\tau\ll1$. Due to the discrete product rule, time shifted functions of the form $f^\tau:=f(\cdot+\tau)$ appear frequently.
Our pair of test-functions is
\begin{align*}
(\partial_t\Delta_\tau\eta_\varepsilon+\bn^\intercal\bfG|_{\partial\Omega}\circ\bfvarphi\,,\,\Delta_\tau\overline\bu_\varepsilon+\bfG),\qquad \bfG=\bfB_{\zeta_\varepsilon}^{-\intercal}\Delta_\tau\bfB_{\zeta_\varepsilon}^\intercal\overline\bu^\tau_\varepsilon.
\end{align*}

Thereby, 
we derive from \eqref{eq:dissstructureTrans} that
\begin{align}
&\sup_I
\int_\omega\big(\vert \partial_t\Delta_\tau\eta_\varepsilon\vert^2
+
\vert  \Dely \Delta_\tau\eta_\varepsilon\vert^2 \big)\dy
 +
 \sup_I
 \int_\Omega\vert  \Delta_\tau\overline{\bu}_\varepsilon\vert^2\dx\nonumber
 \\
&\quad 
+ \varepsilon\int_I\int_\omega\vert  \partial_t\Delta_\tau\naby\eta_\varepsilon\vert^2\dy\dt
+
 \int_I\int_\Omega\vert  \Delta_\tau\nabx\overline{\bu}_\varepsilon\vert^2\dx\dt\nonumber
\\
&\lesssim
\underbrace{\int_\omega\big(\vert \partial_t\Delta_\tau\eta_\varepsilon(0)\vert^2
+
\vert \Dely \Delta_\tau\eta_{\varepsilon}(0)\vert^2\big)\dy}_{1st}
 +\underbrace{
\int_\Omega\vert \Delta_\tau\overline{\bu}_\varepsilon(0)\vert^2\dx}_{2nd}\nonumber
\\
&\quad +
 \underbrace{\int_I\int_\omega\Delta_\tau g \,\partial_t\Delta_\tau\eta_\varepsilon\dy\dt}_{3rd}
-\underbrace{\varepsilon\int_I\int_\omega\partial_t\Delta_\tau\naby\eta_\varepsilon
\cdot\naby(\bn^\intercal\bfG\circ\bfvarphi)\dy\dt}_{4th}\nonumber
\\
&
\quad -\underbrace{\int_I\int_\omega\partial_t^2\Delta_\tau\eta_\varepsilon\,\bn^\intercal\bfG\circ\bfvarphi\dy\dt}_{5th}
-\underbrace{\int_I\int_\omega\Delta_\tau\Dely^2\eta_\varepsilon\, \bn^\intercal\bfG\circ\bfvarphi\dy\dt}_{6th}\nonumber
 \\
&\quad +
\underbrace{\int_I\int_\omega\Delta_\tau g\, \bn^\intercal\bfG\circ\bfvarphi\dy\dt}_{7th}
+
 \underbrace{\int_I\int_\Omega\Big(\partial_t J^\tau_{\zeta_\varepsilon}\frac{\Delta_\tau\overline{\bfv}_\varepsilon}{2}-\Delta_\tau J_{\zeta_\varepsilon}\partial_t\overline\bfv_\varepsilon \Big)\cdot\Delta_\tau\overline{\bv}_\varepsilon
 \dx\dt}_{8th}\nonumber
 \\
&
\quad  +
\underbrace{ \int_I\int_\Omega\Delta_\tau\mathbf{h}_{\zeta_\varepsilon}\cdot\Delta_\tau\overline{\bv}_\varepsilon
 \dx\dt}_{9th}
-
 \underbrace{\int_I\int_\Omega
 \Delta_\tau\bfA_{\zeta_\varepsilon}\nabla\overline{\bv}_\varepsilon:\Delta_\tau\nabla\overline\bu_\varepsilon
 \dx\dt }_{10th}\nonumber
 \\
 &\quad +\underbrace{\int_I\int_\Omega\pi_\varepsilon\Div(\Delta_\tau \bfB_{\zeta_\varepsilon}^\intercal(\Delta_\tau\overline{\bv}_\varepsilon+\bfG))\dx\dt}_{11th}
+\underbrace{\int_I\int_\Omega(\Delta_\tau\bfh_{\zeta_\varepsilon}-\Delta_\tau J_{\zeta_\varepsilon}\partial_t\overline{\bfv}_\varepsilon)\cdot \bfG \dx\dt}_{12th}\nonumber
 \\
&\quad -\underbrace{\int_I\int_\Omega J_{_{\zeta_\varepsilon}}^\tau\partial_t\Delta_\tau\overline{\bv}_\varepsilon\cdot \bfG\dx\dt}_{13th}
-\underbrace{\int_I\int_\Omega\Delta_\tau\bfA_{\zeta_\varepsilon}\nabla\overline{\bv}_\varepsilon:\nabla\bfG\dx\dt}_{14th}\nonumber
 \\
&\quad -\underbrace{\int_I\int_\Omega\bfA_{\zeta_\varepsilon}\nabla\Delta_\tau\overline{\bv}_\varepsilon:\nabla\bfG \dx\dt}_{15th}.
\label{est}
\end{align} 
We proceed to estimate the terms on the right in  \eqref{est}.
Notice that, different from the plate case in \cite{ScSu}, here $\bfG$ does not have zero boundary values leading to the extra terms of the form $\bfG\circ\bfvarphi$ in \eqref{est} which we are now going to estimate. It holds
\begin{align*}
\bfG\circ \bfvarphi&=(\bfB_{\zeta_\varepsilon}^{-\intercal}\Delta_\tau\bfB_{\zeta_\varepsilon}^\intercal)\circ\bfvarphi\,\partial_t\eta^\tau_\varepsilon\bfn
= (\Delta_\tau J_{\zeta_\varepsilon} J_{\zeta_\varepsilon}^{-1}+1)\, \partial_t\eta_\varepsilon^\tau\, \bn
 \approx \Delta_\tau J_{\zeta_\varepsilon} (1+{\zeta_\varepsilon}) \partial_t\eta^\tau_\varepsilon\bfn
\end{align*}
on $\omega$ using the fact that $\nabla\bfPsi_{\zeta_\varepsilon}\circ\bfvarphi\bfn=\bfn$
and \eqref{eq:detPsi}. Note that the 1st 2nd and 3rd integral on the right hand side of \eqref{est} can be treated directly by the data. Hence we obtain starting from the estimate of the 4th integral:
\begin{align*}\nonumber
\text{4th}&\lesssim\varepsilon^2\int_I\|\partial_t\Delta_\tau\naby\eta_\varepsilon\|_{L^2_\by}^2\dt+c(\varepsilon)\int_I\|\naby\bfG\|^2_{L^{2}_\by}\dt\nonumber\\
&\lesssim\varepsilon^2\int_{I^\tau}\|\partial_t\Delta_\tau\naby\eta_\varepsilon\|_{L^{2}_\by}^2\dt+c(\varepsilon)\int_I\|\Delta_\tau\naby{\zeta_\varepsilon}\|^2_{L^{2}_\bx}\|{\zeta_\varepsilon}\|^2_{L^{\infty}_\by}\|\partial_t\eta^\tau_\varepsilon\|^2_{L^{\infty}_\by}\dt\nonumber\\
&\quad +c(\varepsilon)\int_I\|\Delta_\tau{\zeta_\varepsilon}\|^2_{L^{2}_\by}\|\naby{\zeta_\varepsilon}\|^2_{L^{\infty}_\bx}\| \partial_t\eta^\tau_\varepsilon\|_{L^\infty_\by}\dt+c(\varepsilon)\int_I\|\Delta_\tau{\zeta_\varepsilon}\|^2_{L^{\infty}_\by}\|{\zeta_\varepsilon}\|^2_{L^{\infty}_\by}\|\partial_t\naby\eta^\tau_\varepsilon\|^2_{L^{2}_\by}\dt\nonumber\\
&\lesssim \varepsilon^2\int_{I^\tau}\|\partial_t\Delta_\tau\naby\eta_\varepsilon\|_{L^{2}_\by}^2\dt+c(\varepsilon)\int_{I^\tau}\|\partial_t\Dely\eta_\varepsilon\|^2_{L^{2}_\by}\dt,
\end{align*}
uniformly in $\bfn$ and $\varepsilon$ and where $I^\tau=(0,T+\tau)$.
We proceed with the 5th term
\begin{align*}
-\text{5th}&=\int_I\int_\omega \partial_t^2\Delta_\tau\eta_\varepsilon\bn^\intercal\bfG\circ\bfvarphi\dy\dt\\
&\approx -\int_I\int_\omega \partial_t^2\eta_\varepsilon^\tau\,\Delta_\tau\partial_t\eta_\varepsilon\Delta_\tau{\zeta_\varepsilon} (1+{\zeta_\varepsilon})\dy\dt-\int_I\int_\omega \partial_t\Delta_\tau\eta_\varepsilon\partial_t^2{\zeta_\varepsilon}\partial_t\eta^\tau_\varepsilon (1+{\zeta_\varepsilon})\dy\dt\\
&\quad -\int_I\int_\omega\partial_t\Delta_\tau\eta_\varepsilon \partial_t\eta^\tau_\varepsilon |\partial_t{\zeta_\varepsilon}|^2\dy\dt+\int_\omega \partial_t\Delta_\tau\eta_\varepsilon\partial_t{\zeta_\varepsilon}\partial_t\eta^\tau_\varepsilon(1+{\zeta_\varepsilon})\dy\bigg|_{0}^{T}
\\
&\lesssim \int_{I^\tau}\big(\|\partial_t^2\eta_\varepsilon\|_{L^2_\by}^2
+\|\partial_t^2\eta_\varepsilon\|_{L^2_\by}\|\partial_t\eta_\varepsilon\|_{L^\infty_\by}\big)\dt
+\sup_{I}\|\partial_t\Delta_\tau\eta_\varepsilon\|_{L^2_\by}\|\partial_t\eta_\varepsilon^\tau\|_{L^2_\by}\\
&\lesssim \int_{I^\tau}\big(\|\partial_t^2\eta_\varepsilon\|_{L^2_\by}^2
+\|\partial_t\Dely\eta_\varepsilon\|_{L^2_\by}^2\big)\dt+\delta\sup_{I}\|\partial_t\Delta_\tau\eta_\varepsilon\|_{L^2_\by}^2+c(\delta),
\end{align*}
using $\zeta_\epsilon\in X_{\tt shell}(0,T)$ and \eqref{3.81}. Similarly, it holds for the 6th term that
\begin{align*}
&\int_I\int_\omega \Delta_\tau\Dely^2\eta_\varepsilon\bn^\intercal\bfG\circ\bfvarphi\dy\dt
\\&\approx\int_I\int_\omega \Delta_\tau\Dely\eta_\varepsilon\partial_t\Dely\eta_\varepsilon^\tau\,\partial_t{\zeta_\varepsilon}(1+{\zeta_\varepsilon})\dy\dt
+\int_I\int_\omega \Delta_\tau\Dely\eta_\varepsilon\partial_t\Dely{\zeta_\varepsilon}
\partial_t\eta^\tau_\varepsilon(1+{\zeta_\varepsilon})\dy\dt\\
&\quad +2\int_I\int_\omega \Delta_\tau\Dely\eta_\varepsilon\partial_t\naby
{\zeta_\varepsilon}\cdot\partial_t\naby\eta^\tau_\varepsilon(1+{\zeta_\varepsilon})\dy\dt+\int_I\int_\omega\Delta_\tau\Dely\eta_\varepsilon\partial_t\eta_\varepsilon\partial_t{\zeta_\varepsilon}\Dely{\zeta_\varepsilon}\dy\dt\\
&\quad +2\int_I\int_\omega\Delta_\tau\Dely\eta_\varepsilon
\partial_t\naby\eta^\tau\partial_t{\zeta_\varepsilon}\cdot\naby{\zeta_\varepsilon}\dy\dt+2\int_I\int_\omega\Delta_\tau\Dely\eta_\varepsilon
\partial_t\eta^\tau_\varepsilon\partial_t\naby{\zeta_\varepsilon}\cdot\naby{\zeta_\varepsilon}\dy\dt\\
&\lesssim \int_I(\|\Delta_\tau\Dely\eta_\varepsilon\|_{L^2_\by}^2+\|\partial_t\eta^\tau_\varepsilon\|_{L^2_\by}^2)\dt+\int_I \|\Delta_\tau\Dely\eta_\varepsilon\|_{L^2_\by}\|\partial_t\eta^\tau_\varepsilon\|_{L^\infty_\by}\|\partial_t\Dely{\zeta_\varepsilon}\|_{L^2_\by}\dt\\
&\quad +\int_I \|\Delta_\tau\Dely\eta_\varepsilon\|_{L^2_\by}\|\partial_t\naby\eta^\tau_\varepsilon\|_{L^4_\by}\|\partial_t\naby{\zeta_\varepsilon}\|_{L^4_\by}\dt\\
&\lesssim \int_{I^\tau} \|\partial_t\Dely\eta_\varepsilon\|_{L^2_\by}^2\dt,
\end{align*}
and,  for the 7th term
\begin{align*}
\text{7th}=\int_I\int_\omega\Delta_\tau g\, \bn^\intercal\bfG\circ\bfvarphi\dy\dt
&\approx \int_I\int_\omega\Delta_\tau g \partial_t\eta^\tau_\varepsilon\partial_t{\zeta_\varepsilon}(1+{\zeta_\varepsilon})\dy\dt\lesssim\int_{I^\tau}\|\partial_tg\|_{L^2_\by}^2\dt+\int_{I^\tau}\|\partial_t\eta_\varepsilon\|_{L^2_\by}^2\dt.
\end{align*}
Now we look into the remaining terms in  \eqref{est}.
The 8th integral on the right-hand side 
can be estimated based on \eqref{eq:detPsi} and the assumption on ${\zeta_\varepsilon}$,
\begin{align*}
	\text{8th}=\int_I\int_\Omega\Big(\partial_t J^\tau_{\zeta_\varepsilon}\frac{\Delta_\tau\overline{\bfv}_\varepsilon}{2}-\Delta_\tau J_{\zeta_\varepsilon}\partial_t\overline\bfv_\varepsilon \Big)\cdot\Delta_\tau\overline{\bv}_\varepsilon
\dx\dt\lesssim \,\int_{I^\tau}\|\partial_t\overline{\bfv}_\varepsilon\|_{L^2_\bx}^2\dt.
\end{align*}
The 9th  integral on the right hand side of \eqref{est} decomposes into
\begin{align*}
({\tt II})^{\bfh}&:=- \int_I\int_\Omega\Delta_\tau\big(
J_{\zeta_\varepsilon} \nabx\overline{\bu}_\varepsilon\cdot \partial_t \bfPsi_{\zeta_\varepsilon}^{-1}\circ \bfPsi_{\zeta_\varepsilon}\big)\cdot\Delta_\tau\overline{\bv}_\varepsilon\dxt,
\\({\tt III })^{\bfh}&:=- \int_I\int_\Omega\Delta_\tau\big(\mathbf{B}_{\zeta_\varepsilon}\nabx\overline{\bfv}_\varepsilon~\overline{\bu}_\varepsilon\big)\cdot\Delta_\tau\overline{\bv}_\varepsilon\dxt,\\
({\tt IV})^{\bfh}&:= \int_I\int_\Omega\Delta_\tau\big(
J_{\zeta_\varepsilon}  \bff\circ \bfPsi_{\zeta_\varepsilon}\big)\cdot\Delta_\tau\overline{\bv}_\varepsilon\dxt.
\end{align*}
Clearly, the last term is of lower order and it holds
\begin{align*}
({\tt IV})^{\bfh}
&\lesssim \int_{I^\tau}\big(\|\partial_t\bff\circ \bfPsi_{\zeta_\varepsilon}\|^2_{L^2_\bx}+\|\bff\circ \bfPsi_{\zeta_\varepsilon}\|_{W^{1,2}_\bx}^2\big)\dt+\int_{I^\tau}\|\partial_t\overline{\bv}_\varepsilon\|_{L^2_\bx}^2\dt.
\end{align*}
We further have
\begin{align*}
({\tt III })^{\bfh}&=- \int_I\int_\Omega\Delta_\tau\mathbf{B}_{\zeta_\varepsilon}\nabx\overline{\bfv}_\varepsilon^\tau~\overline{\bu}^\tau_\varepsilon\cdot\Delta_\tau\overline{\bv}_\varepsilon\dxt- \int_I\int_\Omega\mathbf{B}_{\zeta_\varepsilon}\Delta_\tau\nabx\overline{\bfv}_\varepsilon~\overline{\bu}^\tau_\varepsilon\cdot\Delta_\tau\overline{\bv}_\varepsilon\dxt\\
&\quad - \int_I\int_\Omega\mathbf{B}_{\zeta_\varepsilon}\nabx\overline{\bfv}_\varepsilon~\Delta_\tau\overline{\bu}_\varepsilon\cdot\Delta_\tau\overline{\bv}_\varepsilon\dxt\\
&=:({\tt III })^{\bfh}_1+({\tt III })^{\bfh}_2+({\tt III })^{\bfh}_3.
\end{align*}
First of all we have
\begin{align*}
({\tt III })^{\bfh}_1&\lesssim\int_I\|\nabla\overline\bu_\varepsilon^\tau\|_{L^2_\bx}\|\overline\bu_\varepsilon^\tau\|_{L^6_\bx}\|\Delta_\tau\naby{\zeta_\varepsilon}\|_{L^6_\by}\|\Delta_\tau\overline\bu_\varepsilon\|_{L^6_\bx}\dt\\
&\lesssim\int_I\|\overline\bu^\tau_\varepsilon\|_{W^{1,2}_\bx}^2\|\Delta_\tau\naby{\zeta_\varepsilon}\|_{W^{1,2}_\by}\|\Delta_\tau\overline\bu_\varepsilon\|_{W^{1,2}_\bx}\dt\\
&\lesssim \delta\int_I\| \Delta_\tau\overline\bu_\varepsilon\|_{W^{1,2}_\bx}^2\dt+c(\delta)\int_I\|\overline\bu_\varepsilon^\tau\|^4_{W^{1,2}_\bx}\dt,
\end{align*}
where we used that $\partial_t{\zeta_\varepsilon} \in L^\infty(I; W^{2,2}(\omega))$.

We have by \eqref{210and212}, \eqref{218} and ${\zeta_\varepsilon}\in X_{\tt shell}(0,T)$ that for arbitrary  $\delta>0$,
\begin{align*}
({\tt III })^{\bfh}_2&\lesssim \int_I\|\mathbf B_{\zeta_\varepsilon}\|_{L^{\infty}_\bx}\|\Delta_\tau\nabx\overline{\bfv}_\varepsilon\|_{L^2_\bx}\|\overline{\bu}^\tau_\varepsilon\|_{L^6_\bx}\|\Delta_\tau\overline{\bv}_\varepsilon\|_{L^3_\bx}\dt\\
&\lesssim \int_I(1+\|\naby{\zeta_\varepsilon}\|_{L^{\infty}_\by})\|\Delta_\tau\overline{\bu}_\varepsilon\|_{W^{1,2}_\bx}\|\overline{\bu}^\tau_\varepsilon\|_{W^{1,2}_\bx}\|\Delta_\tau\overline{\bv}_\varepsilon\|_{L^2_\bx}^{\frac{1}{2}}\|\Delta_\tau\overline\bu_\varepsilon\|_{W^{1,2}_\bx}^{\frac{1}{2}}\dt\\
&\lesssim \int_I\|\Delta_\tau\overline{\bfv}_\varepsilon\|_{W^{1,2}_\bx}^{\frac{3}{2}}\|\Delta_\tau\overline{\bv}_\varepsilon\|_{L^2_\bx}^{\frac{1}{2}}\|\overline{\bv}^\tau_\varepsilon\|_{W^{1,2}_\bx}\dt\\
&\leq\,\delta \int_{I}\|\Delta_\tau\overline{\bfv}_\varepsilon\|_{W^{1,2}_\bx}^2\dt+c(\delta)\int_{I}\|\Delta_\tau\overline{\bv}_\varepsilon\|_{L^2_\bx}^{4}\dt+c(\delta)\int_{I^\tau}\|\overline{\bv}_\varepsilon\|_{W^{1,2}_\bx}^{8}\dt,
\end{align*}
as well as
\begin{align*}
({\tt III })^{\bfh}_3&\lesssim \int_I\|\mathbf B_{\zeta_\varepsilon}\|_{L^{\infty}_\bx}\|\nabx\overline{\bfv}_\varepsilon\|_{L^2_\bx}\|\Delta_\tau\overline{\bu}_\varepsilon\|_{L^4_\bx}^2\dt\\
&\lesssim \int_I(1+\|\naby{\zeta_\varepsilon}\|_{L^{\infty}_\by})\|\overline{\bu}_\varepsilon\|_{W^{1,2}_\bx}\|\Delta_\tau\overline{\bu}_\varepsilon\|_{W^{1,2}_\bx}^{\frac{3}{2}}\|\Delta_\tau\overline{\bv}_\varepsilon\|_{L^2_\bx}^{\frac{1}{2}}\dt\\
&\leq\,\delta \int_I\|\Delta_\tau\overline{\bfv}_\varepsilon\|_{W^{1,2}_\bx}^2\dt+c(\delta)\int_{I^\tau}\big(\|\partial_t\overline{\bv}_\varepsilon\|_{L^2_\bx}^{4}+\|\overline{\bv}_\varepsilon\|_{W^{1,2}_\bx}^{8}\big)\dt.
\end{align*}
We proceed similarly for $({\tt II})^{\bfh}$ obtaining
\begin{align*}
({\tt II})^{\bfh}=&- \int_I\int_\Omega\Delta_\tau
J_{\zeta_\varepsilon} \nabx\overline{\bu}_\varepsilon\cdot \partial_t \bfPsi_{\zeta_\varepsilon}^{-1}\circ \bfPsi_{\zeta_\varepsilon}\cdot\Delta_\tau\overline{\bv}_\varepsilon\dxt\\
&- \int_I\int_\Omega
J_{\zeta_\varepsilon}^\tau \Delta_\tau\nabx\overline{\bu}_\varepsilon\cdot (\partial_t \bfPsi_{\zeta_\varepsilon}^{-1}\circ \bfPsi_{\zeta_\varepsilon})^\tau\cdot\Delta_\tau\overline{\bv}_\varepsilon\dxt\\
&- \int_I\int_\Omega
J_{\zeta_\varepsilon}^\tau \nabx\overline{\bu}^\tau_\varepsilon\cdot \Delta_\tau(\partial_t \bfPsi_{\zeta_\varepsilon}^{-1}\circ \bfPsi_{\zeta_\varepsilon})\cdot\Delta_\tau\overline{\bv}_\varepsilon\dxt\\
=:&({\tt II})^{\bfh}_1+({\tt II})^{\bfh}_2+({\tt II})^{\bfh}_3.
\end{align*}
Now $({\tt II})^{\bfh}_1$ and $({\tt II})^{\bfh}_2$ can be estimated directly using \eqref{eq:detPsi} and \eqref{218}. We have
\begin{align*}
({\tt II})^{\bfh}_1+({\tt II})^{\bfh}_2&\lesssim 
\int_I\|\partial_t{\zeta_\varepsilon}\|_{L^\infty_\by}^2\|\nabla\overline\bu_\varepsilon\|_{L^2_\bx}\|\Delta_\tau\overline\bu_\varepsilon\|_{L^2_\bx}\dt+\int_I\|\Delta_\tau\nabla\overline\bu_\varepsilon\|_{L^2_\bx}\|\Delta_\tau\overline\bu\|_{L^2_\bx}\dt\\
&\leq\,\delta \int_I\|\Delta_\tau\overline{\bfv}_\varepsilon\|_{W^{1,2}_\bx}^2\dt+c(\delta)\int_{I^\tau}\big(\|\partial_t\overline{\bv}_\varepsilon\|_{L^2_\bx}^{2}+\|\overline{\bv}_\varepsilon\|_{W^{1,2}_\bx}^{2}\big)\dt.
\end{align*}
Finally, using again \eqref{1embed} we see that
\begin{align*}
({\tt II})^{\bfh}_3
&\lesssim \int_I\|\nabx\overline{\bu}^\tau_\varepsilon\|_{L^2_\bx}\|\partial_t\Delta_\tau{\zeta_\varepsilon}\|_{L^3_\by}\|\Delta_\tau\overline{\bv}_\varepsilon\|_{L^6_\bx}\dt+\int_I\|\nabx\overline{\bu}^\tau\|_{L^2_\bx}\|(\partial_t\naby{\zeta_\varepsilon})^\tau\|_{L^4_\bx}\|\partial_t{\zeta_\varepsilon}\|_{L^\infty_\by}\|\Delta_\tau\overline{\bv}_\varepsilon\|_{L^4_\bx}\dt\\
&\lesssim \int_I\|\overline{\bu}^\tau_\varepsilon\|_{W^{1,2}_\bx}\|\partial_t\Delta_\tau{\zeta_\varepsilon}\|_{L^3_\by}\|\Delta_\tau\overline{\bv}_\varepsilon\|_{W^{1,2}_\bx}\dt+\int_I\|\nabla\overline\bu^\tau_\varepsilon\|_{L^2_\bx}\|\Delta_\tau\overline\bu_\varepsilon\|_{W^{1,2}_\bx}\|\partial_t\nabla_\by\zeta_\varepsilon\|_{W^{1,2}_\by}\dt\\
&\leq\,\delta \int_I\|\Delta_\tau\overline{\bfv}_\varepsilon\|_{W^{1,2}_\bx}^2\dt+c(\delta)\int_{I^\tau}(1+\|\overline{\bv}_\varepsilon\|_{W^{1,2}_\bx}^{4})\dt+c(\delta)\int_I\|\partial_t^2{\zeta_\varepsilon}\|_{L^3_\by}^4\dt.
\end{align*}
Lastly, for $0<s\leq \frac{1}{4}$ using the embedding
\begin{equation}\label{ineq2}
	\zeta_\varepsilon\in W^{2,\infty}(I; L^2(\omega))\cap L^\infty(I;W^{4,2}(\omega))\hookrightarrow W^{\frac{3-s}{2},2}(I; W^{s+2,2}(\omega))\hookrightarrow W^{1,4}(I; W^{s+2,2}(\omega)),
\end{equation}
we have
\begin{align*}
\int_I\int_\Omega\Delta_\tau\bfA_{\zeta_\varepsilon}\nabla\overline\bu_\varepsilon:\Delta_\tau\nabla\overline\bu_\varepsilon\dxt&\leq \int_I\|\Delta_\tau\nabla_\by\zeta_\varepsilon\|_{L^\infty_\by}\|\nabla\overline \bu_\varepsilon\|_{L^2_\bx}\|\Delta_\tau\nabla\overline \bu_\varepsilon\|_{L^2_\bx}\dt\\
&\lesssim\delta\int_I\|\Delta_\tau\nabla\overline\bu_\varepsilon\|_{L^2_\bx}^2\dt+c(\delta)\int_{I^\tau}\|\partial_t{\zeta_\varepsilon}\|_{W^{2+s,2}_\by}^4\dt+\int_I\|\nabla\overline\bu_\varepsilon\|_{L^2_\bx}^4\dt
.
\end{align*}
Next we consider now the 11th term that is the pressure term:
\begin{align*}
\text{11th}=\int_I\int_\Omega&\overline\pi_\varepsilon^\tau\Div(\Delta_\tau \bfB_{\zeta_\varepsilon}^\intercal(\Delta_\tau\overline{\bv}_\varepsilon+\bfG))\dx\dt.
\end{align*}
We split the 11th term into two parts.  First
\begin{align*}
&\int_I\int_\Omega\overline \pi_\varepsilon^\tau\Div(\Delta_\tau \bfB_{\zeta_\varepsilon}^\intercal\Delta_\tau\overline{\bv}_\varepsilon)\dx\dt\\
&\lesssim \int_I\|\overline\pi_\varepsilon\|_{L^4_\bx}\|\Delta_\tau\naby^2{\zeta_\varepsilon}\|_{L^2_\bx}\|\Delta_\tau\overline\bv_\varepsilon \|_{L^4_\bx}\dt+\int_I\|\overline\pi_\varepsilon \|_{L^4_\bx}\|\Delta_\tau\naby{\zeta_\varepsilon}\|_{L^4_\bx}\|\Delta_\tau\nabla\overline\bv_\varepsilon\|_{L^2_\bx}\dt\\
&\lesssim \int_I\|\overline\pi_\varepsilon\|_{W^{1,2}_\bx}\|\Delta_\tau\zeta_\varepsilon\|_{W^{2,2}_\by}\|\Delta_\tau\overline\bv_\varepsilon\|_{W^{1,2}_\bx}\dt\\
&\lesssim \delta\int_I\|\Delta_\tau\overline\bv_\varepsilon\|_{W^{1,2}_\bx}^2\dt+c(\delta)\int_I\|\overline\pi_\varepsilon\|_{W^{1,2}_\bx}^2\dt.
\end{align*}
The first term can eventually be absorbed while we estimate the second one using Lemma~\ref{lem:pres}. Next we estimate 
\begin{align*}
&\int_I\int_\Omega\overline \pi_\varepsilon^\tau\Div(\Delta_\tau \bfB_{\zeta_\varepsilon}^\intercal\bfG)\dx\dt\\
&\lesssim \int_I\|\overline \pi_\varepsilon\|_{L^6_\bx}\|\Delta_\tau\naby^2{\zeta_\varepsilon}\|_{L^2_\by}\|\Delta_\tau\naby{\zeta_\varepsilon}\|_{L^6_\by}(1+\|\naby{\zeta_\varepsilon}\|_{L^\infty_\by})\|\overline\bu_\varepsilon^\tau\|_{L^6_\bx}\dt\\
&\quad + \int_I\|\overline\pi_\varepsilon\|_{L^6_\bx}\|\Delta_\tau\naby{\zeta_\varepsilon}\|_{L^6_\by}\|\partial_t\nabla_\by\zeta_\varepsilon\|_{L^2_\by}\|\nabla^2_\by\zeta_\varepsilon\|_{L^\infty_\by}\|\overline\bu_\varepsilon^\tau\|_{L^6_\bx}\dt\\
&\quad + \int_I\|\overline \pi_\varepsilon\|_{L^6_x}\|\Delta_\tau\naby{\zeta_\varepsilon}\|^2_{L^6_\by}(1+\|\naby{\zeta_\varepsilon}\|_{L^\infty_\by})\|\nabla\overline\bu_\varepsilon^\tau\|_{L^2_\bx}\dt\\
&\lesssim \int_I\big(\|\overline \pi_\varepsilon\|_{W^{1,2}_\bx}^2+\|\overline\bu_\varepsilon^\tau\|_{W^{1,2}_\bx}^2\big)\dt.
\end{align*}
 Combining the two terms and using Lemma~\ref{lem:pres} we finally get
\begin{align*}
\text{11th} \lesssim  \delta\int_I\|\Delta_\tau\overline\bv_\varepsilon\|_{W^{1,2}_\bx}^2\dt+\int_{I^\tau}\big(\mathcal E_\eta^{5}+1\big)\dt+\int_I\|\bff\circ\bfPsi_{\zeta_\varepsilon}\|_{L^2_\bx}^2\dt.
\end{align*}

{\textbf{Estimates for the divergence-corrector $G$, terms 12  until 15.}}

\noindent
It remains to estimate the terms related to $\bfG=\bfB_{\zeta_\varepsilon}^{-\intercal}\Delta_\tau\bfB_{\zeta_\varepsilon}^\intercal\overline\bu_\varepsilon^\tau$. For the 12th and 13th integral in \eqref{est}, we decompose it into (using the definition of $\bfh_\zeta$)
\begin{align*}
({\tt I})^{\bfh,\bfG}&:= -\int_I\int_\Omega\Delta_\tau\big(J_{\zeta_\varepsilon}\partial_t \overline{\bu}_\varepsilon\big)\cdot \bfG\dxt\\
({\tt II})^{\bfh,\bfG}&:=- \int_I\int_\Omega\Delta_\tau\big(
J_{\zeta_\varepsilon} \nabx\overline{\bu}_\varepsilon\cdot \partial_t \bfPsi_{\zeta_\varepsilon}^{-1}\circ \bfPsi_{\zeta_\varepsilon}\big)\cdot\bfG\dxt
\\({\tt III })^{\bfh,\bfG}&:=- \int_I\int_\Omega\Delta_\tau\big(\mathbf{B}_{\zeta_\varepsilon}\nabx\overline{\bfv}_\varepsilon~\overline{\bu}_\varepsilon\big)\cdot\bfG\dxt\\
({\tt IV})^{\bfh,\bfG}&:= \int_I\int_\Omega\Delta_\tau\big(
J_{\zeta_\varepsilon}  \bff\circ \bfPsi_{\zeta_\varepsilon}\big)\cdot\bfG\dxt.
\end{align*}
Note that the last term is of lower order and we obtain that
\begin{align*}
({\tt IV})^{\bfh,\bfG}
&\lesssim \int_{I^\tau}\big(\|\bff\circ \bfPsi_{\zeta_\varepsilon}\|_{L^2_\bx}^2+\|\partial_t\bff\circ \bfPsi_{\zeta_\varepsilon}\|_{L^2_\bx}^2+\|\nabla\bff\circ \bfPsi_{\zeta_\varepsilon}\|^2_{L^2_\bx}\big)\dt+\int_{I^\tau}\|\overline\bu_\varepsilon\|_{W^{1,2}_\bx}^2\dt.
\end{align*}
With the embedding \eqref{ineq2},
we now perform the following estimate:
\begin{align*}
({\tt III })^{\bfh,\bfG}&=- \int_I\int_\Omega\Delta_\tau\mathbf{B}_{\zeta_\varepsilon}\nabx\overline{\bfv}_\varepsilon~\overline{\bu}_\varepsilon\cdot\bfG\dxt- \int_I\int_\Omega\mathbf{B}_{\zeta_\varepsilon}^\tau\Delta_\tau\nabx\overline{\bfv}_\varepsilon~\overline{\bu}_\varepsilon\cdot\bfG\dxt\\
&\quad - \int_I\int_\Omega\mathbf{B}^\tau_{\zeta_\varepsilon}\nabx\overline{\bfv}_\varepsilon^\tau~\Delta_\tau\overline{\bu}_\varepsilon\cdot\bfG\dxt\\
&\lesssim \int_I\|\nabla\overline\bu_\varepsilon\|_{L^2_\bx}\|\Delta_\tau\naby{\zeta_\varepsilon}\|_{L^\infty_\by}\|\overline\bu_\varepsilon\|_{L^6_\bx}^2\|\Delta_\tau\naby{\zeta_\varepsilon}\|_{L^6_\by}\dt+\int_I\|\Delta_\tau\nabla\overline\bu_\varepsilon\|_{L^2_\bx}\|\overline\bu_\varepsilon\|_{L^6_\bx}^2\|\Delta_\tau\naby{\zeta_\varepsilon}\|_{L^6_\by}\dt\\
&\quad +\int_I\|\nabla\overline\bu_\varepsilon^\tau\|_{L^2_\bx}\|\Delta_\tau\overline\bu_\varepsilon\|_{L^6_\bx}\|\Delta_\tau\naby{\zeta_\varepsilon}\|_{L^6_\by}\|\overline\bu_\varepsilon^\tau\|_{L^6_\bx}\dt\\
&\lesssim \int_{I}\|\partial_t\naby{\zeta_\varepsilon}\|_{W^{1+s,2}_\by}^4\dt+c(\delta)\int_{I^\tau}\|\overline\bu_\varepsilon\|_{W^{1,2}_\bx}^4\dt+\delta\int_I\|\Delta_\tau\overline\bu_\varepsilon\|_{W^{1,2}_\bx}^2\dt,
\end{align*}
and decompose
\begin{align*}
({\tt II})^{\bfh,\bfG}=&- \int_I\int_\Omega\Delta_\tau
J_{\zeta_\varepsilon} \nabx\overline{\bu}_\varepsilon\cdot \partial_t \bfPsi_{\zeta_\varepsilon}^{-1}\circ \bfPsi_{\zeta_\varepsilon}\cdot\bfG\dxt\\
&- \int_I\int_\Omega
J_{\zeta_\varepsilon}^\tau \Delta_\tau\nabx\overline{\bu}_\varepsilon\cdot (\partial_t \bfPsi_{\zeta_\varepsilon}^{-1}\circ \bfPsi_{\zeta_\varepsilon})^\tau\cdot\bfG\dxt\\
&- \int_I\int_\Omega
J_{\zeta_\varepsilon}^\tau \nabx\overline{\bu}_\varepsilon^\tau\cdot \Delta_\tau(\partial_t \bfPsi_{\zeta_\varepsilon}^{-1}\circ \bfPsi_{\zeta_\varepsilon})\cdot\bfG\dxt\\
=:&({\tt II})^{\bfh,\bfG}_1+({\tt II})^{\bfh,\bfG}_2+({\tt II})^{\bfh,\bfG}_3.
\end{align*}
Then we have by \eqref{210and212}, \eqref{218} and \eqref{eq:detPsi}
\begin{align*}
({\tt II})^{\bfh,\bfG}_1+({\tt II})^{\bfh,\bfG}_2
&\lesssim  \int_I(\|\nabla\overline\bu_\varepsilon\|_{L^2_\bx}+\|\Delta_\tau\nabla\overline\bu_\varepsilon\|_{L^2_\bx})\|\partial_t\zeta_\varepsilon\|_{L^\infty_\by}^2\|\Delta_\tau\nabla\zeta_\varepsilon\|_{L^4_\by}\|\overline\bu_\varepsilon\|_{L^4_\bx}\dt\\
&\lesssim  \delta\int_I\|\Delta_\tau\overline\bu_\varepsilon\|_{W^{1,2}_\bx}^2\dt+c(\delta)\int_I\|\overline\bu\|_{W^{1,2}_\bx}^2\dt.
\end{align*}
Finally, as for $({\tt II})^{\bfh}_3$ above,
\begin{align*}
({\tt II})^{\bfh,\bfG}_3
&\lesssim\,\int_I\|\nabla\overline\bu_\varepsilon\|_{L^2_\bx}\|\Delta_\tau\partial_t\zeta_\varepsilon\|_{L^4_\by}\|\Delta_\tau\nabla_\by\zeta_\varepsilon\|_{L^{20}_\by}\|\overline \bu_\varepsilon\|_{L^5_\bx}\dt+\int_I\|\nabla\overline\bu_\varepsilon\|_{L^2_\bx}\|\partial_t\nabla_\by\zeta_\varepsilon\|_{L^6_\by}^2\|\overline\bu_\varepsilon\|_{L^6_\bx}\dt\\
&\lesssim \int_I\|\overline \bu_\varepsilon\|_{W^{1,2}_\bx}^{\frac{19}{10}}\|\Delta_\tau\partial_t\zeta_\varepsilon\|_{W^{\frac{1}{2},2}_\by}\|\Delta_\tau\nabla_\by\zeta_\varepsilon\|_{W^{1,2}_{\by}}\|\overline\bu_\varepsilon\|_{L^2_\bx}^{\frac{1}{10}}\dt+\int_I\|\overline\bu_\varepsilon\|_{W^{1,2}_\bx}^2\|\partial_t\nabla_\by\zeta_\varepsilon\|_{W^{1,2}_\by}^2\dt\\
&\lesssim \int_{I}(1+\|\overline \bu_\varepsilon\|_{W^{1,2}_\bx}^4)\dt+\int_I|\|\partial_t^2\zeta_\varepsilon\|_{W^{\frac{1}{2},2}_\by}^2\dt,
\end{align*}
where we used the interpolation inequality:
\begin{equation}\label{ineq1}
\|\overline\bu_\varepsilon\|_{L^5_\bx}\lesssim \|\overline\bu_\varepsilon\|_{L^2_\bx}^{\frac{1}{10}}\|\overline\bu\varepsilon\|_{W^{1,2}_\bx}^{\frac{9}{10}}.
\end{equation} 
Now we estimate the 12th term that was defined as $({\tt I})^{\bfh,\bfG}$.  Taking an integration by parts we have 
\begin{align*}
\text{12th} &= ({\tt I})^{\bfh,\bfG}=-\int_I\int_\Omega\Delta_\tau J_{\zeta_\varepsilon}\partial_t\overline\bu_\varepsilon\cdot \bfG\dxt-\int_I\int_\Omega J^\tau_{\zeta_\varepsilon}\partial_t\Delta_\tau\overline\bu_\varepsilon\cdot \bfG\dxt\\
&=-\int_I\int_\Omega\Delta_\tau J_{\zeta_\varepsilon}\partial_t\overline\bu_\varepsilon\cdot \bfG\dxt+\int_I\int_\Omega \partial_t J^\tau_{\zeta_\varepsilon}\Delta_\tau\overline\bu_\varepsilon\cdot \bfG\dxt+\int_I\int_\Omega  J^\tau_{\zeta_\varepsilon}\Delta_\tau\overline\bu_\varepsilon\cdot \partial_t\bfG\dxt\\
&\quad -\int_\Omega J^\tau_{\zeta_\varepsilon}\Delta_\tau\overline\bu_\varepsilon\cdot \bfG\dx\bigg|_0^T = ({\tt I})^{\bfh,\bfG}_1+({\tt I})^{\bfh,\bfG}_2+({\tt I})^{\bfh,\bfG}_3+({\tt I})^{\bfh,\bfG}_4.
\end{align*}
The first two terms on the right hand side above are bounded by
\begin{align*}
({\tt I})^{\bfh,\bfG}_1+({\tt I})^{\bfh,\bfG}_2
&\lesssim  \int_I\|\partial_t\zeta_\varepsilon\|_{L^\infty_\by}(\partial_t\overline\bu_\varepsilon\|_{L^2_\bx}+\|\Delta_\tau\overline\bu_\varepsilon\|_{L^2_\bx})\|\Delta_\tau\nabla_\by\zeta_\varepsilon\|_{L^4_\by}\|\overline\bu_\varepsilon\|_{L^4_\bx}\dt\\
&\lesssim \int_I\|\partial_t\overline\bu_\varepsilon\|_{L^2_\bx}\|\Delta_\tau\nabla_\by\zeta_\varepsilon\|_{W^{1,2}_\by}\|\overline\bu_\varepsilon\|_{W^{1,2}_\bx}\dt\\
&\lesssim \int_I\|\partial_t\overline\bu_\varepsilon\|_{L^2_\bx}^2\dt+\int_I\|\overline\bu_\varepsilon\|_{W^{1,2}_\bx}^2\dt.
\end{align*}
The last integral in $({\tt I})^{\bfh,\bfG}$ can be estimated as
\begin{align*}
({\tt I})^{\bfh,\bfG}_4&\lesssim \sup_{I^\tau}\|\Delta_\tau\overline\bu_\varepsilon\|_{L^2_\bx}\|\overline\bu_\varepsilon\|_{L^4_\bx}\|\naby{\zeta_\varepsilon}\|_{L^{\infty}_\by}\|\Delta_\tau\naby{\zeta_\varepsilon}\|_{L^{4}_\by}\\
&\lesssim \sup_{I^\tau}\|\Delta_\tau\overline\bu_\varepsilon\|_{L^2_\bx}\|\overline\bu_\varepsilon\|_{L^2_\bx}^{1/4}\|\overline\bu_\varepsilon\|_{W^{1,2}_\bx}^{3/4}\\
&\lesssim \delta\sup_{I^\tau}\Big(\|\Delta_\tau\overline\bu_\varepsilon\|_{L^2_\bx}^2+\|\overline\bu_\varepsilon\|_{W^{1,2}_\bx}^{2} \Big)+c(\delta)\sup_{I^\tau}\|\overline\bu_\varepsilon\|^{2}_{L^2_\bx}.
\end{align*}

{\bf Control of the term  $({\tt I})^{\bfh,\bfG}_3$.}

\noindent
This is the most complicated term.  For this we compute $\partial_t\bfG$. 
 Since $\bfB_{\zeta_\varepsilon}=\mathrm{cof}(\nabla\bfPsi_{\zeta_\varepsilon})$, there is a bilinear mapping $\mathcal A:\R^{3\times 3}\times\R^{3\times 3}\rightarrow\R^{3\times 3}$
such that
\begin{align}
\partial_t\bfG&=\bfB_{\zeta_\varepsilon}^{-\intercal}\Delta_\tau\bfB_{\zeta_\varepsilon}^\intercal\partial_t\overline\bu_\varepsilon^\tau+\partial_t\bfB_{\zeta_\varepsilon}^{-\intercal}\Delta_\tau\bfB_{\zeta_\varepsilon}^\intercal\overline\bu_\varepsilon^\tau+\bfB_{\zeta_\varepsilon}^{-\intercal}\partial_t\Delta_\tau\big(\mathcal A(\nabla \bfPsi_{\zeta_\varepsilon}^\intercal,\nabla\bfPsi_{\zeta_\varepsilon}^\intercal)\big)\overline\bu_\varepsilon^\tau\label{dtg}\\
&= \bfB_{\zeta_\varepsilon}^{-\intercal}\Delta_\tau\bfB_{\zeta_\varepsilon}\partial_t\overline\bu_\varepsilon^\tau+\partial_t\bfB_{\zeta_\varepsilon}^{-\intercal}\Delta_\tau\bfB_{\zeta_\varepsilon}\overline\bu_\varepsilon^\tau+2\bfB_{\zeta_\varepsilon}^{-\intercal}\big(\mathcal A( \Delta_\tau\nabla \bfPsi_{\zeta_\varepsilon},\partial_t\nabla\bfPsi_{\zeta_\varepsilon})\big)\overline\bu_\varepsilon^\tau\nonumber\\
&\quad +\bfB_{\zeta_\varepsilon}^{-\intercal}\big(\mathcal A( \partial_t\Delta_\tau\nabla \bfPsi_{\zeta_\varepsilon},\nabla\bfPsi_{\zeta_\varepsilon})\big)\overline\bu_\varepsilon^\tau=: A_1+A_2+A_3.\nonumber
\end{align}
The most critical term in \eqref{dtg} is the last one. Writing $\mathcal A=\sum_{i,j=1}^3\mathcal A_{ij}$, with $\mathcal A_{ij}:\R^3\times\R^3\rightarrow\R^{3\times 3}$ acting on the column vectors of the matrices we obtain that
\begin{align*}
A_3&=-\int_I\int_\Omega J_{\zeta_\varepsilon}^\tau\Delta_{\tau}\overline\bu_\varepsilon\cdot\bfB_{\zeta_\varepsilon}^{-\intercal}\big(\mathcal A( \partial_t\Delta_\tau\nabla \bfPsi_{\zeta_\varepsilon},\nabla\bfPsi_{\zeta_\varepsilon})\big)\overline\bu_\varepsilon^\tau\dxt\\
&=-\sum_{i,j=1}^3\int_I\int_\Omega J_{\zeta_\varepsilon}^\tau \Delta_{\tau}\bu_\varepsilon\cdot\bfB_{\zeta_\varepsilon}^{-\intercal}\big(\mathcal A_{ij}( \partial_t\Delta_\tau\partial_i \bfPsi_{\zeta_\varepsilon},\partial_j\bfPsi_{\zeta_\varepsilon})\big)\overline\bu_\varepsilon^\tau\dxt\\
&=-\sum_{i,j=1}^3\int_I\int_{\partial\Omega}J_{{\zeta_\varepsilon}}^\tau\Delta_{\tau}\overline\bu_\varepsilon\cdot\bfB_{\zeta_\varepsilon}^{-\intercal}\big(\mathcal A_{ij}( \partial_t\Delta_\tau \bfPsi_{\zeta_\varepsilon},\partial_j\bfPsi_{\zeta_\varepsilon})\big)\overline\bu_\varepsilon\,n^i\,\dd \mathcal{H}^2\dt\\
&\quad +\sum_{i,j=1}^3\int_I\int_\Omega J_{\zeta_\varepsilon}^\tau\Delta_{\tau}\overline\bu_\varepsilon\cdot\bfB_{\zeta_\varepsilon}^{-\intercal}\big(\mathcal A_{ij}( \partial_t\Delta_\tau \bfPsi_{\zeta_\varepsilon},\partial_{ij}\bfPsi_{\zeta_\varepsilon})\big)\overline\bu_\varepsilon^\tau\dxt\\
&\quad +\sum_{i,j=1}^3\int_I\int_\Omega\partial_iJ_{\zeta_\varepsilon}^\tau\Delta_{\tau}\overline\bu_\varepsilon\cdot\bfB_{\zeta_\varepsilon}^{-\intercal}\big(\mathcal A_{ij}( \partial_t\Delta_\tau \bfPsi_{\zeta_\varepsilon},\partial_j\bfPsi_{\zeta_\varepsilon})\big)\overline\bu_\varepsilon^\tau\dxt\\
&\quad +\sum_{i,j=1}^3\int_I\int_\Omega J_{\zeta_\varepsilon}^\tau\Delta_{\tau}\partial_i\overline\bu_\varepsilon\cdot\bfB_{\zeta_\varepsilon}^{-\intercal}\big(\mathcal A_{ij}( \partial_t\Delta_\tau \bfPsi_{\zeta_\varepsilon},\partial_j\bfPsi_{\zeta_\varepsilon})\big)\overline\bu_\varepsilon^\tau\dxt\\
&\quad +\sum_{i,j=1}^3\int_I\int_\Omega J_{\zeta_\varepsilon}^\tau\Delta_{\tau}\overline\bu_\varepsilon\cdot\partial_i\bfB_{\zeta_\varepsilon}^{-\intercal}\big(\mathcal A_{ij}( \partial_t\Delta_\tau \bfPsi_{\zeta_\varepsilon},\partial_j\bfPsi_{\zeta_\varepsilon})\big)\overline\bu_\varepsilon^\tau\dxt\\
&\quad +\sum_{i,j=1}^3\int_I\int_\Omega J_{\zeta_\varepsilon}^\tau\Delta_{\tau}\overline\bu_\varepsilon\cdot\bfB_{\zeta_\varepsilon}^{-\intercal}\big(\mathcal A_{ij}( \partial_t\Delta_\tau \bfPsi_{\zeta_\varepsilon},\partial_j\bfPsi_{\zeta_\varepsilon})\big)\partial_i\overline\bu_\varepsilon^\tau\dxt\\
&=:J_0+J_1+\dots J_5.
\end{align*}
Using the estimates for $\bfPsi_{\zeta_\varepsilon}$ (cf. \eqref{210and212}, \eqref{218} and \eqref{eq:detPsi}) as well as $\zeta_\varepsilon \in X_{\tt shell}(0,T)$, we have by the trace embedding $W^{1,2}(\Omega)\hookrightarrow W^{1/2,2}(\partial\Omega)\hookrightarrow L^4(\partial\Omega)$
\begin{align*}
J_0&\lesssim \int_I\big(1+\|\naby{\zeta_\varepsilon}\|_{L^{\infty}_\by}^3\big)\|\partial_t^2{\zeta_\varepsilon}\|_{L^2_\by}\|\Delta_{\tau}\overline\bu_\varepsilon\|_{L^4(\partial\Omega)}\|\overline\bu_\varepsilon\|_{L^4(\partial\Omega)}\dt\\
&\lesssim \delta\int_I\|\Delta_{\tau}\overline\bu_\varepsilon\|_{W^{1,2}_\bx}^2\dt+c(\delta) \int_I\|\overline\bu_\varepsilon\|_{W^{1,2}_\bx}^2\dt.
\end{align*}
Similarly, it holds
\begin{align*}
J_1
&\lesssim \int_I\lVert\partial_t^2{\zeta_\varepsilon}\rVert_{L^2_\by}\lVert\Delta_{\tau}\overline\bu_\varepsilon\rVert_{L^4_\bx}\lVert\overline\bu_\varepsilon\rVert_{L^4_\bx}\lVert\naby^2{\zeta_\varepsilon}\rVert_{L^\infty_\by}\dt\\
&\lesssim \delta\int_I\lVert\Delta_{\tau}\overline\bu_\varepsilon\rVert_{W^{1,2}_\bx}^2\dt+c(\delta)\int_I\lVert\overline\bu_\varepsilon\rVert_{W^{1,2}_\bx}^2\dt,
\end{align*}
where we used the fact that $\naby^2{\zeta_\varepsilon}\in L^\infty(I\times \omega)$.
Moreover,
\begin{align*}
J_2&\lesssim \int_I\big(1+\|\naby{\zeta_\varepsilon}\|_{L^{\infty}_\by}^3\big)\|\partial_t^2{\zeta_\varepsilon}\|_{L^2_\by}\|\Delta_{\tau}\overline\bu_\varepsilon\|_{L^4_\bx}\|\overline\bu_\varepsilon\|_{L^4_\bx}\dt\\
&\lesssim \delta\int_I\|\Delta_{\tau}\overline\bu_\varepsilon\|_{W^{1,2}_\bx}^2\dt+c(\delta) \int_I\|\overline\bu_\varepsilon\|_{W^{1,2}_\bx}^2\dt,
\end{align*}
\begin{align*}
J_3&\lesssim \int_I\big(1+\|\naby{\zeta_\varepsilon}\|_{L^{\infty}_\by}^3\big)\|\partial_t^2{\zeta_\varepsilon}\|_{L^3_\by}\|\Delta_{\tau}\nabla\overline\bu_\varepsilon\|_{L^2_\bx}\|\overline\bu_\varepsilon\|_{L^6_\bx}\dt\\
&\lesssim \int_I\|\Delta_{\tau}\overline\bu_\varepsilon\|_{W^{1,2}_x}\|\overline\bu_\varepsilon\|_{W^{1,2}_\bx}\|\partial_t^2{\zeta_\varepsilon}\|_{L^3_\by}\dt\\
&\lesssim \delta\int_I\|\Delta_{\tau}\overline\bu_\varepsilon\|_{W^{1,2}_\bx}^2\dt+c(\delta) \int_I\big(\|\partial_t^2{\zeta_\varepsilon}\|_{L^3_\by}^4+\|\overline\bu_\varepsilon\|_{W^{1,2}_\bx}^4\big)\dt\\
&\lesssim \delta\int_I\|\Delta_{\tau}\overline\bu_\varepsilon\|_{W^{1,2}_\bx}^2\dt+c(\delta) \bigg(\int_I\|\overline\bu_\varepsilon\|_{W^{1,2}_\bx}^4\dt+1\bigg),
\end{align*}
where we used the embedding \eqref{1embed}.
Finally, we have
\begin{align*}
J_4&\lesssim\int_I(1+\|\naby{\zeta_\varepsilon}\|_{L^\infty_\by}^2)\|\partial_t^2{\zeta_\varepsilon}\|_{L^2_\by}\|\naby^2{\zeta_\varepsilon}\|_{L^{\infty}_\by}\|\Delta_{\tau}\overline\bu_\varepsilon\|_{L^{4}_\bx}\|\overline\bu_\varepsilon\|_{L^4_\bx}\dt\\
&\lesssim\int_I\|\Delta_{\tau}\overline\bu_\varepsilon\|_{W^{1,2}_\bx}\|\naby^2{\zeta_\varepsilon}\|_{L^{\infty}_\by}\|\overline\bu_\varepsilon\|_{W^{1,2}_\bx}\dt\\
&\lesssim\delta\int_I\|\Delta_{\tau}\overline\bu_\varepsilon\|_{W^{1,2}_\bx}^2\dx+c(\delta)\int_I\|\overline\bu_\varepsilon\|_{W^{1,2}_\bx}^{2}\dt,
\end{align*}
as well as
\begin{align*}
J_5&\lesssim\int_I(1+\|\naby{\zeta_\varepsilon}\|_{L^\infty_\by}^3)\|\partial_t^2{\zeta_\varepsilon}\|_{L^3_\by}\|\Delta_{\tau}\overline\bu_\varepsilon \|_{L^{6}_\bx}\|\nabla\overline\bu_\varepsilon\|_{L^2_\bx}\dt\\
&\lesssim\int_I\|\Delta_{\tau}\overline\bu_\varepsilon\|_{W^{1,2}_\bx}\|\partial_t^2{\zeta_\varepsilon}\|_{L^{3}_\by}\|\overline\bu_\varepsilon\|_{W^{1,2}_\bx}\dt\\
&\lesssim\delta\int_I\|\Delta_{\tau}\overline\bu_\varepsilon\|_{W^{1,2}_\bx}^2\dt+c(\delta)\int_I\Big(\|\partial_t^2{\zeta_\varepsilon}\|_{L^3_\by}^{4}+\|\overline\bu_\varepsilon\|_{W^{1,2}_\bx}^{4}\Big)\dt\\
&\lesssim\delta\int_I\|\Delta_{\tau}\overline\bu_\varepsilon\|_{W^{1,2}_\bx}^2\dx+c(\delta)\bigg(\int_I\|\overline\bu_\varepsilon\|_{W^{1,2}_\bx}^{4}\dt+1\bigg),
\end{align*}
using again ${\zeta_\varepsilon}\in  X_{\tt shell}(0,T)\subset W^{2,4}(I; L^3(\omega))$.

We continue the estimates of the terms involving $\partial_t\bfG$ from \eqref{dtg}. For the first term in \eqref{dtg}, we have
\begin{align*}
A_1&=\int_I\int_\Omega J_{\zeta_\varepsilon}^\tau\Delta_{\tau} \overline{\bu}_\varepsilon\cdot \bfB_{\zeta_\varepsilon}^{-\intercal}\Delta_\tau\bfB_{\zeta_\varepsilon}\partial_t\overline\bu_\varepsilon^\tau\dxt\\
&\lesssim \int_I\lVert\Delta_{\tau}\overline\bu_\varepsilon\rVert_{L^4_\bx}\lVert\naby{\zeta_\varepsilon}\rVert_{L^\infty_\by}\lVert\partial_t\naby{\zeta_\varepsilon}\rVert_{L^4_\by}\lVert\partial_t\overline\bu_\varepsilon^\tau\rVert_{L^2_\bx}\dt\\
&\lesssim \int_I\lVert\Delta_{\tau}\overline\bu_\varepsilon\rVert_{W^{1,2}_\bx}\lVert\partial_t\naby{\zeta_\varepsilon}\rVert_{W^{1,2}_\by}\lVert\partial_t\overline\bu_\varepsilon^\tau\rVert_{L^2_\bx}\dt\\
&\lesssim \delta\int_I\lVert\Delta_{\tau}\overline\bu_\varepsilon\rVert_{W^{1,2}_\bx}^2\dt+c(\delta)\int_{I^\tau}\lVert\partial_t\overline\bu_\varepsilon\rVert_{L^2_\bx}^2\dt,
\end{align*}
where we used $\partial_t\naby{\zeta_\varepsilon}\in L^\infty(I; W^{1,2}(\omega))$. 
 For the second term in \eqref{dtg}, we have 
\begin{align*}
A_2&=\int_I\int_\Omega J_{\zeta_\varepsilon}^\tau\Delta_{\tau} \overline{\bu}_\varepsilon\cdot \Delta_\tau\bfB_{\zeta_\varepsilon}^{-\intercal}\partial_t\bfB_{\zeta_\varepsilon}\overline\bu_\varepsilon^\tau\dxt\\
&\lesssim \int_I\lVert\Delta_{\tau}\overline\bu_\varepsilon\rVert_{L^2_\bx}\lVert\naby{\zeta_\varepsilon}\rVert_{L^\infty_\by}\lVert\partial_t\naby{\zeta_\varepsilon}\rVert_{L^6_\by}^2\lVert\overline\bu_\varepsilon^\tau\rVert_{L^6_\bx}\dt\\
&\lesssim \int_I\lVert\Delta_{\tau}\overline\bu_\varepsilon\rVert_{L^2_\bx}\lVert\overline\bu_\varepsilon^\tau\rVert_{W^{1,2}_\bx}\lVert\partial_t\naby{\zeta_\varepsilon}\rVert_{W^{1,2}_\by}^2\dt\\
&\lesssim \int_I\lVert\Delta_{\tau}\overline\bu_\varepsilon\rVert_{L^2_\bx}^2\dt+\int_{I^\tau}\lVert\overline\bu_\varepsilon\rVert_{W^{1,2}_\bx}^2\dt.
\end{align*}
The integral related to the third term on the right hand side of \eqref{dtg} can be estimated in a similar way, thereby we finished estimating $({\tt I})^{\bfh,\bfG}$.

{\bf The last two integrals related to $\nabla \bfG$ in \eqref{est}, namely the 14th and 15th term .}

\noindent
 For these terms we will be using repeatedly \eqref{210and212} and \eqref{218}. As $\bfG=\bfB_{\zeta_\varepsilon}^{-\intercal}\Delta_\tau\bfB_{\zeta_\varepsilon}^\intercal\overline\bu_\varepsilon^\tau$, we have 
\begin{align*}
\text{14th}&=\int_I\int_\Omega\Delta_\tau\bfA_{\zeta_\varepsilon}\nabla\overline\bu_\varepsilon:\nabla\bfG\dx\dt\\
&=\int_I\int_\Omega\Delta_\tau\bfA_{\zeta_\varepsilon}\nabla\overline\bu_\varepsilon:(\nabla\bfB_{\zeta_\varepsilon}^{-\intercal}\Delta_\tau\bfB_{\zeta_\varepsilon}^\intercal\overline\bu_\varepsilon^\tau+\bfB_{\zeta_\varepsilon}^{-\intercal}\Delta_\tau\nabla\bfB_{\zeta_\varepsilon}^\intercal\overline\bu_\varepsilon^\tau+\bfB_{\zeta_\varepsilon}^{-\intercal}\Delta_\tau\bfB_{\zeta_\varepsilon}^\intercal\nabla\overline\bu_\varepsilon^\tau)\dx\dt.
\end{align*}
Using again the fact that $\naby^2{\zeta_\varepsilon}\in L^\infty(I\times \omega)$, we derive that
\begin{align*}
&\int_I\int_\Omega\Delta_\tau\bfA_{\zeta_\varepsilon}\nabla\overline\bu_\varepsilon:\nabla\bfB_{\zeta_\varepsilon}^{-\intercal}\Delta_\tau\bfB_{\zeta_\varepsilon}^\intercal\overline\bu_\varepsilon^\tau\dx\dt\\
&\lesssim \int_I\lVert\Delta_\tau\naby{\zeta_\varepsilon}\rVert_{L^6_\by}^2\lVert\nabla\overline\bu_\varepsilon\rVert_{L^2_\bx}\lVert\naby^2{\zeta_\varepsilon}\rVert_{L^\infty_\by}\lVert\overline\bu_\varepsilon^\tau\rVert_{L^6_\bx}\dt\\
&\lesssim \int_I\lVert\Delta_\tau\naby{\zeta_\varepsilon}\rVert_{W^{1,2}_\by}^2\big(\lVert\overline\bu_\varepsilon\rVert_{W^{1,2}_\bx}^2+\lVert\overline\bu_\varepsilon^\tau\rVert_{W^{1,2}_\bx}^2\big)\dt\lesssim \int_{I^\tau}\lVert\overline\bu_\varepsilon\rVert_{W^{1,2}_\bx}^2\dt.
\end{align*}
By using the embedding,
\begin{equation*}
\zeta_\varepsilon\in W^{2,\infty}(I; L^2(\omega))\cap L^\infty(I;W^{4,2}(\omega))\hookrightarrow W^{\frac{5}{4},2}(I;W^{2,4}(\omega))\hookrightarrow W^{1,2}(I; W^{2,4}(\omega)),
\end{equation*}
 and \eqref{ineq1} we continue  by estimating
\begin{align*}
&\int_I\int_\Omega\Delta_\tau\bfA_{\zeta_\varepsilon}\nabla\overline\bu_\varepsilon:\bfB_{\zeta_\varepsilon}^{-\intercal}\Delta_\tau\nabla\bfB_{\zeta_\varepsilon}^\intercal\overline\bu_\varepsilon^\tau\dx\dt\\
&\lesssim \int_I\|\Delta_\tau\nabla_\by\zeta_\varepsilon\|_{L^{20}_\by}\|\nabla\overline\bu_\varepsilon\|_{L^2_\bx}\|\nabla_\by\zeta_\varepsilon\|_{L^\infty_\by}\|\Delta_\tau\nabla_\by^2\zeta_\varepsilon\|_{L^4_\by}\|\overline\bu_\varepsilon\|_{L^5_\bx}\dt\\
&\lesssim \int_I\|\Delta_\tau\nabla_\by\zeta_\varepsilon\|_{W^{1,2}_\by}\|\overline\bu_\varepsilon\|_{W^{1,2}_\bx}^{\frac{19}{10}}\|\Delta_\tau\nabla_\by^2\zeta_\varepsilon\|_{L^4_\by}\|\overline\bu_\varepsilon\|_{L^2_\bx}^{\frac{1}{10}}\dt\\
&\lesssim \int_I\|\partial_t\nabla_\by^2\zeta_\varepsilon\|_{L^4_\by}^2\dt+\int_I(1+\|\overline\bu_\varepsilon\|_{W^{1,2}_\bx}^4)\dt.
\end{align*}
We finally have 
\begin{align*}
&\int_I\int_\Omega\Delta_\tau\bfA_{\zeta_\varepsilon}\nabla\overline\bu_\varepsilon:\bfB_{\zeta_\varepsilon}^{-\intercal}\Delta_\tau\bfB_{\zeta_\varepsilon}^\intercal\nabla\overline\bu_\varepsilon^\tau\dx\dt\\
&\lesssim \int_I\lVert\Delta_\tau\naby{\zeta_\varepsilon}\rVert_{L^\infty_\by}^2\big(\lVert\overline\bu_\varepsilon\rVert_{W^{1,2}_\bx}^2+\lVert\overline\bu_\varepsilon^\tau\rVert_{W^{1,2}_\bx}^2\big)\lVert\naby{\zeta_\varepsilon}\rVert_{L^\infty_\by}\dt\\
&\lesssim \int_I\lVert\partial_t\naby{\zeta_\varepsilon}\rVert_{W^{s+1, 2}_\by}^4\dt+\int_{I^\tau}\lVert\overline\bu_\varepsilon\rVert_{W^{1,2}_\bx}^4\dt,
\end{align*}
where for $0<s\leq \frac{1}{4}$ we use the embedding \eqref{ineq2}. This concludes the estimate of the 14th term.

For the 15th and last integral in \eqref{est}, we have
\begin{align*}
\text{15th}&=\int_I\int_\Omega\bfA_{\zeta_\varepsilon}^\tau\nabla\Delta_\tau\overline\bu_\varepsilon:\nabla\bfG\dx\dt\\
&=\int_I\int_\Omega\bfA_{\zeta_\varepsilon}^\tau\Delta_\tau\nabla\overline\bu_\varepsilon:(\nabla\bfB_{\zeta_\varepsilon}^{-\intercal}\Delta_\tau\bfB_{\zeta_\varepsilon}^\intercal\overline\bu_\varepsilon^\tau+\bfB_{\zeta_\varepsilon}^{-\intercal}\Delta_\tau\nabla\bfB_{\zeta_\varepsilon}^\intercal\overline\bu_\varepsilon^\tau+\bfB_{\zeta_\varepsilon}^{-\intercal}\Delta_\tau\bfB_{\zeta_\varepsilon}^\intercal\nabla\overline\bu_\varepsilon^\tau)\dx\dt.
\end{align*}
Again we estimate the above terms one by one. First, we have
\begin{align*}
&\int_I\int_\Omega\bfA_{\zeta_\varepsilon}^\tau\Delta_\tau\nabla\overline\bu_\varepsilon:\nabla\bfB_{\zeta_\varepsilon}^{-\intercal}\Delta_\tau\bfB_{\zeta_\varepsilon}^\intercal\overline\bu_\varepsilon^\tau\dt\\
&\lesssim \int_I\lVert\naby{\zeta_\varepsilon}\rVert_{L^\infty_\by}\lVert\Delta_\tau\nabla\overline\bu_\varepsilon\rVert_{L^2_\bx}\lVert\naby^2{\zeta_\varepsilon}\rVert_{L^\infty_\by}\lVert\Delta_\tau\naby{\zeta_\varepsilon}\rVert_{L^3_\by}\lVert\overline\bu_\varepsilon^\tau\rVert_{L^6_\bx}\dt\\
&\lesssim \int_I\lVert\Delta_\tau\nabla\overline\bu_\varepsilon\rVert_{L^2_\bx}\lVert\Delta_\tau\naby{\zeta_\varepsilon}\rVert_{W^{1,2}_\by}\lVert\overline\bu_\varepsilon^\tau\rVert_{W^{1,2}_\bx}\dt\\
&\lesssim \delta\int_I\lVert\Delta_\tau\nabla\overline\bu_\varepsilon\rVert_{L^2_\bx}^2\dt+c(\delta)\int_{I^\tau}\lVert\overline\bu_\varepsilon\rVert_{W^{1,2}_\bx}^2\dt.
\end{align*}
Next based on the embedding
$$ \zeta_\varepsilon\in W^{2,\infty}(I; L^2(\omega))\cap L^\infty(I;W^{4,2}(\omega))\hookrightarrow W^{\frac{4}{3},2}(I; W^{2,3}(\omega))\hookrightarrow W^{1,4}(I; W^{2,3}(\omega)),$$
 we estimate
\begin{align*}
&\int_I\int_\Omega\bfA_{\zeta_\varepsilon}^\tau\Delta_\tau\nabla\overline\bu_\varepsilon:\bfB_{\zeta_\varepsilon}^{-\intercal}\Delta_\tau\nabla\bfB_{\zeta_\varepsilon}^\intercal\overline\bu_\varepsilon^\tau\dx\dt\\
&\lesssim \int_I\lVert\Delta_\tau\nabla\overline\bu_\varepsilon\rVert_{L^2_\bx}\lVert\Delta_\tau\naby^2{\zeta_\varepsilon}\rVert_{L^3_\by}\lVert\overline\bu_\varepsilon^\tau\rVert_{L^6_\bx}\dt\\
&\lesssim \delta\int_I\lVert\Delta_\tau\nabla\overline\bu_\varepsilon\rVert_{L^2_\bx}^2\dt+c(\delta)\int_{I^\tau}\lVert\partial_t\naby^2{\zeta_\varepsilon}\rVert_{L^3_\by}^4\dt+c(\delta)\int_{I^\tau}\lVert\overline\bu_\varepsilon\rVert_{W^{1,2}_\bx}^4\dt.
\end{align*}
Finally, using \eqref{ineq2} again we also have
\begin{align*} 
&\int_I\int_\Omega\bfA_{\zeta_\varepsilon}^\tau\Delta_\tau\nabla\overline\bu_\varepsilon:\bfB_{\zeta_\varepsilon}^{-\intercal}\Delta_\tau\bfB_{\zeta_\varepsilon}^\intercal\nabla\overline\bu_\varepsilon^\tau\dx\dt\\
&\lesssim \int_I\lVert\Delta_\tau\nabla\overline\bu_\varepsilon\rVert_{L^2_\bx}\lVert\nabla\overline\bu_\varepsilon^\tau\rVert_{L^2_\bx}\lVert\naby{\zeta_\varepsilon}\rVert_{L^\infty_\by}^2\lVert\Delta_\tau\naby{\zeta_\varepsilon}\rVert_{L^\infty_\by}\dt\\
&\lesssim \int_I\lVert\partial_t\nabla\overline\bu_\varepsilon\rVert_{L^2_\bx}\lVert\Delta_\tau\naby{\zeta_\varepsilon}\rVert_{W^{s+1,2}_\by}\lVert\nabla\overline\bu_\varepsilon^\tau\rVert_{L^2_\bx}\dt\\
&\lesssim \delta\int_I\lVert\Delta_\tau\nabla\overline\bu_\varepsilon\rVert_{L^2_\bx}^2\dt+c(\delta)\int_I\|\overline\bu_\varepsilon^\tau\|_{W^{1,2}_\bx}^4\dt+c(\delta)\int_I\lVert\partial_t\naby{\zeta_\varepsilon}\rVert_{W^{s+1,2}_\by}^4\dt.
\end{align*}
We have finished estimating the terms depending on $\nabla\bfG$. Now all terms in \eqref{est} are estimated. 

Putting \eqref{est} and \eqref{eq:1508} together and using \eqref{eq:1905} we have shown that for $0<\varepsilon\ll 1$, 
\begin{align*}
&\int_{I}\|\Delta_\tau\eta_\varepsilon\|_{W^{s+2,2}_\by}^2\dt+\varepsilon\sup_I\|\Delta_\tau\eta_\varepsilon\|_{W^{s+1,2}_\by}^2+\sup_I
\int_\omega\big(\vert \partial_t\Delta_\tau\eta_\varepsilon\vert^2
+
\vert  \Dely \Delta_\tau\eta_\varepsilon\vert^2\big)\dy
 \\
 &\quad +
 \sup_I
 \int_\Omega\vert  \Delta_\tau\overline{\bu}_\varepsilon\vert^2\dx
+ \varepsilon\int_I\int_\omega\vert  \partial_t\Delta_\tau\naby\eta_\varepsilon\vert^2\dy\dt
+
 \int_I\int_\Omega\vert  \Delta_\tau\nabx\overline{\bu}_\varepsilon\vert^2\dx\dt
\\
&\lesssim 
 \delta\int_{I^\tau}\|\Delta_{\tau}\overline\bu_\varepsilon\|_{W^{1,2}_\bx}^2\dx+\delta\Big(\sup_{I^\tau}\|\overline\bu_\varepsilon\|_{W^{1,2}_\bx}^2+\sup_I\|\Delta_\tau\overline\bu\|_{L^2_\bx}^2+\sup_I\|\Delta_\tau\partial_t\eta_\varepsilon\|_{L^2_\by}^2\Big)\\
&\quad +\int_I(\mathcal E^5_{\eta_\varepsilon}(s)+1)\ds+\int_I H(s)\ds+\varepsilon^2\int_I\|\partial_t\Delta_\tau\naby\eta_\varepsilon\|_{L^2_\by}^2\dt\\
&\quad+\|\partial_t\Delta_\tau\eta_\varepsilon(0)\|_{L^2_\by}^2+ \|\Delta_\by\Delta_\tau\eta(0)\|_{L^2_\by}^2
+\|\Delta_\tau \overline\bu_\varepsilon(0)\|_{L^2_\bx}^2+\varepsilon\|\Delta_\tau\eta_\varepsilon(0)\|_{W^{1+s,2}_\by}^2.
\end{align*} 
Finally, recalling \eqref{eq:cont} we 
find that
\begin{align*}
&\limsup_{\tau\to 0}\|\partial_t\Delta_\tau\eta_\varepsilon(0)\|_{L^2_\by}^2+ \|\Delta_\by\Delta_\tau\eta(0)\|_{L^2_\by}^2
+\|\Delta_\tau \overline\bu_\varepsilon(0)\|_{L^2_\bx}^2)\leq 
\limsup_{\tau\to 0}\dashint_0^\tau \mathcal{E}_{\eta_\varepsilon}\, \ds 
\\
&\quad \leq  c\left(1+\|\overline\bu_0\|_{W^{2,2}_\bx}+\|\eta_0\|_{W^{4,2}_\by}+\|\eta_*\|_{W^{2,2}_\by}+\|\overline\bff_0\|_{L^2_\bx}+\|g_0\|_{L^2_\by}\right).  
\end{align*}
Hence we can pass to the limit $\tau\rightarrow0$ obtaining
 for $0<\varepsilon\ll 1$
\begin{equation*}
\begin{aligned}
&\int_{I}\|\partial_t\eta_\varepsilon\|_{W^{s+2,2}_\by}^2\dt
+\sup_I
\int_\omega\big(\vert \partial_t^2\eta_\varepsilon\vert^2
+
\vert  \Dely \partial_t\eta_\varepsilon\vert^2\big)\dy\\
& \quad +
 \sup_I
 \int_\Omega\vert  \partial_t\overline{\bu}_\varepsilon\vert^2\dx
+
 \sup_I
 \int_\Omega\vert  \nabla\overline{\bu}_\varepsilon\vert^2\dx
+ 
 \int_I\int_\Omega\vert  \partial_t\nabx\overline{\bu}_\varepsilon\vert^2\dx\dt
\\
&\lesssim \delta\sup_{I}\Big(\|\overline\bu_\varepsilon\|_{W^{1,2}_\bx}^2+\|\partial_t\overline\bu_\varepsilon\|_{L^2_\bx}^2+\|\partial_t^2\eta_\varepsilon\|_{L^2_\by}^2\Big)+\int_I(\mathcal E^5_{\eta_\varepsilon}(s)+1)\ds+\int_I H(s)\ds\\
&\quad +1+\|\overline\bu_0\|_{W^{2,2}_\bx}+\|\eta_0\|_{W^{4,2}_\by}+\|\eta_*\|_{W^{2,2}_\by}+\|\overline\bff_0\|_{L^2_\bx}+\|g_0\|_{L^2_\by}+\|\eta_*\|_{W^{1+s,2}_\by}^2,
\end{aligned}
\end{equation*}
where we used \eqref{initialconv} and the embedding
$$W^{1,2}(I;L^2(\Omega))\hookrightarrow L^\infty(I;L^2(\Omega)).$$ 

We can now absorb the $\delta$-terms
obtaining the estimate
\begin{align}\label{eq:12.11}
\mathcal E_{\eta_\varepsilon}(t)+\int_0^t\|\partial_t\nabla\overline\bu_\varepsilon\|^2_{L^2_\bx}\ds+\int_0^t\|\partial_t{\eta_\varepsilon}\|_{W^{2+s,2}_\by}^2\ds\leq\,c_0+c_1\int_0^t H(s)\ds+c_2\int_0^t \mathcal E_{\eta_\varepsilon}^{5}(s)\ds.
\end{align}
Note that this estimate is independent from $\varepsilon$ which allows us to extend the existence time in Proposition \ref{thm:accestn} to an $\varepsilon$-independent time $T_{\tt max}$. Now we pass to the limit $\varepsilon\to 0$ in \eqref{eq:12.11} which finishes the proof.
\end{proof}

\begin{corollary}\label{cor:self}
Under the assumptions of Proposition \ref{thm:transformedSystemLinear}, we have
$(\eta,\overline \bu,\overline \pi)\in X_{\tt shell}\times X_{\tt fluid}$
together with $\|\eta\|_{X_{\tt shell}(0,T_{\tt max})}\leq R$.
\end{corollary}
\begin{proof}
We can now apply the local Gr\"{o}nwall's lemma from \cite[lemma 24]{DiRu2005} (assuming that $T_0(\mathcal E(0)+1)<1$) means of $\mathcal E(0)$ and $\int_0^{T_0}H(s)\ds$. In particular, the left-hand side
can be made (much) smaller than $R^2$, if we choose $T_{\tt max}$ and $R$ accordingly, i.e.
\begin{align}\label{eq:09.08}
\mathcal E_\eta(t)+\int_0^t\|\partial_t\nabla\overline\bu\|^2_{L^2_\bx}\ds+\int_0^t\|\partial_t\eta\|_{W^{2+s,2}_\by}^2\ds\leq\,\delta (T_{\tt max},R) R^2.
\end{align}
Arguing exactly as in the proof of Lemma \ref{lem:pres} (but taking the supremum in time instead of the integral) we obtain
\begin{align*}
\sup_I\Big(\|\overline\bfv\|_{W^{2,2}_\bx(\Omega)}^2+\|\overline\pi\|_{W_\bx^{1,2}(\Omega)}^2\Big)&\lesssim \sup_I\Big(1+\mathcal E_\eta^{5}+\|\bff\circ\bfPsi_{\zeta}\|_{L^2_\bx}^2\Big),
\end{align*}
which is uniformly bounded by \eqref{eq:09.08}.

Based on equation \eqref{shellEqAloneBar'}, we get now further
\begin{align*}
\sup_I\int_\omega|\Dely^2 \eta|^2\dy
&\lesssim \sup_I\int_\omega\left(|\partial_t^2 \eta|^2+|g|^2\right)\dy+\sup_I\|\naby\zeta\|_{L^8(\omega)}^4\|\nabla\overline\bu\|_{L^4(\partial\Omega)}^2+\sup_I\|\naby\zeta\|_{L^4(\omega)}^2\|\overline\pi\|_{L^4(\partial\Omega)}^2\\
&\lesssim \sup_I\int_\omega\left(|\partial_t^2 \eta|^2+|g|^2\right)\dy +\sup_I\|\naby\zeta\|_{W^{1,2}(\omega)}^4\|\nabla\overline\bu\|_{W^{\frac{1}{2},2}(\partial\Omega)}^2\\
&\quad +\sup\|\naby\zeta\|_{W^{1, 2}(\omega)}^2\|\overline\pi\|_{W^{\frac{1}{2},2}(\partial\Omega)}^2\\
&\lesssim \sup_I\int_\omega\left(|\partial_t^2 \eta|^2+|g|^2\right)\dy+\sup_I\|\nabla\overline\bu\|_{W^{1,2}(\Omega)}^2+\sup_I\|\overline\pi\|_{W^{1,2}(\Omega)}^2,
\end{align*}
%
%
%
%
%
as well as
\begin{align*}
\int_I\|\Dely^2 \eta\|^2_{W^{1/2,2}(\omega)}&\lesssim \int_I\|\partial_t^2 \eta\|^2_{W^{1/2,2}(\omega)}+\int_I\int_\Omega|\nabla^2\overline\bfv|^2\dxt+\int_I\int_\Omega|\nabla\overline\pi|^2\dxt+\int_I\|g\|^2_{W^{1/2,2}(\omega)}\dt,
\end{align*}
where 
\begin{align*}
\|\partial_t^2\eta\|_{W^{1/2,2}_\by}\lesssim \|\partial_t\overline\bu\|_{W^{1,2}_\bx},
\end{align*}
where the last term is controlled by Proposition \ref{thm:transformedSystemLinear} yielding a uniform control of $\mathcal E(t)$. 
	We conclude that $(\eta,\bu,\pi)\in X_{\tt shell}\times X_{\tt fluid}$ with $\|\eta\|_{X_{\tt shell}(0,T_{\tt max})}\leq R$.
\end{proof}
%

\subsection{The contraction argument}
\label{sec:3.3}
We argue similarly to \cite[Section 5]{BMSS} (see also \cite{schwarzacher2022weak}), where the weak-strong uniqueness property of the fluid-structure problem with dissipation in the shell is studied. Here, we do not have visco-elastic dissipation at hand. However, we estimate the difference of two strong solutions rather than a weak and a strong one.
We aim to compare two strong solutions
$(\bu_1,\eta_1)$ and $(\bu_2,\eta_2)$ with geometries corresponding to $\zeta_1$ and $\zeta_2$, respectively. 
\begin{proposition}\label{prop:contraction}
Fix $R,T_{\tt max}>0$ as in Proposition \ref{thm:transformedSystemLinear}.
Suppose that the dataset
$(\bff,g, \eta_0, \eta_*, \overline{\bu}_0 )$
satisfies \eqref{dataset'}
and that $\zeta_1,\zeta_2\in X_{\tt shell}(0,T_{\tt max})$ with $\|\zeta_1\|_{X_{\tt shell}(0,T_{\tt max})}\leq R$ and $\|\zeta_2\|_{X_{\tt shell}(0,T_{\tt max})}\leq R$. Suppose that there are two solutions
$( \eta_1, \overline{\bu}_1)$ and $(\eta_2, \overline\bu_2)$ of \eqref{contEqAloneBar'}--\eqref{momEqAloneBar'} in $X_{\tt shell}\times X_{\tt fluid}$  corresponding to $\zeta_1$ and $\zeta_2$, respectively. Then there is $T_{\tt max}'\in(0,T_{\tt max}]$ such that
\begin{align*}
&\sup_{0<t<T'_{\tt max}}\Big(\|\partial_t(\eta_1-\eta_2)(t)\|_{L^2_\by}^2+\|\Dely(\eta_1-\eta_2)(t)\|_{L^2_\by}^2\Big)\\
&\leq\,\tfrac{1}{2}\sup_{0<t<T'_{\tt max}}\Big(\|\partial_t(\zeta_1-\zeta_2)(t)\|_{L^2_\by}^2+\|\Dely(\zeta_1-\zeta_2)(t)\|_{L^2_\by}^2\Big).
			\end{align*}
\end{proposition}

\begin{proof}
We transform $\bu_2$ and $\pi_2$ to the domain of the pair $(\bu_1,\eta_1)$ (that is $\Omega_{\zeta_1}$) by setting
\begin{align*}
\underline{\bu}_2:=\bu_2\circ\bfPsi_{\zeta_2-\zeta_1},\quad \underline\pi_2:=\pi_2\circ\bfPsi_{\zeta_2-\zeta_1},
\quad \underline{\mathbf f}_2:=\mathbf f_2\circ\bfPsi_{\zeta_2-\zeta_1},
\end{align*}
where the Hanzawa transform $\bm{\Psi}_{\zeta_2-\zeta_1} : \Omega_{\zeta_1} \rightarrow\Omega_{\zeta_2}$ is defined similarly as in \eqref{map}. Next we introduce a suitable Bogovskij operator for our setting:
\[
	\Bog_{\zeta_1}(f):=\Bog(f\chi_{\Omega_{\zeta_1}}),
\] 
where $\Bog$ is defined in Theorem~\ref{thm:ndBog}, depending on $\norm{\zeta_1}_{L^\infty_\by}=:L$ and $\norm{\naby\zeta_1}_{L^\infty_\by}=:C_L$ only. 
 	
To obtain the difference estimate, we test the equation for $(\bu_1-\underline\bu_2,\eta_1-\eta_2)$ by the pair $(\bu_1-\underline\bu_2+\mathrm{Bog}_{\zeta_1}(\Div\underline\bu_2),\partial_t(\eta_1-\eta_2))$. We find that
\begin{align}
&\tfrac{1}{2}\int_{\Omega_{\zeta_1(t)}}|\bu_1(t)-\underline{\bu}_2(t)|^2\dx+\int_0^t
\int_{\Omega_{\zeta_1( s )}}|\nabla(\bu_1-\underline{\bu}_2)|^2\dx\ds\nonumber\\
&\quad+\tfrac{1}{2}\int_\omega|\partial_t(\eta_1-\eta_2)(t)|^2\dy+\tfrac{1}{2}\int_\omega|\Dely(\eta_1-\eta_2)(t)|^2\dy\nonumber\\
&\leq 
\tfrac{1}{2}\int_{\Omega_{\zeta_1(0)}}|\bu_1(0)-\underline{\bu}_2(0)|^2\dx+\tfrac{1}{2}\int_\omega|\partial_t(\eta_1-\eta_2)(0)|^2\dy+\tfrac{1}{2}\int_\omega|\Dely(\eta_1-\eta_2)(0)|^2\dy\nonumber\\
&\quad-\int_0^t\int_{\Omega_{\zeta_1( s )}}\big((\mathbb I_{3\times 3}-J_{\zeta_2-\zeta_1})\partial_t \underline{\bu}_2\big)\cdot\big(\bu_1- \underline\bu_2+ \mathrm{Bog}_{\zeta_1}(\Div\underline\bu_2)\big)\dx\ds\nonumber\\
&\quad+\int_0^t\int_{\Omega_{\zeta_1( s )}}\big(J_{\zeta_2-\zeta_1}\nabx \underline{\bu}_2 \partial_t \bfPsi_{\zeta_2-\zeta_1}^{-1}\circ \bfPsi_{\zeta_2-\zeta_1}\big)\cdot\big(\bu_1- \underline\bu_2+ \mathrm{Bog}_{\zeta_1}(\Div\underline\bu_2)\big)\dx\ds\nonumber\\
&\quad+\int_0^t\int_{\Omega_{\zeta_1( s )}}\nabx \underline{\bu}_2  \underline{\bu}_2\cdot\big(\bu_1- \underline\bu_2+ \mathrm{Bog}_{\zeta_1}(\Div\underline\bu_2)\big)\dx\ds\nonumber
\\
&\quad+\tfrac{1}{2} \int_0^t\int_{\partial\Omega_{\zeta_1( s )}}\bfn\circ\bfvarphi^{-1}_ {\zeta_1}\cdot\partial_t\eta_1 \bn_{\zeta_1 }\circ\bfvarphi^{-1}_ {\zeta_1}|\underline\bu_2 |^2 \dd \mathcal H^2\ds\nonumber\\
&\quad- \int_0^t\int_{\partial\Omega_{\zeta_1( s )}}\bfn\circ\bfvarphi^{-1}_ {\zeta_1}\cdot\partial_t\eta_2 \bn_{\zeta_1 }\circ\bfvarphi^{-1}_ {\zeta_1}| \bu_1 |^2 \dd \mathcal H^2\ds\nonumber
	\\
&\quad
+\int_0^t\int_{\Omega_{\zeta_1( s )}}\big( 
\mathbf B_{\zeta_2-\zeta_1}- \mathbb I
_{3\times 3}\big)\nabx \underline{\bu}_2  \underline{\bu}_2\cdot\big(\bu_1- \underline\bu_2+ \mathrm{Bog}_{\zeta_1}(\Div\underline\bu_2)\big)\dx\ds\nonumber\\
&\quad+
\int_0^t\int_{\Omega_{\zeta_1( s )}}(\bu_1-\underline\bu_2) \cdot \partial_t\mathrm{Bog}_{\zeta_1}(\Div\underline\bu_2)\dx\ds-
\int_{\Omega_{\zeta_1(t)}}(\bu_1-\underline\bu_2) \cdot \mathrm{Bog}_{\zeta_1}(\Div\underline\bu_2)\dx\nonumber\\
&\quad+
\int_{\Omega_{\zeta_1(0)}}(\bu_1-\underline\bu_2)(0) \cdot \mathrm{Bog}_{\zeta_1(0)}(\Div\underline\bu_2(0))\dx
-\int_0^t\int_{\Omega_{\zeta_1( s )}}\nabla(\bu_1- \underline\bu_2):\nabla\mathrm{Bog}_{\zeta_1}(\Div\underline\bu_2)\dx\ds\nonumber
	\\
&\quad+\int_0^t\int_{\Omega_{\zeta_1( s )}}\big(\mathbf A_{\zeta_2-\zeta_1}-\mathbb I_{3\times 3}\big)\nabla\underline\bu_2:\nabla(\bu_1-\underline \bu_2+\mathrm{Bog}_{\zeta_1}(\Div\underline\bu_2))\dx\ds\nonumber\\
&\quad+\int_0^t\int_{\Omega_{\zeta_1( s )}}\big(\mathbb I_{3\times 3}-\mathbf B_{\zeta_2-\zeta_1}\big)\underline \pi_2:\nabla(\bu_1-\underline \bu_2+\mathrm{Bog}_{\zeta_1}(\Div\underline\bu_2))\dx\ds\nonumber\\
&\quad+\int_0^t \int_\omega ( g_1-g_2)\partial_t(\eta_1-\eta_2)\dy\ds
+\int_0^t \int_{\Omega_{\zeta_1( s )}}  (\bu_1-\underline \bu_2)\cdot (\mathbf{f}_1-\underline{\mathbf f}_2)\dx\ds\nonumber\\
&\quad+\int_0^t \int_{\Omega_{\zeta_1( s )}} (\mathbb I_{3\times 3}-J_{ \zeta_2-\zeta_1})\underline{\mathbf f}_2\cdot (\bu_1-\underline \bu_2)\dx\ds+\int_0^t\int_{\Omega_{\zeta_1}}(\mathbb I_{3\times 3}-J_{\zeta_2-\zeta_1})\underline{\mathbf f}_2\mathrm{Bog}_{\zeta_1}(\Div \underline \bu_2)\dx\ds\nonumber\\
&\quad+\int_0^t\int_{\Omega_{\zeta_1( s )}}\bu_1\otimes \bu_1:\nabla\mathrm{Bog}_{\zeta_1}(\Div\underline \bu_2)\dx\ds+\int_0^t\int_{\Omega_{\zeta_1( s )}}\bu_1\cdot\nabla\bu_1\cdot\underline \bu_2\dx\ds\nonumber\\
&\quad+\int_0^t\int_{\Omega_{\zeta_1( s )}}(\mathbf f_1-\underline{\mathbf f}_2)\mathrm{Bog}_{\zeta_1}(\Div\underline \bu_2)\dx\ds\nonumber\\
&=:\sum_{i=1}^{22} R_{i}.\nonumber
\end{align}
Except for $R_{5}$, $R_{10}$ and $R_{11}$, all terms can be estimated as in \cite[Section 5]{BMSS} and are bounded by
\begin{align*}
& \delta\int_0^t\left(\lVert\nabla(\bu_1-\underline \bu_2)\rVert_{L^{2}(\Omega_{\zeta_1(s)})}^2\right)\ds+C(\delta)\int_0^t\|\partial_t\eta_1-\partial_t\eta_2\|_{L^2(\omega)}^2\ds\\
&+C(\delta)\int_0^t\left(\lVert g_1-g_2\rVert_{L^2(\omega)}^2+\lVert \mathbf f_1-\underline {\mathbf f}_2\rVert_{L^2(\Omega_{\zeta_1(s)})}^2\right)\ds
\\
&+C(\delta)\lVert \bu_1(0)-\underline \bu_2(0)\rVert_{L^2(\Omega_{\zeta_1(0)})}^2
\\
&+C(\delta) \int_0^t\Big(\lVert\underline {\mathbf f}_2\rVert_{L^2(\Omega_{\zeta_1(s)})}^2+\lVert\partial_t\underline \bu_2\rVert_{L^2(\Omega_{\zeta_1(s)})}^2+\lVert\underline \bu_2\rVert_{W^{2,2}(\Omega_{\zeta_1(s)})}^2+\lVert\underline \pi_2\rVert_{W^{1,2}(\Omega_{\zeta_1(s)})}^2\Big)\lVert\zeta_1-\zeta_2\rVert_{W^{2,2}(\omega)}^2\ds
\\
&+C(\delta) \int_0^t\Big(1+\lVert\underline {\mathbf f}_2\rVert_{L^2(\Omega_{\zeta_1(s)})}^2+\lVert\partial_t\underline \bu_2\rVert_{L^2(\Omega_{\zeta_1(s)})}^2+\lVert\underline \bu_2\rVert_{W^{2,2}(\Omega_{\zeta_1(s)})}^2\Big)\lVert\partial_t(\zeta_1-\zeta_2)\rVert_{L^{2}(\omega)}^2\ds
\\
&+C(\delta)\int_0^t\left(1+\lVert\underline \bu_2\rVert_{W^{2,2}(\Omega_{\zeta_1(s)})}^2+\lVert\partial_t\eta_2\rVert_{W^{2,2}(\omega)}^2\right)\lVert\bu_1-\underline\bu_2\rVert_{L^2(\Omega_{\zeta_1(s)})}^2\ds,
\end{align*}
which is bounded by (taking $g_1=g_2$ and $\mathbf f_1=\underline{\mathbf f}_2$ and $\underline\bu_2(0)=\bu_1(0)\circ\bfPsi_{\zeta_2(0)-\zeta_1(0)}$)
\begin{align*}
& \delta\int_0^t\left(\lVert\nabla(\bu_1-\underline \bu_2)\rVert_{L^{2}(\Omega_{\zeta_1(s)})}^2\right)\ds+C(\delta)\int_0^t\|\partial_t\eta_1-\partial_t\eta_2\|_{L^2(\omega)}^2\ds
\\
&\quad +C(\delta,R) \int_0^t\Big(\lVert\zeta_1-\zeta_2\rVert_{W^{2,2}(\omega)}^2+\lVert\partial_t(\zeta_1-\zeta_2)\rVert_{L^{2}(\omega)}^2+\lVert\bu_1-\underline\bu_2\rVert_{L^2(\Omega_{\zeta_1(s)})}^2\Big)\ds
\\
&\leq \delta\int_0^t\lVert\nabla(\bu_1-\underline \bu_2)\rVert_{L^{2}(\Omega_{\zeta_1(s)})}^2\ds
+C(\delta) T_{\tt max}\sup_t\|\partial_t\eta_1-\partial_t\eta_2\|_{L^2(\omega)}^2\\
&\quad +C(\delta,R)T_{\tt max} \sup_t\Big(\lVert\zeta_1-\zeta_2\rVert_{W^{2,2}(\omega)}^2+\lVert\partial_t(\zeta_1-\zeta_2)\rVert_{L^{2}(\omega)}^2+\lVert\bu_1-\underline\bu_2\rVert_{L^2(\Omega_{\zeta_1(s)})}^2\Big).
\end{align*}
On the other hand we have by continuity of $\mathrm{Bog}_{\zeta_1}$ that
\begin{align*}
R_5 & \lesssim \int_0^t\lVert\nabla\underline \bu_2\rVert_{L^4(\Omega_{\zeta_1(s)})}\lVert\partial_t\bfPsi_{\zeta_2-\zeta_1}^{-1}\circ\bfPsi_{\zeta_2-\zeta_1}\rVert_{L^2(\Omega_{\zeta_1(s)})}\lVert\bu_1-\underline \bu_2\rVert_{L^4(\Omega_{\zeta_1(s)})}\ds 
\\
& \lesssim \int_0^t\lVert \underline \bu_2\rVert_{W^{2,2}(\Omega_{\zeta_1(s)})}\lVert\partial_t(\zeta_1-\zeta_2)\rVert_{L^2(\omega)}\lVert\bu_1-\underline \bu_2\rVert_{L^4(\Omega_{\zeta_1(s)})}\ds
\\
& \leq\delta\int_0^t\lVert\bu_1-\underline \bu_2\rVert_{W^{1,2}(\Omega_{\zeta_1(s)})}^2\ds+C(\delta)\int_0^t\lVert\underline \bu_2\rVert_{W^{2,2}(\Omega_{\zeta_1(s)})}^2\lVert\partial_t(\zeta_1-\zeta_2)\rVert_{L^2(\omega)}^2\ds.
\end{align*}
We rewrite $\Div\underline\bu_2$ as
$$\Div\underline\bu_2=\mathbb{I}_{3\times 3}:\nabla\underline\bu_2=(\mathbb{I}_{3\times 3}-\bfB_{\zeta_2-\zeta_1}):\nabla\underline\bu_2,$$
and thereby have 
\begin{align*}
R_{10}&=\int_0^t\int_{\Omega_{\zeta_1( s )}}(\bu_1-\underline \bu_2)\cdot\mathrm{Bog}_{\zeta_1}\big((\mathbb I_{3\times 3}-\mathbf B_{\zeta_2-\zeta_1}):\partial_t\nabla\underline\bfv_2\big)\dx\ds\\
&\quad -\int_0^t\int_{\Omega_{\zeta_1( s )}}(\bu_1-\underline \bu_2)\cdot\mathrm{Bog}_{\zeta_1}\big(\partial_t\mathbf B_{\zeta_2-\zeta_1}:\nabla\underline{\bfv}_2\big)\dx\ds\\
&=:R_{10}^1+R_{10}^2.
\end{align*}
By \eqref{210and212} and continuity of $\mathrm{Bog}$, we see that
\begin{align*}
R_{10}^1&\leq 
\int_0^t\lVert\bu_1-\underline \bu_2\rVert_{L^6(\Omega_{\zeta_1( s )})}\lVert(\mathbb I_{3\times 3}-\mathbf B_{\zeta_2-\zeta_1})\partial_t\nabla \underline \bu_2\rVert_{W^{-1,\frac{6}{5}}(\Omega_{\zeta_1( s )})}\ds\\
&\lesssim \int_0^t\lVert\bu_1-\underline \bu_2\rVert_{W^{1,2}(\Omega_{\zeta_1( s )})}\lVert\mathbb I_{3\times 3}-\mathbf B_{\zeta_2-\zeta_1}\rVert_{W^{1,2}(\Omega_{\zeta_1( s ))}}\lVert\partial_t\underline \bu_2\rVert_{L^2(\Omega_{\zeta_1( s )})}\ds\\
&\leq \delta \int_0^t\lVert \bu_1-\underline\bu_2\rVert_{W^{1,2}(\Omega_{\zeta_1( s )})}^2\ds+C(\delta)\int_0^t\lVert\partial_t\underline \bu_2\rVert_{L^2(\Omega_{\zeta_1( s )})}^2\lVert\zeta_1-\zeta_2\rVert_{W^{2,2}(\omega)}^2\ds,
\end{align*}
as well as (using the Piola identity, the definition of $\mathbf B_{\zeta_2-\zeta_1}$ and \eqref{eq:detPsi})
\begin{align*}
R_{10}^2&\leq 
\int_0^t\lVert\bu_1-\underline \bu_2\rVert_{L^6(\Omega_{\zeta_1( s )})}\lVert \partial_t\mathbf B_{\zeta_2-\zeta_1}:\nabla\underline{\bfv}_2\rVert_{W^{-1,\frac{6}{5}}(\Omega_{\zeta_1( s )})}\ds\\
&\leq 
\int_0^t\lVert\bu_1-\underline \bu_2\rVert_{W^{1,2}(\Omega_{\zeta_1( s )})}\|\partial_t(\zeta_2-\zeta_1)\|_{L^2(\omega)}\|\nabla\underline \bu_2\rVert_{W^{1,2}(\Omega_{\zeta_1( s )})}\ds\\
&\leq \delta \int_0^t\lVert \bu_1-\underline\bu_2\rVert_{W^{1,2}(\Omega_{\zeta_1( s )})}^2\ds+C(\delta)\int_0^t\lVert\underline \bu_2\rVert_{W^{2,2}(\Omega_{\zeta_1( s )})}^2\lVert\partial_t(\zeta_1-\zeta_2)\rVert_{L^2(\omega)}^2\ds.
\end{align*}
Note that we also took into account that $\zeta_1$ and $\zeta_2$ belong to the class $L^\infty(I;C^{0,1}(\omega))$.

Finally, we estimate 
\begin{align*}
R_{11}
&\leq \lVert\bu_1-\underline\bu_2\rVert_{L^2(\Omega_{\zeta_1(s)})}\lVert\mathrm{Bog}_{\zeta_1}\left((\mathbb I_{3\times 3}-\mathbf B_{\zeta_2-\zeta_1}):\nabla\underline\bu_2\right)\rVert_{L^2(\Omega_{\zeta_1(s)})}
\\
&\lesssim \lVert\bu_1-\underline\bu_2\rVert_{L^2(\Omega_{\zeta_1(s)})}\lVert\mathrm{Bog}_{\zeta_1}\left((\mathbb I_{3\times 3}-\mathbf B_{\zeta_2-\zeta_1}):\nabla\underline\bu_2\right)\rVert_{W^{1,6/5}(\Omega_{\zeta_1(s)})}
\\
&\leq \delta\lVert\bu_1-\underline\bu_2\rVert_{ L^2(\Omega_{\zeta_1(s)})}^2+C(\delta)\lVert(\mathbb I_{3\times 3}-\mathbf B_{\zeta_2-\zeta_1}):\nabla\underline\bu_2\rVert_{ L^{\frac{6}{5}}(\Omega_{\zeta_1(s)})}^2
\\
&\leq \delta\lVert\bu_1-\underline\bu_2\rVert_{ L^2(\Omega_{\zeta_1(s)})}^2+C(\delta)\lVert\nabla\underline\bu_2\rVert_{ L^2(\Omega_{\zeta_1(s)})}^2\lVert\nabla_{\by}(\zeta_1-\zeta_2)\rVert_{L^3(\omega)}^2
\\
&\leq \delta\lVert\bu_1-\underline\bu_2\rVert_{ L^2(\Omega_{\zeta_1(s)})}^2+C(\delta)\lVert\zeta_1-\zeta_2\rVert_{W^{2,2}(\omega)}^{\frac{4}{3}}\lVert\zeta_1-\zeta_2\rVert_{L^{2}(\omega)}^{\frac{2}{3}}
\\
&\leq \delta\lVert\bu_1-\underline\bu_2\rVert_{ L^2(\Omega_{\zeta_1(s)})}^2+\nu\lVert\zeta_1-\zeta_2\rVert_{W^{2,2}(\omega)}^2+C(\delta,\nu)\lVert\zeta_1-\zeta_2\rVert_{L^{2}(\omega)}^2
\\
&\leq \delta\lVert\bu_1-\underline\bu_2\rVert_{ L^2(\Omega_{\zeta_1(s)})}^2+\nu\lVert\zeta_1-\zeta_2\rVert_{W^{2,2}(\omega)}^2+C(\delta,\nu)\int_0^t\lVert\partial_t(\zeta_1-\zeta_2)\rVert_{L^{2}(\omega)}^2\ds.
\end{align*}

Combining everything yields (choosing $T_{\tt max},\delta,\nu$ small enough)
\begin{align}
&\tfrac{1}{2}\sup_{0\leq t\leq T_{\tt max}}\int_{\Omega_{\zeta_1(t)}}|\bu_1(t)-\underline{\bu}_2(t)|^2\dx+\int_0^{T_{\tt max}}
\int_{\Omega_{\zeta_1(s)}}|\nabla(\bu_1-\underline{\bu}_2)|^2\dx\ds\nonumber\\
&\quad+\sup_{0\leq t\leq T_{\tt max}}\bigg(\tfrac{1}{2}\int_\omega|\partial_t(\eta_1-\eta_2)(t)|^2\dy+\tfrac{1}{2}\int_\omega|\Dely(\eta_1-\eta_2)(t)|^2\dy\bigg)\nonumber\\
&\leq \tfrac{1}{2}\sup_{0\leq t\leq T_{\tt max}}\bigg(\tfrac{1}{2}\int_\omega|\partial_t(\zeta_1-\zeta_2)(t)|^2\dy+\tfrac{1}{2}\int_\omega|\Dely(\zeta_1-\zeta_2)(t)|^2\dy\bigg),\nonumber
\end{align}
which gives the claim.
\end{proof}

\subsection{Conclusion}\label{sec:3.4}
\begin{proof}[Proof of Theorem \ref{thm:NS3}]
As a consequence of Proposition \ref{prop:contraction} the weak solution obtained in Proposition \ref{thm:transformedSystemLinear} is unique. Hence the mapping $\zeta\mapsto \eta$ is well-defined. By Corollary \ref{cor:self},
it maps the convex set $$B^{ X_{\tt shell}(0,T_{\tt max})}_R(0)\cap \{\zeta\in X_{\tt shell}(0,T_{\tt max}):\,\zeta(0)=\eta_0\}$$
into itself. Again, by Proposition \ref{prop:contraction} the mapping $\zeta\mapsto \eta$ is a contraction with respect to the topology from $$W^{1,\infty}(0,T_{\tt max};L^2(\omega))\cap L^{\infty}(0,T_{\tt max};W^{2,2}(\omega)).$$
By Banach's fixed-point theorem, we obtain the existence of the solution for the system \eqref{contEqAloneBar}--\eqref{momEqAloneBar}.
\end{proof}

\section{2D Navier--Stokes equations}
In this section we deal with the two-dimensional case. The claim of Theorem \ref{thm:NS2} follows from the following proposition.
\begin{proposition}
\label{thm:transformedSystemLinea2D}
Suppose that the dataset
$(\overline{\bff},g, \eta_0, \eta_*, \overline{\bu}_0)$
satisfies \eqref{dataset'}.
Then there exists a strong solution $( \eta, \overline{\bu},  \overline{\pi} )$ of \eqref{contEqAloneBar}--\eqref{momEqAloneBar} in 2D such that
\begin{equation}
\begin{aligned}
\label{energyEstLinear}
&\sup_I\int_\omega
\big(\vert \partial_t^2\eta\vert^2 
+
\vert \partial_t\Dely \eta\vert^2
\big)
\dy+\int_I
\big(\|\eta\|_{W^{3,2}(\omega)}^2
+\|\partial_t\Dely\eta\|_{W^{1/4,2}(\omega)}^2
\big)
\dt
\\&\quad +
\sup_I\int_\Omega\big(\vert\partial_t \overline{\bu}\vert^2+\vert\nabla\overline{\bu}\vert^2\big)\dx
+
\int_I\int_\Omega\vert \partial_t\nabla\overline{\bu} \vert^2 \dx\dt
 \\&\leq
c\Big(\|\eta_0\|_{W^{2,2}_\by},\|\eta_*\|_{W^{2,2}_\by},\|\partial_t^2\eta(0)\|_{L^2_\by},\|\partial_t\overline\bu(0)\|_{L^2_\bx},\|\overline\bu_0\|_{L^2_\bx},\| \overline{\bff} \|_{W^{1,2}(I;L^2_\bx)},\|g\|_{W^{1,2}(I;L^2_\by)}\Big).
\end{aligned}
\end{equation}
\end{proposition}
First observe that the pair of test functions $(\partial_t \eta, \overline{\bu})$ for
the fluid-structure system and transform the energy estimate to the reference domain. Since $\naby\eta$ is bounded in space-time
we can use the ellipticity of $\mathbf{A}_{\eta_0} $ (notation as in Section \ref{sec:loc}) to transform to the reference domain obtaining
\begin{equation*}
\begin{aligned}
&\sup_I
\int_\omega\big(\vert \partial_t\eta\vert^2
+
\vert  \Dely \eta\vert^2 \big)\dy
 +
 \sup_I
 \int_\Omega\vert  \overline{\bu}\vert^2\dx
+
 \int_I\int_\Omega\vert  \nabx\overline{\bu}\vert^2\dx\dt
\\
&\lesssim
\int_\omega\big(\vert \eta_*\vert^2
+
\vert \Dely \eta_0\vert^2 \big)\dy
 +
 \int_I\int_\omega|g|^2\dy\dt
  +
   \int_\Omega\vert \overline{\bu}_0\vert^2\dx
+
 \int_I\int_\Omega\vert  \overline{\mathbf{f}}\vert^2 
 \dx\dt.
\end{aligned}
\end{equation*}
Moreover, based on Theorem \ref{thm:NS3}, which also hold for $n=2$ by constant elongation. The aim is to obtain uniform estimates that hold uniformly in time. The first step is the following lemma
\begin{lemma}\label{lem:4.2}
We have $\eta\in L^2(W^{3,2})$, in particular, and the estimate
\begin{align}\label{eq:MuSc'2D}
\begin{aligned}
\int_I\|\eta\|^2_{W^{3,2}_\by}\dt
&\lesssim 1+\int_I\big(\|\partial_t^2\eta\|_{L^2_\by}^2+\|\partial_t \overline\bu\|^2_{L^2_\bx}+\|\nabla \overline\bu\|_{L^2_\bx}^3+\|\overline\bff\|_{L^2_\bx}^2+\|g\|_{L^2_\bx}^2\big)\dt\\
&\quad +\int_I\left(1+\|\nabla\overline\bu\|_{L^2_\bx}^2\right)\|\partial_t\Dely\eta\|^2_{L^{2}_\by}\dt.
\end{aligned}
\end{align}
\end{lemma}
\begin{proof}
For that purpose consider the test function pair for the fluid-structure system \eqref{contEqAloneBar}--\eqref{momEqAloneBar} as follows:
\begin{align*}   
(\overline{\bm{\varphi}}^\eta,\phi^\eta)=\left(\Delta_{-h}\Delta_h \partial_t\eta-\mathscr K_{{\eta}}(\Delta_{-h}\Delta_h \partial_t\eta),\, \mathscr F_{\eta}^{\mathrm{div}}(\Delta_{-h}\Delta_h \partial_t\eta-\mathscr K_{{\eta}}(\Delta_{-h}\Delta_h \partial_t\eta))\circ\bfPsi_{\eta}\right).
\end{align*}
where $\mathscr F_{\eta}^{\mathrm{div}}$ is the corresponding solenoidal extension introduced in Lemma \ref{prop:musc} and for $h>0$, $$\Delta_hv(\by):=h^{-1}(v(\by+h\bm{e}_\alpha)-v(\by))$$
is the difference quotient for space variable in direction $\bm{e}_\alpha$ with $\alpha\in\{1,2\}$.
Similarly to \eqref{mom:reg}, we have
\begin{align}
&\int_I \int_{\omega} |\Delta_h\Dely\partial_t\eta|^2\dy\dt\nonumber\\
&=\int_I\int_\Omega J_{\eta}\partial_t\overline{\bv}\cdot\partial_t\overline{\bm{\varphi}}^\eta\dx\dt -\int_I\int_\Omega\partial_t\bfA_{\eta}\nabla\overline{\bv}:\nabla\overline{\bm{\varphi}} ^\eta\dx\dt-\int_I\int_\Omega\bfA_{\eta}\nabla\partial_t\overline{\bv}:\nabla\overline{\bm{\varphi}}^\eta\dx\dt\nonumber\\
&\quad -\int_I\frac{\dd}{\dt}\left(\int_\Omega J_{\eta}\partial_t\overline{\bv}\cdot\overline{\bm{\varphi}}^\eta\dx+\int_\omega\partial_t^2\eta\cdot\phi^\eta\dy\right)\dt+\int_I\int_\omega(\partial_t^2\eta\, \partial_t\phi^\eta+\partial_t g \,\phi^\eta)\dy\dt\nonumber\\
&\quad +\int_I\int_\Omega\pi\Div(\partial_t\bfB_{\eta}^\intercal \overline{\bm{\varphi}}^\eta)\dx\dt+\int_I\int_\Omega\partial_t\bfh_{\eta}\cdot \overline{\bm{\varphi}}^\eta\dx\dt 
\nonumber\\
&\quad+\int_I\int_\Omega\left(-\partial_t J_{\eta} \overline \bu\cdot \partial_t\overline{\bm{\varphi}}+\frac{\dd}{\dt}\left(\partial_tJ_{\eta}\overline \bu\cdot \overline{\bm{\varphi}}^\eta\right)-\partial_t(\partial_t J_{\eta}\overline \bu)\cdot \overline{\bm{\varphi}}^\eta\right)\dx\dt \nonumber\\
&=:(I)^\eta+...+(VII)^\eta+(VIII)^\eta.\label{mom:reg2D}
\end{align}
We first estimate $(I)^\eta$  as follows:
\begin{align}
(I)^\eta&\lesssim \int_I\|\partial_t\overline\bu\|_{L^2_\bx}\|\partial_t\overline\bfvarphi^\eta\|_{L^2_\bx}\dt\nonumber\\
&\lesssim \int_I\|\partial_t\overline\bu\|_{L^2_\bx}\left(\|\partial_t\Delta_{-h}\Delta_h\partial_t\eta\|_{L^2_\by}+\|\Delta_{-h}\Delta_h\partial_t\eta\,\partial_t\eta\|_{L^2_\by}\right)\dt\nonumber\\
&\lesssim \int_I\|\partial_t\overline\bu\|_{L^2_\bx}\left(\|\partial_t^2\eta\|_{W^{2s,2}_\by}+\|\Delta_{-h}\Delta_h\partial_t\eta\|_{L^2_\by}\|\partial_t\eta\|_{L^\infty_\by}\right)\dt\nonumber\\
&\lesssim \int_I\|\partial_t\overline\bu\|_{L^2_\bx}(\|\partial_t\overline\bu\|_{W^{1,2}_\bx}+\|\overline\bu\|_{W^{1,2}_\bx}\|\partial_t\naby\eta\|_{L^2_\by})\dt\nonumber\\
&\lesssim \delta\int_I\|\partial_t\overline\bu\|_{W^{1,2}_\bx}^2\dt+c(\delta)\int_I\|\partial_t\overline\bu\|_{L^2_\bx}^2\dt+\int_I\|\overline\bu\|_{W^{1,2}_\bx}^2\|\partial_t\naby\eta\|_{L^2_\by}^2\dt.\label{later2}
\end{align}
Note that $\nabla\eta\in L^\infty_{t,\by}$, then we have $|\partial_t\bfA_\eta|\lesssim  |\partial_t\nabla\eta|+|\partial_t\eta|$.
Using Lemma \ref{prop:musc} and \eqref{218} we estimate further 
\begin{align}
(II)^\eta&\lesssim \int_I\|\partial_t\naby\eta\|_{L^4_\by}\|\nabla\overline\bu\|_{L^2_\bx}\|\nabla\overline\bfvarphi^{\eta}\|_{L^4_\bx}\dx+\int_I\|\nabla\overline\bu\|_{L^2_\bx}\|\partial_t\eta\|_{L^4_\by}\|\nabla\overline\bfvarphi^\eta\|_{L^4_\bx}\dt\nonumber\\
&\lesssim \int_I\|\partial_t\naby\eta\|_{W^{1,2}_\by}\|\nabla\overline\bu\|_{L^2_\bx}\|\partial_t\eta\|_{W^{1+2s,4}_\by}\dt+\int_I\|\overline\bu\|_{W^{1,2}_\bx}\|\partial_t\eta\|_{W^{2,2}_\by}^{\frac{1}{8}}\|\partial_t\eta\|_{W^{2,2}_\by}^{\frac{5}{8}+s}\dt\nonumber\\
&\lesssim \int_I\Big(1+\|\overline\bu\|_{W^{1,2}_\bx}^2\Big)(1+\|\partial_t\eta\|_{W^{2,2}_\by}^2)\dt,\label{later1}
\end{align}
where we used $\partial_t\eta\in L^\infty_t(L^2_\by)$ and 
$$\|\partial_t\eta\|_{L^4_\by}\lesssim \|\partial_t\eta\|_{L^2_\by}^{\frac{7}{8}}\|\partial_t\eta\|_{W^{2,2}_\by}^{\frac{1}{8}}, $$
$$\|\partial_t\eta\|_{W^{1+2s,4}_\by}\lesssim \|\partial_t\eta\|_{L^2_\by}^{\frac{3}{8}-s}\|\partial_t\eta\|_{W^{2,2}_\by}^{\frac{5}{8}+s}.$$
As before (but now using that $\naby\eta\in L^\infty_{t,\by}$ by the 1D Sobolev embedding),
\begin{align*}
(III)^\eta
&\lesssim \int_I\Big(\|\partial_t\overline\bu\|_{W^{1,2}_\bx}^2+\|\partial_t\Dely\eta\|_{L^{2}_\by}^2\Big)\dt,\\
(IV)^\eta
&\lesssim\sup_I\Big(\|\partial_t\bu\|_{L^2_\bx}^2+\|\partial_t^2\eta\|_{L^2_\by}^2+\|\partial_t\Dely\eta\|_{L^{2}_\by}^2\Big),\\
(V)^\eta
&\lesssim \int_I\Big(\|\partial_t\bu\|_{W^{1,2}_\bx}^2+\|\partial_tg\|_{L^2_\by}^2+\|\partial_t\Dely\eta\|_{L^{2}_\by}^2\Big)\dt.
\end{align*}
We now treat $(VII)^\eta$ and $(VIII)^\eta$ before dealing with the more subtle $(VI)^\eta$.
Similar to how $(VII)^\eta$ was treated in  \eqref{mom:reg}, we can again decompose $(VII)^\eta$ into the sum of $(VII)^{\eta}_1,\dots, (VII)^{\eta}_3$, where
\begin{align*}
(VIII)^\eta&:=- \int_I\int_\Omega\partial_t J_\eta \overline\bu\cdot \partial_t\overline{\bm{\varphi}}^\eta\dx\dt+\int_\Omega\partial_tJ_{\eta} \overline{\bu}\cdot \overline{\bm{\varphi}}^\eta\dx \bigg|_0^T+\int_I\int_\Omega\partial_t(\partial_t J_{\eta}\overline{\bu})\cdot \overline{\bm{\varphi}}^\eta\dx\dt,\\
(VII)^\eta_1&:= \int_I\int_\Omega
J_\eta\nabx\overline{\bu}\cdot \partial_t \bfPsi_{\eta}^{-1}\circ \bfPsi_{\eta}\cdot \partial_t\overline{\bm{\varphi}}^\eta\dx\dt-\int_I\int_\Omega
\partial_t\Big(J_{\eta} \nabx\overline{\bu}\cdot \partial_t \bfPsi_{\eta}^{-1}\circ \bfPsi_{\eta}\cdot\overline{\bm{\varphi}}^\eta\Big)\dx\dt,
\\(VII)^\eta_2&:=- \int_I\int_\Omega\partial_t\big(\mathbf{B}_{\eta}\nabx\overline{\bfv}~\overline{\bu}\big)\cdot\overline{\bm{\varphi}}^\eta\dx\dt,\\
(VII)^\eta_3&:= \int_I\int_\Omega\partial_t\big(
J_\eta\bff\circ \bfPsi_{\eta}\big)\cdot\overline{\bm{\varphi}}^\eta\dx\dt.
\end{align*}
Arguing as before (and using again $\naby\eta\in L^{\infty}_{t,\by}$) we have
\begin{align*}
(VII)^{\eta}_3&\lesssim \int_I\Big(\|\partial_t\overline{\bff}\|_{L^2_\bx}^2+\|\overline{\bff}\|_{L^2_\bx}^2\Big)\dt +\int_I(1+\|\overline\bu\|_{W^{1,2}_\bx}^2)\|\partial_t\Dely\eta\|_{L^2_\by}^{2}\dt.
\end{align*}
Arguing as in \eqref{later2} we obtain
\begin{align*}
(VIII)^\eta&\lesssim \int_I\|\partial_t\eta\|_{L^\infty_\by}\|\overline \bu\|_{L^2_\bx}\|\partial_t\overline\bfvarphi^\eta\|_{L^2_\bx}\dt+\sup_I\|\partial_t\eta\|_{L^4_\by}\|\overline\bu\|_{L^2_\bx}\|\overline\bfvarphi^\eta\|_{L^4_\bx}\\
&\quad +\int_I\|\partial_t^2\eta\|_{L^2_\by}\|\overline \bu\|_{L^4_\bx}\|\overline\bfvarphi^\eta\|_{L^4_\bx}\dt+\int_I\|\partial_t\eta\|_{L^\infty_\by}\|\partial_t\overline\bu\|_{L^2_\bx}\|\overline\bfvarphi^\eta\|_{L^2_\bx}\dt\\
&\lesssim \int_I\|\partial_t\eta\|_{W^{2,2}_\by}\|\partial_t\overline\bu\|_{W^{1,2}_\bx}+\sup_I\|\partial_t\eta\|_{W^{2,2}_\by}\|\partial_t\eta\|_{W^{2s,4}_\by}\\
&\quad+\int_I\|\partial_t^2\eta\|_{L^2_\by}\|\overline\bu\|_{W^{1,2}_\bx}\|\partial_t\eta\|_{W^{2s,4}_\by}\dt +\int_I\|\partial_t\eta\|_{W^{2,2}_\by}\|\partial_t\overline\bu\|_{L^2_\bx}\|\partial_t\eta\|_{W^{2s,2}_\by}\dt\\
&\lesssim \int_I\|\partial_t\overline\bu\|_{W^{1,2}_\bx}^2\dt+\int_I\left(\|\partial_t\eta\|_{W^{2,2}_\by}^2+\|\partial_t^2\eta\|_{L^2_\by}^2\right)\dt+\sup_I\|\partial_t\eta\|_{W^{2,2}_\by}^2\\
&\quad +\int_I\|\overline\bu\|_{W^{1,2}_\bx}^2(\|\partial_t\eta\|_{W^{2,2}_\by}^2+\|\partial_t\overline\bu\|_{L^2_\bx}^2)\dt.
\end{align*}
Now we estimate the remaining terms in \eqref{mom:reg2D}:
\begin{align*}
(VII)^\eta_1&\lesssim \int_I\|\nabla\overline\bu\|_{L^2_\bx}\|\partial_t\eta\|_{L^\infty_\by}\|\partial_t\overline\bfvarphi^\eta\|_{L^2_\bx}\dt+\sup_I\|\nabla\overline\bu\|_{L^2_\bx}\|\partial_t\eta\|_{L^\infty_\by}\|\overline\bfvarphi^\eta\|_{L^2_\bx}\dt
\\
&\lesssim \int_I\|\nabla\overline\bu\|_{L^2_\bx}\|\partial_t\eta\|_{L^2_\by}^{\frac{1}{2}}\|\partial_t\naby\eta\|_{L^2_\by}^{\frac{1}{2}}(\|\partial_t^2\Delta_{-h}^s\Delta_h^s\eta\|_{L^2_\by}+\|\partial_t\Delta_{-h}^s\Delta_h^s\eta\,\partial_t\eta\|_{L^2_\by})\dt
\\
&\quad +\sup_I\|\nabla\overline\bu\|_{L^2_\bx}\|\partial_t\eta\|_{L^2_\by}^{\frac{3}{4}}\|\partial_t\Dely\eta\|_{L^2_\by}^{\frac{1}{4}}\|\partial_t\eta\|_{L^2_\by}^{1-s}\|\partial_t\Dely\eta\|_{L^2_\by}^s\dt
\\
&\lesssim \int_I\|\nabla\overline\bu\|_{L^2_\bx}\left(\|\partial_t\naby\eta\|_{L^2_\by}^{\frac{1}{2}}\|\partial_t\overline\bu\|_{W^{1,2}_\bx}+\|\overline\bu\|_{W^{1,2}_\bx}\|\partial_t\naby\eta\|_{L^2_\by}^{\frac{3}{2}}\right)\dt\\&\quad +\sup_I\|\nabla\overline\bu\|_{L^2_\bx}\|\partial_t\eta\|_{L^2_\by}^{\frac{7}{4}-s}\|\partial_t\Dely\eta\|_{L^2_\by}^{\frac{1}{4}+s}\dt
\\
&\lesssim \delta\int_I\|\partial_t\overline\bu\|_{W^{1,2}_\bx}^2\dt+c(\delta)\int_I(\|\partial_t\Dely\eta\|_{L^2_\by}^2+1)\|\nabla\overline\bu\|_{L^2_\bx}^2\dt\\&\quad +\delta\sup_I\|\nabla\overline\bu\|_{L^2_\bx}^2+\varepsilon\sup_I\|\partial_t\Dely\eta\|_{L^2_\by}^2+c(\delta, \varepsilon).
\end{align*}
Note that $|\partial_t\bfB_\eta|\lesssim |\partial_t\nabla\eta|+|\partial_t\eta|$,
thereby we have
\begin{align*}
(VII)_2^\eta&\lesssim \int_I\|\partial_t\naby\eta\|_{L^\infty_\by}\|\nabla\overline\bu\|_{L^2_\bx}\|\overline\bu\|_{L^4_\bx}\|\overline\bfvarphi^\eta\|_{L^4_\bx}\dt+\int_I\|\partial_t\eta\|_{L^\infty_\by}\|\nabla\overline\bu\|_{L^2_\bx}\|\overline\bu\|_{L^3_\bx}\|\overline\bfvarphi^\eta\|_{L^6_\bx}\dt\\
&\quad +\int_I\|\partial_t\nabla\overline\bu\|_{L^2_\bx}\|\overline\bu\|_{L^4_\bx}\|\overline\bfvarphi^\eta\|_{L^4_\bx} +\|\nabla\overline\bu\|_{L^2_\bx}\|\partial_t\overline\bu\|_{L^4_\bx}\|\overline\bfvarphi^\eta\|_{L^4_\bx}\dt
\\
&\lesssim \int_I\|\partial_t\eta\|_{L^2_\by}^{\frac{1}{4}}\|\partial_t\eta\|_{W^{2,2}_\by}^{\frac{3}{4}}\|\nabla\overline\bu\|_{L^2_\bx}^{\frac{3}{2}}\|\overline\bu\|_{L^2_\bx}^{\frac{1}{2}}\|\partial_t\eta\|_{L^2_\by}^{\frac{7}{8}-s}\|\partial_t\eta\|_{W^{2,2}_\by}^{\frac{1}{8}+s}\dt\\
&\quad+\int_I\|\partial_t\eta\|_{L^2_\by}^{\frac{3}{4}}\|\partial_t\eta\|_{W^{2,2}_\by}^{\frac{1}{4}}\|\overline\bu\|_{L^2_\bx}^{\frac{2}{3}}\|\overline\bu\|_{W^{1,2}_\bx}^{\frac{4}{3}}\|\partial_t\eta\|_{L^2_\by}^{\frac{5}{6}-s}\|\partial_t\eta\|_{W^{2,2}_\by}^{\frac{1}{6}+s}\dt\\
&\quad +\int_I\|\partial_t\overline\bu\|_{W^{1,2}_\bx}\|\nabla\overline\bu\|_{L^2_\bx}\|\partial_t\eta\|_{W^{2,2}_\by}\dt
\\
&\lesssim \delta\int_I\|\partial_t\overline\bu\|_{W^{1,2}_\bx}^2\dt+c(\delta)\int_I(\|\overline\bu\|_{W^{1,2}_\bx}^2+1)(\|\overline\bu\|_{W^{1,2}_\bx}^2+\|\partial_t\eta\|_{W^{2,2}_\by}^2)\dt,
\end{align*}
where the $1D$ interpolation inequalities:
$$ \|\partial_t\naby\eta\|_{L^\infty_\by}\lesssim \|\partial_t\naby\eta\|_{W^{-1,2}_\by}^{\frac{1}{4}}\|\partial_t\naby\eta\|_{W^{1,2}_\by}^{\frac{3}{4}},\qquad   \|\partial_t\eta\|_{L^\infty_\by}\lesssim \|\partial_t\eta\|_{L^2_\by}^{\frac{3}{4}}\|\partial_t\eta\|_{W^{2,2}_\by}^{\frac{1}{4}},$$
$$ \|\partial_t\eta\|_{W^{2s,4}_\by}\lesssim \|\partial_t\eta\|_{L^2_\by}^{\frac{7}{8}-s}\|\partial_t\eta\|_{W^{2,2}_\by}^{\frac{1}{8}+s},\qquad  \|\partial_t\eta\|_{W^{2s,6}_\by}\lesssim \|\partial_t\eta\|_{L^2_\by}^{\frac{5}{6}-s}\|\partial_t\eta\|_{W^{2,2}_\by}^{\frac{1}{6}+s},$$
and $2D$ interpolation inequalities:
$$\|\overline\bu\|_{L^4_\bx}\lesssim \|\overline\bu\|_{L^2_\bx}^{\frac{1}{2}}\|\overline\bu\|_{W^{1,2}_\bx}^{\frac{1}{2}}, \qquad \|\overline \bu\|_{L^3_\bx}\lesssim \|\overline\bu\|_{L^2_\bx}^{\frac{2}{3}}\|\overline\bu\|_{W^{1,2}_\bx}^{\frac{1}{3}}. $$

We finally return to treat $(VI)^\eta$. For this,
 we use the weak formulation of the coupled system to represent the pressure integrals involved in  \eqref{mom:reg2D} as 
%
\begin{align}
&\int_I\int_\Omega\pi\Div(\partial_t\bfB_\eta^\intercal \overline\bfvarphi^\eta)\dx\dt
=\int_I\int_\Omega\pi\Div(\bfB_\eta^\intercal\boldsymbol\Upsilon^\eta)\dx\dt\nonumber
\\
&=\int_I\int_\Omega J_{\eta}\partial_t\overline{\bu}\cdot\boldsymbol\Upsilon^\eta\dx\dt+\int_I\int_\Omega\bfA_\eta\nabla\overline{\bu}:\nabla\boldsymbol\Upsilon^\eta\dx\dt-\int_I\int_\Omega\bfh_\eta(\overline{\bu})\cdot\boldsymbol\Upsilon^\eta\dx\dt\nonumber
\\
&\quad +\int_I\int_\omega(\partial_t^2\eta+\Dely^2\eta)\bfn^\intercal \boldsymbol\Upsilon^\eta\circ\bfvarphi\dy\dt
-\int_I\int_\omega g\, \bfn^\intercal \boldsymbol{\Upsilon}^\eta\circ\bfvarphi \dy\dt,\label{for2dlater}
\end{align}
where 
$\boldsymbol\Upsilon^\eta=\bfB_\eta^{-\intercal}\partial_t\bfB_\eta^\intercal\overline{\bfvarphi}^\eta  
$
and \begin{align*}
\boldsymbol{\Upsilon}^\eta\circ\bfvarphi&=\bfB_\eta^{-\intercal}\circ\bfvarphi \partial_t\bfB_\eta^\intercal\circ\bfvarphi \Dely^{1/4}\partial_t\eta \bfn
= J_\eta ^{-1}\partial_t J_\eta \Dely^{1/4}\partial_t\eta \bfn
\\
&\approx (1+\eta)\partial_t\eta \Dely^{1/4}\partial_t\eta\bfn,
\end{align*}
by \eqref{eq:detPsi} and \eqref{eq:dnPsi}.
We derive from Lemma \ref{prop:musc} that
\begin{align*}
&\int_I\int_\Omega J_{\eta}\partial_t\overline{\bu}\cdot\boldsymbol\Upsilon^\eta\dx\dt
\lesssim \int_I(\partial_t\overline\bu\,\partial_t\eta+\partial_t\overline\bu\,\partial_t\nabla\eta)\overline\bfvarphi^\eta\dt
\\
&\lesssim \int_I(\|\partial_t\eta\|_{L^\infty_\by}+\|\partial_t\naby\eta\|_{L^\infty_\by})\|\partial_t\overline\bu\|_{L^2_\bx}\|\partial_t\Delta_{-h}^s\Delta_h^s\eta\|_{L^2_\bx}\dt \\
&\lesssim \int_I\|\partial_t\eta\|_{W^{2,2}_\by}\|\partial_t\overline\bu\|_{L^2_\bx}\|\overline\bu\|_{W^{1,2}_\bx}\dt\\
&\lesssim \int_I\Big(\|\partial_t\overline\bu\|_{L^2_\bx}^2+\|\overline\bu\|_{W^{1,2}_\bx}^2\|\partial_t\Dely\eta\|_{L^2_\by}^2\Big)\dt,
\end{align*}
Using again \eqref{matrices} and \eqref{210and212}--\eqref{218}, we see that the critical term in $|\partial_t\nabla\bfB_\eta|$ is $\partial_t\nabla^2\eta$, then
\begin{align*}
&\int_I\int_\Omega\bfA_\eta\nabla\overline{\bu}:\nabla\boldsymbol\Upsilon^\eta\dx\dt\\
&\sim \int_I\int_\Omega\nabla\overline\bu\left(\partial_t\nabla^2\eta+\nabla^2\eta\partial_t\nabla\eta\right)\overline\bfvarphi^\eta\dt +\int_I\int_\Omega\nabla\overline\bu\partial_t\nabla\eta\nabla\overline\bfvarphi^\eta\dx\dt\\
&:=I_{\bfA}+II_{\bfA}+III_{\bfA}.
\end{align*}
We see that
\begin{align*}
I_{\bfA}&\lesssim \int_I\|\nabla\overline\bu\|_{L^2_\bx}\|\partial_t\nabla^2\eta\|_{L^4_\by}\|\overline\bfvarphi^\eta\|_{L^4_\bx}\dt\\
&\lesssim \int_I\|\overline\bu\|_{W^{1,2}_\bx}\|\partial_t\nabla^2\eta\|_{W^{\frac{1}{4},2}_\by}\|\partial_t\eta\|_{W^{2,2}_\by}\dt\\
&\leq \delta\int_I\|\partial_t\nabla^2\eta\|_{W^{\frac{1}{4},2}_\by}^2\dt+c(\delta)\int_I\|\overline\bu\|_{W^{1,2}_\bx}^2\|\partial_t\eta\|_{W^{2,2}_\by}^2\dt.
\end{align*}
Moreover, we have
\begin{align*}
II_\bfA&\lesssim \int_I\|\nabla\overline\bu\|_{L^2_\bx}\|\nabla^2\eta\|_{L^4_\by}\|\partial_t\nabla\eta\|_{L^\infty_\by}\|\overline\bfvarphi^\eta\|_{L^4_\bx}\dt\\
&\lesssim \int_I\|\overline\bu\|_{W^{1,2}_\bx}\|\nabla^2\eta\|_{L^2_\by}^{\frac{3}{4}}\|\eta\|_{W^{3,2}_\by}^{\frac{1}{4}}\|\partial_t\eta\|_{W^{2,2}_\by}\|\partial_t\eta\|_{L^2_\by}^{\frac{7}{8}-s}\|\partial_t\eta\|_{W^{2,2}_\by}^{\frac{1}{8}+s}\dt\\
&\lesssim \delta\int_I\|\eta\|_{W^{3,2}_\by}^2\dt+c(\delta)\int_I(1+\|\overline\bu\|_{W^{1,2}_\bx}^2)(1+\|\partial_t\eta\|_{W^{2,2}_\by}^2)\dt.
\end{align*}
Finally we notice that $III_\bfA$ has been estimated in \eqref{later1}.

Recalling the definition of $\bfh_{\eta}$ in \eqref{matrices}, we have
\begin{align*}
&\int_I\int_\Omega\bfh_\eta(\overline{\bu})\cdot\boldsymbol\Upsilon^\eta\dx\dt\\
&\sim \int_I\int_\Omega \nabla\overline\bu\,\overline\bu\partial_t\nabla\eta\overline\bfvarphi^\eta\dx\dt +\int_I\int_\Omega\nabla\overline\bu\partial_t\eta\,\partial_t\nabla\eta\,\overline\bfvarphi^\eta\dt +\int_I\int_\Omega\overline\bff\partial_t\nabla\eta\overline\bfvarphi^\eta\dt.
\end{align*}
The first integral has been estimated in $(VII)_2^\eta$. We treat the remaining integrals below:
\begin{align*}
\int_I\int_\Omega\nabla\overline\bu\partial_t\eta\,\partial_t\nabla\eta\,\overline\bfvarphi^\eta\dx\dt&\lesssim \int_I\|\nabla\overline\bu\|_{L^2_\bx}\|\partial_t\eta\|_{L^\infty_\by}\|\partial_t\naby\eta\|_{L^\infty_\by}\|\partial_t\eta\|_{W^{2s,2}_\by}\dt\\
&\lesssim \int_I\|\nabla\overline\bu\|_{L^2_\bx}\|\partial_t\eta\|_{L^2_\by}^{\frac{3}{4}}\|\partial_t\Dely\eta\|_{L^2_\by}^{\frac{1}{4}}\|\partial_t\eta\|_{L^2_\by}^{\frac{1}{4}}\|\partial_t\Dely\eta\|_{L^2_\by}^{\frac{3}{4}}\|\partial_t\eta\|_{W^{2s,2}_\by}\dt\\
&\lesssim \int_I(1+\|\overline\bu\|_{W^{1,2}_\bx}^4)\dt+\int_I\|\partial_t\eta\|_{W^{2,2}_\by}^2\dt.
\end{align*}
Fnally using $\overline\bff\in L^\infty_t(L^2_\bx)$, we have
\begin{align*}
\int_I\int_\Omega\overline\bff\partial_t\nabla\eta\overline\bfvarphi^\eta\dt&\lesssim \int_I\|\overline\bff\|_{L^2_\bx}\|\partial_t\naby\eta\|_{L^\infty_\by}\|\partial_t\Delta_{-h}^s\Delta_h^s\eta\|_{L^2_\by}\dt\\
&\lesssim \int_I\left(\|\partial_t\eta\|_{W^{2,2}_\by}^2+\|\overline\bu\|_{W^{1,2}_\bx}^2\right)\dt.
\end{align*}

We continue the estimate of the solid part in \eqref{for2dlater}. We have
\begin{align*}
\int_I\int_\omega\partial_t^2\eta\bn^\intercal\boldsymbol{\Upsilon}^\eta\circ\bfvarphi \dy\dt&\leq \int_I\|\partial_t^2\eta\|_{L^2_\by}\|\partial_t\eta\|_{L^\infty_\by}\|\partial_t\Delta_\by^{1/4}\eta\|_{L^2_\by}\dt\nonumber\\
&\leq \int_I\|\partial_t^2\eta\|_{L^2_\by}^2\dt+\int_I\|\partial_t\eta\|_{W^{2,2}_\by}^2\|\overline\bu\|_{W^{1,2}_\bx}^2\dt.
\end{align*}
For the integral depends on $\Dely^2\eta$ we take an integration by parts and obtain
\begin{align*}
&\int_I\int_\omega\Dely^2\eta \bn^\intercal\boldsymbol{\Upsilon}^\eta\circ\bfvarphi\dy\dt\\
&=-\int_I\int_\omega\naby\Dely\eta\naby\eta
\partial_t\eta(\Delta_{-h}^s\Delta_h^s\partial_t\eta-\mathscr K_\eta(\Delta_{-h}^s\Delta_h^s\partial_t\eta))\dy\dt\\
&\quad-\int_I\int_\omega\naby\Dely\eta (1+\eta)\partial_t\naby\eta(\Delta_{-h}^s\Delta_h^s\partial_t\eta-\mathscr K_\eta(\Delta_{-h}^s\Delta_h^s\partial_t\eta))\dy\dt\\
&\quad -\int_I\int_\omega\naby\Dely\eta (1+\eta)\partial_t\eta \Delta_{-h}^s\Delta_h^s\partial_t\naby\eta\dy\dt
\\
&\lesssim \int_I\|\eta\|_{W^{3,2}_\by}\|\naby\eta\|_{L^{\infty}_\by}\|\partial_t\eta\|_{L^{2}_\by}\|\partial_t\eta\|_{W^{1/2,\infty}_\by}\dt+ \int_I\|\eta\|_{W^{3,2}_\by}\|\partial_t\naby\eta\|_{L^{\infty}_\by}\|\partial_t\eta\|_{W^{1/2,2}_\by}\dt
\\
&\quad + \int_I\|\eta\|_{W^{3,2}_\by}\|\partial_t\eta\|_{L^{4}_\by}\|\partial_t\eta\|_{W^{1+2s,4}_\by}\dt
\\
&\lesssim \int_I\|\eta\|_{W^{3,2}_\by}\|\partial_t\eta\|_{W^{2,2}_\by}\dt+ \int_I\|\eta\|_{W^{3,2}_\by}\|\partial_t\eta\|_{W^{2,2}_\by}\|\overline\bu\|_{W^{1,2}_\bx}\dt+\int_I\|\eta\|_{W^{3,2}_\by}\|\partial_t\Dely\eta\|_{L^2_\by}^{\frac{3}{4}+s}\dt\\
&\leq\,\delta \int_I\|\eta\|_{W^{3,2}_\by}^2\dt+c(\delta)\int_I\left(1+\|\partial_t\eta\|^2_{W^{2,2}_\by}\right)\Big(\|\overline\bu\|_{W^{1,2}_\bx}^2+1\Big)\dt.
\end{align*}
Finally, we also have
\begin{align*}
\int_I\int_\omega g\bn^\intercal\boldsymbol{\Upsilon}^\eta\circ\bfvarphi \dy\dt\lesssim \int_I\|g\|_{L^2_\by}\|\partial_t\eta\|_{L^2_\by}\|\partial_t\eta\|_{W^{2s,\infty}_\by}\dt\lesssim \int_I\|g\|_{L^2_\by}^2\dt+\int_I\|\partial_t\eta\|_{W^{2,2}_\by}^2\dt.
\end{align*}
Thus we finish the estimate of \eqref{mom:reg2D}.
\end{proof}
Now we are in a position to prove Proposition \ref{thm:transformedSystemLinea2D}.

\begin{proof}[Proof of Theorem~\ref{thm:transformedSystemLinea2D}]
Based on Theorem \ref{thm:NS3}, which also hold for $n=2$ by constant elongation. Hence there exists a local strong solution for short times. In particular $\partial_t^2\eta,\partial_t \overline\bu $ have well defined initial conditions and
\begin{align}
\label{eq:initial}
\int_\omega\abs{\Delta_\tau\partial_t\eta(0)}^2\dy +\int_{\Omega} \abs{\Delta_\tau\overline\bv(0)}^2\dx <\infty
\end{align}
uniformly $h$. 
The goal is to obtain estimates which do not blow up in time and to extend the local solution to a global one.
The regularity of the strong solution does not suffice to differentiate \eqref{contEqAloneBar}--\eqref{momEqAloneBar} in time and use the test function pair:
$$(\partial_t^2\eta+\bn^\intercal \bfG\circ\bfvarphi, \quad   \partial_t\overline\bu+\bfG),$$
 where $\bfG=\bfB_\eta^{-\intercal}\partial_t\bfB_\eta^\intercal\overline\bu$. Hence the following computations are formal. They do hold nevertheless as it can be easily checked that the very same computations can be performed for the respective version using finite in time-differences, in particular because of \eqref{eq:initial}. While in the 3D case we took the effort to perform all estimates on the finite differences in the 2D case we chose to use derivatives as this allows to follow the argument better.
Hence (formally) using $(\partial_t^2\eta+\bn^\intercal \bfG\circ\bfvarphi, \,\, \partial_t\overline\bu+\bfG),$ as test function on the (formally) time-differentiated system implies
\begin{align}
&\sup_I
\int_\omega\big(\vert \partial_t^2\eta\vert^2
+
\vert   \partial_t\Dely\eta\vert^2 \big)\dy
 +
 \sup_I
 \int_\Omega\vert  \partial_t\overline{\bu}\vert^2\dx
+ \sup_I
 \int_\Omega\vert  \nabla\overline{\bu}\vert^2\dx
+
 \int_I\int_\Omega\vert  \partial_t\nabx\overline{\bu}\vert^2\dx\dt\nonumber
\\
&\lesssim
\int_\omega\big(\vert \partial_t^2\eta(0)\vert^2
+
\vert \Dely \eta_*\vert^2 \big)\dy +
\int_\Omega\vert \partial_t\overline{\bu}(0)\vert^2\dx
+ \int_\Omega\vert  \nabx\overline{\bu}_0\vert^2\dx\nonumber\\
 &
\quad +
\int_I\int_\omega\partial_tg \,\partial_t^2\eta\dy\dt+
 \int_I\int_\Omega
 \partial_t\bfA_\eta\nabla\overline{\bv}\cdot\partial_t\nabla\overline\bu
 \dx\dt+
 \int_I\int_\Omega\big(\partial_t\mathbf{h}_\eta-\partial_t J_\eta\partial_t\overline\bu\big)\cdot\partial_t\overline{\bv}
 \dx\dt\nonumber\\
 &\quad +\int_I\int_\Omega\pi\Div(\partial_t \bfB_\eta^\intercal(\partial_t\overline{\bv}+\bfG))\dx\dt
 +\int_I\int_\Omega\big(\partial_t\bfh_\eta-\partial_t J_\eta\partial_t\overline\bu\big)\cdot \bfG \dx\dt\nonumber\\
 &\quad -\int_I\int_\Omega J_\eta\partial_t^2\overline{\bv}\cdot \bfG\dx\dt
 -\int_I\int_\Omega\partial_t\bfA_\eta\nabla\overline{\bv}:\nabla\bfG\dx\dt -\int_I\int_\Omega\bfA_\eta\nabla\partial_t\overline{\bv}:\nabla\bfG \dx\dt\nonumber\\
 &\quad +\int_I\int_\omega\partial_t g\, \bn^\intercal\bfG\circ\bfvarphi\dy\dt
-\int_I\int_\omega\partial_t^3\eta\,\bn^\intercal\bfG\circ\bfvarphi\dy\dt
-\int_I\int_\omega\partial_t\Delta^2_\by\eta\, \bn^\intercal\bfG\circ\bfvarphi\dy\dt,
\label{dt3.8a2d}
\end{align}
analogously to \eqref{est}.
Here $\partial_t\overline \bu(0)$ and $\partial_t^2\eta (0)$ can be computed from the data using the equation at $t=0$. Now we make use of another test in order to improve the spatial regularity of $\eta$. 

{\bf Estimate of the solid part in \eqref{dt3.8a2d}.} Note that 
$\partial_t\bfB_\eta^\intercal\circ\bfvarphi \bn=0$ does not hold in the case of shells in general. Thus the integrals in the last line of \eqref{dt3.8a2d} appear additionally to those arsing in the case of plates 
from \cite{ScSu}. 
For this recalling that $\bfG\circ\bfvarphi\sim |\partial_t\eta|^2(1+\eta)\bfn$ we have
\begin{align*}
&\int_I\int_\omega\partial_tg \bn^\intercal\bfG\circ\bfvarphi\dy\dt\lesssim \int_I\|\partial_tg\|_{L^2_\by}\|\partial_t\eta\|_{L^4_\by}^2\dt\\
&\lesssim \int_I\|\partial_tg\|_{L^2_\by}\|\partial_t\eta\|_{L^2_\by}\|\partial_t\eta\|_{W^{1/2,2}_\by}\dt\lesssim \int_I\|\partial_tg\|_{L^2_\by}^2\dt+\int_I\|\overline\bu\|_{W^{1,2}_\bx}^2\dt.
\end{align*}
Taking an integration by parts in time, we see that
\begin{align}
&\int_I\int_\omega \partial_t^3\eta\partial_t\eta\partial_t\eta(1+\eta)\dy\dt\nonumber\\
&=-2\int_I\int_\omega |\partial_t^2\eta|^2\partial_t\eta(1+\eta)\dy\dt-\int_I\int_\omega\partial_t^2\eta|\partial_t\eta|^3\dy\dt
+\int_\omega \partial_t^2\eta|\partial_t\eta|^2(1+\eta)\dy\bigg|_{s=0}^{s=T}\nonumber
\\
&\lesssim \int_I\|\partial_t^2\eta\|^2_{L^4_\by}\|\partial_t\eta\|_{L^2_\by}\dt+\int_I\|\partial_t^2\eta\|_{L^2_\by}\|\partial_t\eta\|_{L^6_\by}^3\dt
+\sup_I\|\partial_t^2\eta\|_{L^2_\by}\|\partial_t\eta\|_{L^2_\by}\|\partial_t\eta\|_{L^\infty_\by}\nonumber\\
&\lesssim \int_I\|\partial_t^2\eta\|_{L^2_\by}\|\partial_t^2\eta\|_{W^{1/2,2}_\by}
\dt+\int_I\|\partial_t^2\eta\|_{L^2_\by}\|\partial_t\eta\|_{L^2_\by}^2\|\partial_t\eta\|_{W^{1,2}_\by}\dt
+\sup_I\|\partial_t^2\eta\|_{L^2_\by}\|\partial_t\eta\|_{L^2_\by}^{\frac{1}{2}}
\|\partial_t\nabla\eta\|_{L^2_\by}^{\frac{1}{2}}\nonumber\\
&\lesssim c(\delta)\int_I\|\partial_t^2\eta\|_{L^2_\by}^2\dt+\delta\int_I\|\partial_t\overline\bu\|_{W^{1,2}_\by}^2\dt+\int_I\|\partial_t\eta\|_{W^{2,2}_\by}^2\dt\nonumber\\
&\quad +\delta\sup_I\big(\|\partial_t^2\eta\|_{L^2_\by}^2
+\|\partial_t\Dely\eta\|_{L^2_\by}^2\big)+c(\delta),\label{etaestimate}
\end{align}
where we used the trace theorem and the embedding $W^{1/4,2}_\by\hookrightarrow L^4_\by$ and the following $1D$ interpolation inequality
$$\|\partial_t^2\eta\|_{W^{\frac{1}{4},2}_\by}\lesssim \|\partial_t^2\eta\|_{L^2_\by}^{\frac{1}{2}}\|\partial_t^2\eta\|_{W^{\frac{1}{2},2}_\by}^{\frac{1}{2}}. $$ 
Similarly, for the last integral in \eqref{dt3.8a2d} after taking integration by parts in space we just write the high-order terms here:
\begin{align*}
&\int_I\int_\omega \partial_t\Dely^2\eta|\partial_t\eta|^2(1+\eta)\dy\dt\\
&\lesssim \int_I\int_\omega |\partial_t\Dely\eta|^2\partial_t\eta\dy\dt
+\int_I\int_\omega \partial_t\Dely\eta|\partial_t\naby\eta|^2\dy\dt+\int_I\int_\omega\partial_t\Dely\eta|\partial_t\eta|^2\naby^2\eta\dx\dt
\\
&\lesssim \int_I \big(\|\partial_t\Dely\eta\|_{L^4_\by}
\|\partial_t\Dely\eta\|_{L^2_\by}\|\partial_t\eta\|_{L^4_\by}
+\|\partial_t\Dely\eta\|_{L^2_\by}\|\partial_t\naby\eta\|_{L^4_\by}^2\big)\dt+\int_I\|\partial_t\Dely\eta\|_{L^2_\by}\|\partial_t\eta\|_{L^4_\by}^2\|\Dely\eta\|_{L^\infty_\by}\dt\\
&\lesssim c(\delta)\int_I\|\partial_t\Dely\eta\|^2_{L^2_\by}
\|\partial_t\eta\|_{W^{1/4,2}_\by}^2\dt+\delta\int_I\|\partial_t\Dely\eta\|^2_{W^{1/4,2}_\by}
+\int_I\|\partial_t\Dely\eta\|_{L^2_\by}\|\partial_t\eta\|_{W^{1/2,2}_\by}\|\partial_t\eta\|_{W^{2,2}_\by}\dt\\
&\quad +\int_I\|\partial_t\eta\|_{W^{2,2}_\by}\|\partial_t\eta\|_{L^2_\by}\|\partial_t\eta\|_{W^{1/2,2}_\by}\|\eta\|_{W^{3,2}_\by}\dt\\
&\lesssim c(\delta)\int_I\|\partial_t\eta\|^2_{W^{2,2}_\by}
\big(\|\overline\bu\|_{W^{1,2}_\bx}^2+1\big)\dt+\delta\int_I\|\partial_t\Dely\eta\|^2_{W^{1/4,2}_\by}+\delta\int_I\|\eta\|_{W^{3,2}_\by}^2\dt.
\end{align*}
Note that the last term can be handled by means of \eqref{eq:MuSc'2D}.
Now we estimate the remaining integrals on the right hand side of \eqref{dt3.8a2d} in order. 

We note that
$$\int_I\int_\omega\partial_t g\partial_t^2\eta\dy\dt\lesssim \int_I\|\partial_tg\|_{L^2_\by}^2\dt+\int_I\|\partial_t^2\eta\|_{L^2_\by}^2\dt. $$
Then we estimate further that
\begin{align*}
	&\int_I\int_\Omega\partial_t\bfA_\eta\nabla\overline\bu\cdot\partial_t\nabla\overline\bu\dxt\\
	&\lesssim\int_I\|\partial_t\nabla\overline\bu\|_{L^2_\bx}\|\partial_t\naby\eta\|_{L^\infty_\by}\|\nabla\overline\bu\|_{L^2_\bx}\dt\\
	&\lesssim\delta\int_I\|\partial_t\nabla\overline\bu\|_{L^2_\bx}^2\dt+c(\delta)\int_I\|\partial_t\naby\eta\|_{L^\infty_\by}^2\|\nabla\overline\bu\|_{L^2_\bx}^2\dt\\
	&\lesssim\delta\int_I\|\partial_t\nabla\overline\bu\|_{L^2_\bx}^2\dt+c(\delta)\int_I\|\partial_t\Dely\eta\|^2_{L^{2}_\by}\|\nabla\overline\bu\|_{L^2_\bx}^2\dt.
\end{align*}
Recalling $\bfh_\eta$ in \eqref{matrices}, we decompose the integral $\int_I\int_\Omega (\partial_t\bfh_\eta-\partial_tJ_\eta\partial_t\overline\bu)\cdot\partial_t\overline\bu \dx\dt$ into
\begin{align*}
({\tt I})^{\bfh}&:=- \int_I\int_\Omega\partial_t J_\eta|\partial_t \overline{\bu}|^2\dxt,\\
({\tt II})^{\bfh}&:=- \int_I\int_\Omega\partial_t\big(
J_\eta \nabx\overline{\bu}\cdot \partial_t \bfPsi_\eta^{-1}\circ \bfPsi_\eta\big)\cdot\partial_t\overline{\bv}\dxt,
\\({\tt III })^{\bfh}&:=- \int_I\int_\Omega\partial_t\big(\mathbf{B}_\eta\nabx\overline{\bfv}~\overline{\bu}\big)\cdot\partial_t\overline{\bv}\dxt,\\
({\tt IV})^{\bfh}&:= \int_I\int_\Omega\partial_t\big(
J_\eta  \bff\circ \bfPsi_\eta\big)\cdot\partial_t\overline{\bv}\dxt,
\end{align*}
where the last term clearly is of lower order.
As for $(\mathrm{I})^{\bfh}$ it holds that
\begin{align*}
	(\mathrm{I})^{\bfh}&
	\lesssim\int_I\|\partial_t\eta\|_{L^2_\by}\|\partial_t\overline\bu\|^2_{L^4_\bx}\dt\lesssim \int_I\|\partial_t\overline\bu\|_{L^2_\bx}\|\partial_t\overline\bu\|_{W^{1,2}_\bx}\dt\\
	&\lesssim \delta\int_I\|\partial_t\overline\bu\|_{W^{1,2}_\bx}^2\dt+c(\delta)\int_I\|\partial_t\overline\bu\|_{L^2_\bx}^2\dt.
\end{align*}
For $({\tt III })^{\bfh}$, we have
\begin{align*}
({\tt III })^{\bfh}&=- \int_I\int_\Omega\partial_t\mathbf{B}_\eta\nabx\overline{\bfv}~\overline{\bu}\cdot\partial_t\overline{\bv}\dxt- \int_I\int_\Omega\mathbf{B}_\eta\partial_t\nabx\overline{\bfv}~\overline{\bu}\cdot\partial_t\overline{\bv}\dxt- \int_I\int_\Omega\mathbf{B}_\eta\nabx\overline{\bfv}~\partial_t\overline{\bu}\cdot\partial_t\overline{\bv}\dxt\\
&=:({\tt III })^{\bfh}_1+({\tt III })^{\bfh}_2+({\tt III })^{\bfh}_3,
\end{align*}
where, by \eqref{210and212} and \eqref{218},
\begin{align*}
({\tt III })^{\bfh}_1&\lesssim \int_I\|\partial_t\mathbf B_\eta\|_{L^{\infty}_\bx}\|\nabx\overline{\bfv}\|_{L^2_\bx}\|\overline{\bu}\|_{L^4_\bx}\|\partial_t\overline{\bv}\|_{L^4_\bx}\dt\\
&\lesssim \int_I\|\partial_t\Dely\eta\|_{L^2_\by}\|\overline{\bfv}\|_{W^{1,2}_\bx}^{3/2}\|\overline{\bfv}\|_{L^{2}_\bx}^{1/2}\|\partial_t\overline{\bv}\|_{L^2_\bx}^{1/2}\|\partial_t\overline{\bv}\|_{W^{1,2}_\bx}^{1/2}\dt\\
&\leq\,\delta \int_I\|\partial_t\overline{\bfv}\|_{W^{1,2}_\bx}^2\dt+c(\delta)\int_I\|\overline{\bv}\|_{W^{1,2}_\bx}^{2}\big(1+\|\partial_t\Dely\eta\|_{L^2_\by}^{2}+\|\partial_t\overline{\bv}\|_{L^2_\bx}^{2}\big)\dt,
\end{align*}
\begin{align*}
({\tt III })^{\bfh}_2&\lesssim \int_I\|\mathbf B_\eta\|_{L^{\infty}_\bx}\|\partial_t\nabx\overline{\bfv}\|_{L^2_\bx}\|\overline{\bu}\|_{L^4_\bx}\|\partial_t\overline{\bv}\|_{L^4_\bx}\dt\\
&\lesssim \int_I\big(1+\|\naby\eta\|_{L^{\infty}_\by} \big)\|\partial_t\overline{\bu}\|_{W^{1,2}_\bx}^{3/2}\|\overline{\bu}\|_{W^{1,2}_\bx}^{1/2}\|\overline{\bu}\|_{L^{2}_\bx}^{1/2}\|\partial_t\overline{\bv}\|_{L^2_\bx}^{1/2}\dt\\
&\leq\,\delta \int_I\|\partial_t\overline{\bfv}\|_{W^{1,2}_\bx}^2\dt+c(\delta)\int_I\|\overline{\bv}\|_{W^{1,2}_\bx}^{2}\|\partial_t\overline{\bv}\|_{L^2_\bx}^{2}\dt,
\end{align*}
as well as
\begin{align*}
({\tt III })^{\bfh}_3&\lesssim \int_I\|\mathbf B_\eta\|_{L^{\infty}_\bx}\|\nabx\overline{\bfv}\|_{L^2_\bx}\|\partial_t\overline{\bu}\|_{L^4_\bx}^2\dt\\
&\lesssim \int_I\big(1+\|\naby\eta\|_{L^{\infty}_\by}\big)\|\overline{\bu}\|_{W^{1,2}_\bx}\|\partial_t\overline{\bu}\|_{W^{1,2}_\bx}\|\partial_t\overline{\bv}\|_{L^2_\bx}\dt\\
&\leq\,\delta \int_I\|\partial_t\overline{\bfv}\|_{W^{1,2}_\bx}^2\dt+c(\delta)\int_I\|\overline{\bv}\|_{W^{1,2}_\bx}^{2}\|\partial_t\overline{\bv}\|_{L^2_\bx}^{2}\dt.
\end{align*}

We proceed similarly for $({\tt II})^{\bfh}$ obtaining
\begin{align*}
({\tt II})^{\bfh}=&- \int_I\int_\Omega\partial_t
J_\eta \nabx\overline{\bu}\cdot \partial_t \bfPsi_\eta^{-1}\circ \bfPsi_\eta\cdot\partial_t\overline{\bv}\dxt\\
&- \int_I\int_\Omega
J_\eta \partial_t\nabx\overline{\bu}\cdot \partial_t \bfPsi_\eta^{-1}\circ \bfPsi_\eta\cdot\partial_t\overline{\bv}\dxt\\
&- \int_I\int_\Omega
J_\eta \nabx\overline{\bu}\cdot \partial_t(\partial_t \bfPsi_\eta^{-1}\circ \bfPsi_\eta)\cdot\partial_t\overline{\bv}\dxt\\
=:&({\tt II})^{\bfh}_1+({\tt II})^{\bfh}_2+({\tt II})^{\bfh}_3.
\end{align*}
We first see that
\begin{align*}
	({\tt II})^{\bfh}_1&\lesssim \int_I\|\partial_t\eta\|_{L_\by^\infty}^2\|\nabla\overline{\bfv}\|_{L^2_\bx}\|\partial_t\overline{\bfv}\|_{L^2_\bx}\dt\\
	&\lesssim \int_I\|\partial_t\eta\|_{W^{2,2}_\by}^2\|\nabla\overline{\bfv}\|_{L^2_\bx}\|\partial_t\overline{\bfv}\|_{L^2_\bx}\dt\\
	&\lesssim \int_I\|\partial_t\eta\|_{W^{2,2}_\by}^{2}\big(\|\overline{\bv}\|_{W^{1,2}_\bx}^2+\|\partial_t\overline{\bv}\|_{L^2_\bx}^{2}\big).
\end{align*}
Now we use a trick frequently used in \cite{ScSu}
to write the integral over $\Omega$ as an iterated integral and use 1D embeddings. This is more complicated for shells than for plates. Fortunately, the function $\partial_t\bfPsi_\eta$ is supported in $S_L$, where coordinates $(s,y)\in(-L,L)\times\omega$ exists together with a smooth diffeomorphism $\Lambda :(-L,L)\times \partial\Omega\rightarrow S_L$, cf. Section \ref{ssec:geom}. Moreover, we can parametrise $\partial\Omega$ via the function $\bfvarphi:\omega\rightarrow\partial\Omega$.
 The function
\begin{equation}\label{mapneed}
\Lambda^\bfvarphi:(-L,L)\times \omega\rightarrow S_L,\quad (s,y)\mapsto \Lambda(s,\bfvarphi(y)) 
\end{equation}
is also a smooth diffeomorphism.
We estimate $(\tt II)_2^{\bfh}$ as follows using also \eqref{218} and \eqref{energyEstLinear} that
\begin{align}\label{eq:2405}
\begin{aligned}
(\tt II)_2^{\bfh}&\lesssim \int_I\int_{-L}^0\|\partial_t\nabla\overline\bu\circ\Lambda^\bfvarphi\|_{L^2_\by}\|\partial_t\eta\|_{L^2_\by}\|\partial_t\overline\bu\circ\Lambda^\bfvarphi\|_{L^\infty_\by}\dd z\dt
\\
&\lesssim \int_I\int_{-L}^0\|\partial_t\nabla\overline\bu\circ\Lambda^\bfvarphi\|_{L^2_\by}\|\partial_t\eta\|_{L^2_\by}\|\partial_t\overline\bu\circ\Lambda^\bfvarphi\|_{L^2_\by}^{\frac{1}{2}}\|\partial_t\overline\bu\circ\Lambda^\bfvarphi\|_{W^{1,2}_\by}^{\frac{1}{2}}\dd z\dt\\
&\lesssim \int_I\|\partial_t\overline\bu\|_{W^{1,2}_\bx}^{\frac{3}{2}}\|\partial_t\overline\bu\|_{L^2_\bx}^{\frac{1}{2}}\dt
\\
&\lesssim\delta\int_I\|\partial_t\overline\bu\|_{W^{1,2}_\bx}^2\dt+c(\delta)\int_I\|\partial_t\overline\bu\|_{L^2_\bx}^2\dt.
\end{aligned}
\end{align}
transforming back to $S_L$ in the penultimate step. 

Finally, using the interpolation of $L^\infty_\by$ between $L^2_\by$ and $W^{1,2}_\by$ we obtain
arguing as in \eqref{eq:2405} that
\begin{align*}
({\tt II})^{\bfh}_3&=\int_I\int_\Omega
J_\eta \nabx\overline{\bu}\cdot \partial_t^2 \bfPsi_\eta^{-1}\circ \bfPsi_\eta\cdot\partial_t\overline{\bv}\dxt+\int_I\int_\Omega
J_\eta \nabx\overline{\bu}\cdot \partial_t\nabla \bfPsi_\eta^{-1}\circ \bfPsi_\eta\partial_t\bfPsi_\eta\cdot\partial_t\overline{\bv}\dxt\\
&\lesssim \int_I\int_{-L}^0\|\nabx\overline{\bu}\circ \Lambda^\bfvarphi\|_{L^2_\by}\|\partial_t^2\eta\|_{L^2_\by}\|\partial_t\overline{\bv}\circ \Lambda^\bfvarphi\|_{L^\infty_\by}\,\dd z\dt
\\
&\quad +\int_I\|\nabx\overline{\bu}\|_{L^2_\bx}\|\partial_t\naby\eta\|_{L^\infty_\by}\|\partial_t\eta\|_{L^\infty_\by}\|\partial_t\overline{\bv}\|_{L^2_\bx}\dt
\\
&\lesssim \int_I\int_{-L}^0\|\nabx\overline{\bu}\circ \Lambda^\bfvarphi\|_{L^2_\by}\|\partial_t^2\eta\|_{L^2_\by}\|\partial_t\overline{\bv}\circ \Lambda^\bfvarphi\|_{L^2_\by}^{1/2}\|\partial_t\overline{\bv}\|_{W^{1,2}_\bx}^{1/2}\,\dd z\dt
\\&\quad +\int_I\|\nabx\overline{\bu}\|_{L^2_\bx}\|\partial_t\naby\eta\|_{L^2_\bx}\|\partial_t\Dely\eta\|_{L^2_\by}^{1/2}\|\partial_t\eta\|_{L^2_\by}^{1/2}\|\partial_t\overline{\bv}\|_{L^2_\bx}\dt\\
&\lesssim \int_I\|\nabx\overline{\bu}\|_{L^2_\bx}\|\partial_t^2\eta\|_{L^2_\by}\|\partial_t\overline{\bv}\|_{L^2_\bx}^{1/2}\|\partial_t\overline{\bv}\|_{W^{1,2}_\bx}^{1/2}\dt+\int_I\|\nabx\overline{\bu}\|_{L^2_\bx}\|\partial_t\Dely\eta\|_{L^2_\by}\|\partial_t\overline{\bv}\|_{L^2_\bx}\dt\\
&\lesssim \delta\int_I\|\partial_t\overline{\bu}\|_{W^{1,2}_\bx}^2\dt+c(\delta)\int_I\big(\|\overline{\bv}\|_{W^{1,2}_\bx}^2+1\big)\big(\|\partial_t\Dely\eta\|_{L^2_\by}^{2}+\|\partial_t^2\eta\|_{L^2_\by}^{2}+\|\partial_t\overline{\bv}\|_{L^2_\bx}^{2}\big)\dt.
\end{align*}

It remains to estimate the terms in \eqref{dt3.8a2d} related to $\bfG=\bfB_\eta^{-\intercal}\partial_t\bfB_\eta\overline \bu$. We now consider $\int_I\int_\Omega(\partial_t\bfh_\eta-\partial_t J_\eta\partial_t\overline\bu)\cdot \bfG\dx\dt$, which is again decomposed as
\begin{align*}
({\tt I})^{\bfh,\bfG}&:=- \int_I\int_\Omega\partial_tJ_\eta\partial_t \overline{\bu}\cdot \bfG\dxt,\\
({\tt II})^{\bfh,\bfG}&:=- \int_I\int_\Omega\partial_t\big(
J_\eta \nabx\overline{\bu}\cdot \partial_t \bfPsi_\eta^{-1}\circ \bfPsi_\eta\big)\cdot\bfG\dxt,
\\({\tt III })^{\bfh,\bfG}&:=- \int_I\int_\Omega\partial_t\big(\mathbf{B}_\eta\nabx\overline{\bfv}~\overline{\bu}\big)\cdot\bfG\dxt,\\
({\tt IV})^{\bfh,\bfG}&:= \int_I\int_\Omega\partial_t\big(
J_\eta  \bff\circ \bfPsi_\eta\big)\cdot\bfG\dxt.
\end{align*}
By \eqref{210and212}, \eqref{218} and \eqref{energyEstLinear}, we have
\begin{align*}
	({\tt I})^{\bfh,\bfG}&\lesssim \int_I\|\partial_t\eta\|_{L^\infty_\by}\|\partial_t \overline{\bu}\|_{L^2_\bx}\big(1+\|\naby\eta\|_{L^\infty_\by}\big)\|\partial_t\naby\eta\|_{L^\infty_\by}\|\overline\bu\|_{L^2_\bx}\dt
	\\
	&\lesssim \int_I\|\partial_t\eta\|_{L^2_\by}^{3/4}\|\partial_t\Dely\eta\|_{L^2_\by}^{1/4}\|\partial_t \overline{\bu}\|_{L^2_\bx}\|\partial_t\eta\|_{L^2_\by}^{1/4}\|\partial_t\Dely\eta\|_{L^2_\by}^{3/4}\dt\\
	&\lesssim  \int_I\big(\|\partial_t
	\Dely\eta\|^2_{L^2_\by}+ \|\partial_t\overline{\bu}\|_{L^2_\bx}^2\big)\dt.
\end{align*}
Then with the obvious sub-decompositions into
$({\tt III })^{\bfh,\bfG}_1,({\tt III })^{\bfh,\bfG}_2$ and $({\tt III })^{\bfh,\bfG}_3$
as well as
$({\tt II})^{\bfh,\bfG}_1,({\tt II})^{\bfh,\bfG}_2$ and $({\tt II})^{\bfh,\bfG}_3$.
It holds by \eqref{218} and \eqref{energyEstLinear} that
\begin{align*}
({\tt III })^{\bfh,\bfG}_1&\lesssim \int_I\|\partial_t\naby\eta\|_{L^{\infty}_\by}^2\big(1+\|\naby\eta\|_{L^{\infty}_\by}\big)\|\nabx\overline{\bfv}\|_{L^2_\bx}\|\overline{\bu}\|_{L^4_\bx}^2\dt\\
&\lesssim \int_I\|\partial_t\eta\|_{L^{2}_\by}^{\frac{1}{2}}\|\partial_t\Dely\eta\|_{L^{2}_\by}^{\frac{3}{2}}\|\overline{\bfv}\|_{W^{1,2}_\bx}^2\|\overline{\bfv}\|_{L^{2}_\bx}\dt\\
&\lesssim  \int_I\|\partial_t\Dely\eta\|^2_{L^{2}_\by}\|\overline{\bfv}\|_{W^{1,2}_\bx}^2\dt,
\end{align*}
\begin{align*}
({\tt III })^{\bfh,\bfG}_2&\lesssim \int_I\|\partial_t\naby\eta\|_{L^{\infty}_\by}\big(1+\|\naby\eta\|_{L^{\infty}_\by}^2\big)\|\partial_t\nabx\overline{\bfv}\|_{L^2_\bx}\|\overline{\bu}\|_{L^4_\bx}^2\dt\\
&\lesssim \int_I\|\partial_t\Dely\eta\|^{\frac{3}{4}}_{L^{2}_\by}\|\partial_t\eta\|_{L^{2}_\by}^{\frac{1}{4}}\|\partial_t\nabx\overline{\bfv}\|_{L^{2}_\bx}\|\overline\bu\|_{W^{1,2}_\bx}\|\overline\bu\|_{L^{2}_\bx}\dt\\
&\lesssim \delta\int_I\|\partial_t\nabx\overline{\bfv}\|_{L^{2}_\bx}^2\dt+c(\delta)  \int_I(1+\|\partial_t\Dely\eta\|^2_{L^{2}_\by})\|\overline{\bfv}\|_{W^{1,2}_\bx}^2\dt,
\end{align*}
\begin{align*}
({\tt III })^{\bfh,\bfG}_3&\lesssim \int_I\|\partial_t\naby\eta\|_{L^{\infty}_\by}\big(1+\|\naby\eta\|_{L^{\infty}_\by}^2\big)\|\nabx\overline{\bfv}\|_{L^2_\bx}\|\overline{\bu}\|_{L^4_\bx}\|\partial_t\overline{\bu}\|_{L^4_\bx}\dt\\
&\lesssim \int_I\|\partial_t\Dely\eta\|_{L^{2}_\by}^{\frac{3}{4}}\|\partial_t\eta\|_{L^{2}_\by}^{\frac{1}{4}}\|\bu\|_{W^{1,2}_\bx}^{\frac{3}{2}}\|\bu\|_{L^{2}_\bx}^{\frac{1}{2}}\|\partial_t\bu\|_{W^{1,2}_\bx}^{\frac{1}{2}}\|\partial_t\bu\|_{L^{2}_\bx}^{\frac{1}{2}}\dt\\
&\lesssim \delta\int_I\|\partial_t\nabx\overline{\bfv}\|_{L^{2}_\bx}^2\dt+c(\delta)\int_I\|\nabla\overline\bu\|_{L^2_\bx}^2\|\partial_t\Dely\eta\|_{L^2_\by}\|\partial_t\overline\bu\|_{L^2_\bx}^{\frac{2}{3}}\dt\\
&\lesssim \delta\int_I\|\partial_t\nabx\overline{\bfv}\|_{L^{2}_\bx}^2\dt+c(\delta)\int_I\|\nabla\overline\bu\|_{L^2_\bx}^2\left(\|\partial_t\Dely\eta\|_{L^2_\by}^2+\|\partial_t\overline\bu\|_{L^2_\bx}^2+1\right)\dt,
\end{align*}
as well as
\begin{align*}
({\tt II})^{\bfh,\bfG}_1&\lesssim  \int_I\|\partial_t
\eta\|^2_{L^\infty_\by} \|\nabx\overline{\bu}\|_{L^2_\bx}\big(1+\|\naby\eta\|_{L^\infty_\by}\big)\|\partial_t\naby\eta\|_{L^\infty_\by}\|\overline{\bu}\|_{L^2_\bx}\dt\\
&\lesssim  \int_I\|\partial_t
\eta\|_{L^2_\by}^{7/4} \|\nabx\overline{\bu}\|_{L^2_\bx}\|\partial_t\Dely\eta\|^{5/4}_{L^2_\by}\dt\\
&\lesssim\int_I\|\partial_t\Dely\eta\|_{L^2_\by}^2\|\overline\bu\|_{W^{1,2}_\bx}^2\dt+\int_I(1+\|\partial_t\Dely\eta\|_{L^2_\by}^2)\dt,
\end{align*}
\begin{align*}
({\tt II})^{\bfh,\bfG}_2&\lesssim  \int_I \|\partial_t\nabx\overline{\bu}\|_{L^2_\bx}\|\partial_t\eta\|_{L^\infty_\by}\big(1+\|\naby\eta\|_{L^\infty_\by}\big)\|\partial_t\naby\eta\|_{L^\infty_\by}\|\overline{\bu}\|_{L^2_\bx}\dt\\
&\lesssim  \int_I\|\partial_t\nabla\overline\bu\|_{L^2_\bx}\|\partial_t\eta\|_{L^2_\by}^{\frac{3}{4}}\|\partial_t\Dely\eta\|_{L^2_\by}^{\frac{1}{4}}\|\partial_t\eta\|_{L^2_\by}^{\frac{1}{4}}\|\partial_t\Dely\eta\|_{L^2_\by}^{\frac{3}{4}}\dt\\
&\lesssim \delta\int_I\|\partial_t\nabla\overline\bu\|_{L^2_\bx}^2\dt+c(\delta)\int_I\|\partial_t\Dely\eta\|_{L^2_\by}^2\dt,
\end{align*}
and, finally, using the trace embedding $W^{3/4,2}(\Omega)\hookrightarrow W^{1/4,2}(\partial\Omega)\hookrightarrow L^4(\partial\Omega)$ and arguing as in  \eqref{eq:2405}, we have
\begin{align*}
({\tt II})^{\bfh}_3&=\int_I\int_\Omega
J_\eta \nabx\overline{\bu}\cdot \partial_t^2 \bfPsi_\eta^{-1}\circ \bfPsi_\eta\cdot\bfG\dxt+\int_I\int_\Omega
J_\eta \nabx\overline{\bu}\cdot \partial_t\nabla \bfPsi_\eta^{-1}\circ \bfPsi_\eta\partial_t\bfPsi_\eta\cdot\bfG\dxt\\
&\lesssim \int_I\int_{-L}^0\|\nabx\overline{\bu}\circ\Lambda^{\bfvarphi}\|_{L^2_\by}\|\partial_t^2\eta\|_{L^2_\by}\|\overline{\bv}\circ\Lambda^{\bfvarphi}\|_{L^\infty_\by}\big(1+\|\naby\eta\|_{L_\by^{\infty}}\big)\|\partial_t\naby\eta\|_{L_\by^{\infty}}\,\dd z\dt
\\
&\quad +\int_I\|\nabx\overline{\bu}\|_{L^2_\bx}\|\partial_t\naby\eta\|_{L^\infty_\bx}\|\partial_t\eta\|_{L^\infty_\bx}\|\overline{\bv}\|_{L^2_\bx}\big(1+\|\naby\eta\|_{L_\by^{\infty}}\big)\|\partial_t\naby\eta\|_{L_\by^{\infty}}\dt
\\
&\lesssim \int_I\int_{-L}^0\|\nabx\overline{\bu}\circ\Lambda^{\bfvarphi}\|_{L^2_\by}\|\overline{\bv}\circ\Lambda^{\bfvarphi}\|_{W^{1,2}_\by}\|\partial_t^2\eta\|_{L^2_\by}\|\partial_t\Dely\eta\|_{L_\by^{2}}\,\dd z\dt
\\
&\quad +\int_I\|\nabx\overline{\bu}\|_{L^2_\bx}\|\partial_t\eta\|_{L^2_\bx}^{5/4}\|\partial_t\Dely\eta\|_{L^2_\bx}^{7/4}\dt
\\
&\lesssim \int_I\|\overline{\bv}\|_{W^{1,2}_\bx}^{2}\|\partial_t^2\eta\|_{L^2_\by}\|\partial_t\Dely\eta\|_{L_\by^{2}}\dt+\int_I\|\nabx\overline{\bu}\|_{L^2_\bx}\|\partial_t\Dely\eta\|_{L^2_\bx}^{7/4}\dt
\\
&\lesssim \int_I\big(\|\overline{\bv}\|_{W^{1,2}_\bx}^{2}+1\big)\big(\|\partial_t^2\eta\|^2_{L^2_\by}+\|\partial_t\Dely\eta\|_{L_\by^{2}}^{2}\big)\dt+\int_I\|\overline{\bv}\|_{W^{1,2}_\bx}^{2}\dt.
\end{align*}

Now we estimate
$-\int_I\int_\Omega J_\eta\partial_t^2 \overline{\bu}\cdot \bfG\dx\dt
$.
For this, we integrate by parts in time obtaining that
\begin{align*}
-\int_I\int_\Omega J_\eta\partial_t^2 \overline{\bu}\cdot \bfG\dx\dt&=-\int_I\int_\Omega\partial_t\big(J_\eta\partial_t \overline{\bu}\cdot \bfG\big)\dxt+\int_I\int_\Omega\partial_tJ_\eta\partial_t \overline{\bu}\cdot \bfG\dxt\\&\quad +\int_I\int_\Omega J_\eta\partial_t \overline{\bu}\cdot \partial_t\bfG\dxt.
\end{align*}
Note that the second term coincides with $({\tt I})^{\bfh,\bfG}$ estimate above, while we have 
(using boundedness of $\eta$ due to \eqref{energyEstLinear})
\begin{align*}
	&-\int_I\int_\Omega\partial_t\big(J_\eta\partial_t \overline{\bu}\cdot \bfG\big)\dxt\\&\lesssim \|\eta\|_{L^\infty((0,T_0)\times\omega)}\sup_I\Big(\|\partial_t\overline\bu\|_{L^2_\bx}\|\overline\bu\|_{L^2_\bx}\big(1+\|\naby\eta\|_{L^{\infty}_\by}\big)\|\partial_t\naby\eta\|_{L^{\infty}_\by}\Big)\\
&\lesssim \sup_I\Big(\|\partial_t\overline\bu\|_{L^2_\bx}\|\partial_t\eta\|^{1/4}_{L^2_\by}\|\partial_t\Dely\eta\|^{3/4}_{L^{2}_\by}\Big)\\
&\lesssim \delta\sup_I\big(\|\partial_t\overline\bu\|_{L^2_\bx}^2+\|\partial_t\Dely\eta\|^{2}_{L^{2}_\by}\big)+c(\delta).
\end{align*}
{\bf Control of the integral $ \int_I\int_\Omega J_\eta\partial_t \overline{\bu}\cdot \partial_t\bfG\dxt$.}
Since $\bfB_\eta=\mathrm{cof}(\nabla\bfPsi_\eta)$ there is a linear mapping $\mathcal A:\R^{2\times 2}\rightarrow\R^{2\times 2}$
such that
\begin{align*}
\partial_t\bfG&=\bfB_\eta^{-\intercal}\partial_t\bfB_\eta\partial_t\overline\bu+\partial_t\bfB_\eta^{-\intercal}\partial_t\bfB_\eta\overline\bu+\bfB_\eta^{-\intercal}\mathcal A(\nabla\partial_t^2 \bfPsi_\eta\big)\overline\bu.
\end{align*}
Writing $\mathcal A=\sum_{i=1}^2\mathcal A_{i}$ with $\mathcal A_{i}:\R^2\rightarrow\R^{2\times 2}$ acting on the column vectors of the matrices we obtain
\begin{align}\nonumber
&\int_I\int_\Omega J_\eta\partial_t\overline\bu\cdot\bfB_\eta^{-\intercal}\mathcal A( \partial_t^2\nabla \bfPsi_\eta\big)\overline\bu\dxt=\sum_{i=1}^2\int_I\int_\Omega J_\eta\partial_t\overline\bu\cdot\bfB_\eta^{-\intercal}\mathcal A_i( \partial_t^2\partial_i \bfPsi_\eta\big)\overline\bu\dxt\nonumber\\
&=\sum_{i=1}^2\int_I\int_\Omega\partial_i J_\eta\partial_t\overline\bu\cdot\bfB_\eta^{-\intercal}\mathcal A_i( \partial_t^2 \bfPsi_\eta\big)\overline\bu\dxt+\sum_{i=1}^2\int_I\int_\Omega J_\eta\partial_t\partial_i\overline\bu\cdot\bfB_\eta^{-\intercal}\mathcal A_i( \partial_t^2 \bfPsi_\eta\big)\overline\bu\dxt\nonumber\\
&\quad +\int_I\int_\Omega J_\eta\partial_t\overline\bu\cdot\partial_i\bfB_\eta^{-\intercal}\mathcal A_i( \partial_t^2 \bfPsi_\eta\big)\overline\bu\dxt+\sum_{i=1}^2\int_I\int_\Omega J_\eta\partial_t\overline\bu\cdot\bfB_\eta^{-\intercal}\mathcal A_i( \partial_t^2 \bfPsi_\eta\big)\partial_i\overline\bu\dxt\nonumber\\
&\quad +\sum_{i=1}^2\int_I\int_{\partial\Omega}J_\eta\partial_t\overline\bu\cdot\bfB_\eta^{-\intercal}\mathcal A_i( \partial_t^2\bfPsi_\eta\big)\overline\bu\, n^i\,\dd\mathscr H^1\dt\nonumber\\
&=J_1+\dots J_5.\label{eq:2308}
\end{align}
We have by \eqref{210and212}, \eqref{218}, \eqref{eq:detPsi} and \eqref{energyEstLinear} that 
\begin{align*}
J_1&\lesssim\int_I\|\partial_t\overline\bu\|_{L^{4}_\bx}\big(1+\|\naby\eta\|^2_{L^\infty_\by}\big)\|\partial_t^2\eta\|_{L^2_\by}\|\overline\bu\|_{L^4_\bx}\dt\\
&\lesssim\int_I\|\partial_t\overline\bu\|_{L^2_\bx}^{1/2}\|\partial_t\overline\bu\|_{W^{1,2}_\bx}^{1/2}\|\partial_t^2\eta\|_{L^2_\by}\|\overline\bu\|_{W^{1,2}_\bx}^{1/2}\|\overline\bu\|_{L^{2}_\bx}^{1/2}\dt\\
&\lesssim\delta\int_I\|\partial_t\overline\bu\|_{W^{1,2}_\bx}^{2}\dt+c(\delta)\int_I\|\partial_t\overline\bu\|_{L^2_\bx}^{2/3}\|\partial_t^2\eta\|_{L^2_\by}^{4/3}\|\overline\bu\|_{W^{1,2}_\bx}^{2/3}\dt\\
&\lesssim\delta\int_I\|\partial_t\overline\bu\|_{W^{1,2}_\bx}^{2}\dt+ c(\delta)\int_I\big(\|\partial_t\overline\bu\|_{L^2_\bx}^{2}\|\overline\bu\|_{W^{1,2}_\bx}^{2}+\|\partial_t^2\eta\|_{L^{2}_\by}^{2}\big)\dt,
\end{align*}
and, arguing as in \eqref{eq:2405},
\begin{align*}
J_2&\lesssim\int_I\int_{-L}^0\|\partial_t\overline\bu\circ\Lambda^{\bfvarphi}\|_{W^{1,2}_\by}\big(1+\|\naby\eta\|_{L^\infty_\by}\big)\|\partial_t^2\eta\|_{L^2_\by}\|\overline\bu\circ\Lambda^{\bfvarphi}\|_{L^\infty_\by}\,\dd z\dt\\
&\lesssim\int_I\int_{-L}^0\|\partial_t\overline\bu\circ\Lambda^{\bfvarphi}\|_{W^{1,2}_\by}\|\overline\bu\circ\Lambda^{\bfvarphi}\|_{L^2_\by}^{1/2}\|\overline\bu\circ\Lambda^{\bfvarphi}\|_{W^{1,2}_\by}^{1/2}\,\dd z\|\partial_t^2\eta\|_{L^2_\by}\dt\\
&\lesssim\int_I\|\partial_t\overline\bu\|_{W^{1,2}_\bx}\|\overline\bu\|_{W^{1,2}_\bx}^{1/2}\|\partial_t^2\eta\|_{L^2_\by}\dt\\
&\lesssim\delta\int_I\|\partial_t\overline\bu\|_{W^{1,2}_\bx}^{2}\dt+ c(\delta)\int_I\|\partial_t^2\eta\|_{L^{2}_\by}^{2}\big(\|\overline\bu\|_{W^{1,2}_\bx}^{2}+1\big)\dt,
\end{align*}
as well as
\begin{align*}
J_3&\lesssim\int_I\int_{-L}^0\|\partial_t\overline\bu\circ\Lambda^{\bfvarphi}\|_{L^{\infty}_\by}\big(1+\|\naby^2\eta\|_{L^2_\by}\big)\|\partial_t^2\eta\|_{L^2_\by}\|\overline\bu\circ\Lambda^{\bfvarphi}\|_{L^\infty_\by}\,\dd z\dt\\
&\lesssim\int_I\int_{-L}^0\|\partial_t\overline\bu\circ\Lambda^{\bfvarphi}\|_{L^{2}_\by}^{1/2}\|\partial_t\overline\bu\circ\Lambda^{\bfvarphi}\|_{W^{1,2}_\by}^{1/2}\|\overline\bu\circ\Lambda^{\bfvarphi}\|_{L^2_\by}^{1/2}\|\overline\bu\circ\Lambda^{\bfvarphi}\|_{W^{1,2}_\by}^{1/2}\,\dd z\|\partial_t^2\eta\|_{L^2_\by}\dt\\
&\lesssim\int_I\|\partial_t\overline\bu\|_{L^{2}_\bx}^{1/2}\|\partial_t\overline\bu\|_{W^{1,2}_\bx}^{1/2}\|\overline\bu\|_{L^{2}_\bx}^{1/2}\|\overline\bu\|_{W^{1,2}_\bx}^{1/2}\|\partial_t^2\eta\|_{L^2_\by}\dt\\
&\lesssim\delta\int_I\|\partial_t\overline\bu\|_{W^{1,2}_\bx}^{2}\dt+ c(\delta)\int_I\big(\|\partial_t^2\eta\|_{L^{2}_\by}^{2}+\|\partial_t\overline{\bu}\|_{L^{2}_\bx}^{2}\big)\big(\|\overline\bu\|_{W^{1,2}_\bx}^{2}+1\big)\dt,
\end{align*}
and
\begin{align*}
J_4&\lesssim\int_I\int_{-L}^0\|\partial_t\overline\bu\circ\Lambda^{\bfvarphi}\|_{L^{\infty}_\by}\big(1+\|\naby\eta\|_{L^\infty_\by}\big)\|\partial_t^2\eta\|_{L^2_\by}\|\nabla\overline\bu\circ\Lambda^{\bfvarphi}\|_{L^2_\by}\,\dd z\dt\\
&\lesssim\int_I\int_{-L}^0\|\partial_t\overline\bu\circ\Lambda^{\bfvarphi}\|_{L^{2}_\by}^{1/2}\|\partial_t\overline\bu\circ\Lambda^{\bfvarphi}\|_{W^{1,2}_\by}^{1/2}\|\overline\bu\circ\Lambda^{\bfvarphi}\|_{W^{1,2}_\by}\,\dd z\|\partial_t^2\eta\|_{L^2_\by}\dt
\\
&\lesssim\int_I\|\partial_t\overline\bu\|_{L^{2}_\bx}^{1/2}\|\partial_t\overline\bu\|_{W^{1,2}_\bx}^{1/2}\|\overline\bu\|_{W^{1,2}_\bx}\|\partial_t^2\eta\|_{L^2_\by}\dt\\
&\lesssim\delta\int_I\|\partial_t\overline\bu\|_{W^{1,2}_\bx}^{2}\dt+ c(\delta)\int_I\big(\|\partial_t^2\eta\|_{L^{2}_\by}^{2}+\|\partial_t\overline{\bu}\|_{L^{2}_\bx}^{2}\big)\big(\|\overline\bu\|_{W^{1,2}_\bx}^{2}+1\big)\dt.
\end{align*}
Finally, it holds
\begin{align*}
J_5&\lesssim\int_I\|\partial_t\overline\bu\|_{L^{4}(\partial\Omega)}\big(1+\|\naby\eta\|_{L^\infty_\by}\big)\|\partial_t^2\eta\|_{L^2_\by}\|\overline\bu\|_{L^4(\partial\Omega)}\dt\\
&\lesssim\int_I\|\partial_t\overline\bu\|_{W^{3/4,2}_\bx}\|\partial_t^2\eta\|_{L^2_\by}\|\overline\bu\|_{W^{3/4,2}_\bx}\dt\\
&\lesssim\int_I\|\partial_t\overline\bu\|_{W^{1,2}_\bx}^{3/4}\|\partial_t\overline\bu\|_{L^{2}_\bx}^{1/4}\|\partial_t^2\eta\|_{L^2_\by}\|\overline\bu\|_{W^{1,2}_\bx}^{3/4}\|\overline\bu\|_{L^{2}_\bx}^{1/4}\dt\\
&\lesssim\delta\int_I\|\partial_t\overline\bu\|_{W^{1,2}_\bx}^{2}\dt+ c(\delta)\int_I\big(\|\partial_t^2\eta\|_{L^{2}_\by}^{2}+\|\partial_t\overline{\bu}\|_{L^{2}_\bx}^{2}\big)\big(\|\overline\bu\|_{W^{1,2}_\bx}^{2}+1\big)\dt.
\end{align*}
As for $\int_I\int_\Omega J_\eta\partial_t\overline{\bu}\bfB_\eta^{-\intercal}\partial_t\bfB_\eta\partial_t\overline\bu\dxt$ we argue similarly \eqref{eq:2308} to terms
\begin{align*}
J_1'&:=\sum_{i=1}^2\int_I\int_\Omega\partial_i J_\eta\partial_t\overline\bu\cdot\bfB_\eta^{-\intercal}\mathcal A_i( \partial_t \bfPsi_\eta\big)\partial_t\overline\bu\dxt,\\
J_2'&:=\sum_{i=1}^2\int_I\int_\Omega J_\eta\partial_t\partial_i\overline\bu\cdot\bfB_\eta^{-\intercal}\mathcal A_i( \partial_t\bfPsi_\eta\big)\partial_t\overline\bu\dxt,\\
J_3'&=\int_I\int_\Omega J_\eta\partial_t\overline\bu\cdot\partial_i\bfB_\eta^{-\intercal}\mathcal A_i( \partial_t \bfPsi_\eta\big)\partial_t\overline\bu\dxt,\\
J_4'&:=\sum_{i=1}^2\int_I\int_\Omega J_\eta\partial_t\overline\bu\cdot\bfB_\eta^{-\intercal}\mathcal A_i( \partial_t \bfPsi_\eta\big)\partial_i\partial_t\overline\bu\dxt,\\
J_5'&:=\sum_{i=1}^2\int_I\int_{\partial\Omega}J_\eta\partial_t\overline\bu\cdot\bfB_\eta^{-\intercal}\mathcal A_i( \partial_t\bfPsi_\eta\big)\partial_t\overline\bu\, n^i\,\dd\mathscr H^1\dt.
\end{align*}
They can be estimated analogously to $J_1,\dots,J_5$ above
obtaining by using the function $\Lambda^\bfvarphi$:
\begin{align*}
J_1'&\lesssim\int_I\|\partial_t\eta\|_{L^2_\by}\|\partial_t\overline\bu\|_{L^2_\bx}\|\partial_t\overline\bu\|_{W^{1,2}_\bx}\dt\lesssim \delta\int_I\|\partial_t\overline\bu\|_{W^{1,2}_\bx}^{2}\dt+c(\delta)\int_I\|\partial_t\overline\bu\|_{L^2_\bx}^{2}\dt,
\end{align*}
\begin{align*}
J_2'+J_4'&\lesssim\int_I\|\partial_t\eta\|_{L^2_\by}\|\partial_t\overline\bu\|_{L^2_\bx}^{1/2}\|\partial_t\overline\bu\|_{W^{1,2}_\bx}^{3/2}\dt\lesssim \delta\int_I\|\partial_t\overline\bu\|_{W^{1,2}_\bx}^{2}\dt+c(\delta)\int_I\|\partial_t\overline\bu\|_{L^2_\bx}^{2}\dt,
\end{align*}
\begin{align*}
J_3'&\lesssim\int_I\|\partial_t\eta\|_{L^2_\by}\big(1+\|\eta\|_{W^{2,2}_\by}\big)\|\partial_t\overline\bu\|_{L^2_\bx}\|\partial_t\overline\bu\|_{W^{1,2}_\bx}\dt\\
&\lesssim \delta\int_I\|\partial_t\overline\bu\|_{W^{1,2}_\bx}^{2}\dt+c(\delta)\int_I\|\partial_t\overline\bu\|_{L^2_\bx}^{2}\dt,
\end{align*}
\begin{align*}
J_5'&\lesssim\int_I\|\partial_t\eta\|_{L^2_\by}\|\partial_t\overline\bu\|_{L^4(\partial\Omega)}^{2}\dt\lesssim\int_I\|\partial_t\overline\bu\|_{W^{1/4,2}(\partial\Omega)}^{2}\dt\\
&\lesssim\int_I\|\partial_t\overline\bu\|_{W^{3/4,2}_\bx}^{2}\dt\lesssim \delta\int_I\|\partial_t\overline\bu\|_{W^{1,2}_\bx}^{2}\dt+c(\delta)\int_I\|\partial_t\overline\bu\|_{L^2_\bx}^{2}\dt.
\end{align*}
Finally, we have
\begin{align}
\int_I\int_\Omega J_\eta\partial_t\overline{\bu}\partial_t\bfB_\eta^{-\intercal}\partial_t\bfB_\eta\overline\bu\dxt
&\lesssim\int_I\int_{-L}^0\|\partial_t\overline\bu\circ\Lambda^{\bfvarphi}\|_{L^\infty_\by}\|\partial_t\naby\eta\|_{L^2_\by}^2\|\overline\bu\circ\Lambda^{\bfvarphi}\|_{L^\infty_\by}\,\dd z\dt\nonumber\\
&\lesssim\int_I\|\partial_t\overline\bu\|_{W^{1,2}_\bx}\|\partial_t\eta\|_{L^2_\by}\|\partial_t\Dely\eta\|_{L^2_\by}\|\overline\bu\|_{W^{1,2}_\bx }\dt\nonumber\\
&\lesssim \delta \int_I\|\partial_t\overline\bu\|_{W^{1,2}_\bx}^{2}\dt+c(\delta)\int_I\|\overline\bu\|_{W^{1,2}_\bx }\|\partial_t\Dely\eta\|_{L^2_\by}^2\dt,\label{similar}
\end{align}
using again the argument from \eqref{eq:2405}.

{\bf Estimate of two integrals depending on $\nabla\bf G$ in \eqref{dt3.8a2d}.} By the definition of $\bfG=\bfB_\eta^{-\intercal}\partial_t\bfB_\eta^\intercal\overline\bu$, we see that
\begin{align*}
\int_I\int_\Omega \partial_t\bfA_\eta\nabla\overline\bu:\nabla\bfG\dxt=\int_I\int_\Omega\partial_t\bfA_\eta\nabla\overline\bu:(\nabla\bfB_\eta^{-\intercal}\partial_t\bfB_\eta\overline\bu+\bfB_\eta^{-\intercal}\partial_t\nabla\bfB_\eta\overline\bu+\bfB_\eta^{-\intercal}\partial_t\bfB_\eta\nabla\overline\bu)\dxt.
\end{align*}
We have the estimate
\begin{align*}
\int_I\int_\Omega\partial_t\bfA_\eta\nabla\overline\bu:\nabla\bfB_\eta^{-\intercal}\partial_t\bfB_\eta\overline\bu\dxt&\lesssim\int_I\|\partial_t\nabla_\by\eta\|_{L^\infty_\by}^2\|\nabla\overline\bu\|_{L^2_\bx}\|\nabla^2_\by\eta\|_{L^4_\by}\|\overline\bu\|_{L^4_\bx}\dt\\
&\lesssim \int_I\|\partial_t\eta\|_{L^2_\by}^{\frac{1}{2}}\|\partial_t\Delta_\by\eta\|_{L^2_\by}^{\frac{3}{2}}\|\nabla\overline\bu\|_{L^2_\bx}^{\frac{3}{2}}\|\nabla^2_\by\eta\|_{L^2_\by}^{\frac{3}{4}}\|\eta\|_{W^{3,2}_\by}^{\frac{1}{4}}\|\overline\bu\|_{L^2_\bx}^{\frac{1}{2}}\dt\\
&\leq \delta\int_I\|\eta\|_{W^{3,2}_\by}^2\dt+c(\delta)\int_I(1+\|\partial_t\Delta_\by\eta\|_{L^2_\by}^2)(1+\|\nabla\overline\bu\|_{L^2_\bx}^2)\dt.
\end{align*}
Using the map $\Lambda^\bfvarphi$ in \eqref{mapneed}, we have
\begin{align*}
&\int_I\int_\Omega\partial_t\bfA_\eta\nabla\overline\bu:\bfB_\eta^{-\intercal}\partial_t\nabla\bfB_\eta\overline\bu\dxt\\
&\lesssim \int_I\int_{-L}^0\|\partial_t\nabla_\by\eta\|_{L^\infty_\by}\|\nabla\overline\bu\circ\Lambda^\bfvarphi\|_{L^2_\by}\|\partial_t\naby^2\eta\|_{L^2_\by}\|\overline\bu\circ\Lambda^\bfvarphi\|_{L^\infty_\by}\dd z\dt\\
&\lesssim \int_I\int_{-L}^0\|\partial_t\nabla_\by\eta\|_{L^2_\by}^{\frac{1}{2}}\|\partial_t\nabla^2_\by\eta\|_{L^2_\by}^{\frac{3}{2}}\|\overline\bu\circ\Lambda^\bfvarphi\|_{L^2_\by}^{\frac{1}{2}}\|\nabla\overline\bu\circ\Lambda^\bfvarphi\|_{L^2_\by}^{\frac{3}{2}}\dd z\dt\\
&\lesssim \int_I\|\partial_t\nabla_\by\eta\|_{L^2_\by}^{\frac{1}{2}}\|\partial_t\nabla^2_\by\eta\|_{L^2_\by}^{\frac{3}{2}}\|\overline\bu\|_{L^2_\bx}^{\frac{1}{2}}\|\nabla\overline\bu\|_{L^2_\bx}^{\frac{3}{2}}\dt\\
&\leq \int_I\|\partial_t\eta\|_{W^{2,2}_\by}^2\dt+\int_I\|\partial_t\eta\|_{W^{2,2}_\by}^2\|\nabla\overline\bu\|_{L^2_\bx}^2\dt.
\end{align*}
Also we estimate:
\begin{align*}
	\int_I\int_\Omega\partial_t\bfA_\eta\nabla\overline\bu:\bfB_\eta^{-\intercal}\partial_t\bfB_\eta\nabla\overline\bu\dxt&\lesssim \int_I\|\partial_t\naby\eta\|_{L^\infty_\by}^2\|\nabla\overline\bu\|_{L^2_\bx}^2\dxt\\
	&\lesssim \int_I\|\partial_t\eta\|_{L^2_\by}^{\frac{1}{2}}\|\partial_t\Dely\eta\|_{L^2_\by}^{\frac{3}{2}}\|\nabla\overline\bu\|_{L^2_\bx}^2\dt\\
	&\lesssim \int_I(1+\|\partial_t\Dely\eta\|_{L^2_\by}^2)\|\nabla\overline\bu\|_{L^2_\bx}^2\dt.
\end{align*}

Now for the second integral on $\nabla\bfG$, we have
\begin{align*}
\int_I\int_\Omega\bfA_\eta\nabla\partial_t\overline\bu:\nabla\bfG\dxt=\int_I\int_\Omega \bfA_\eta\nabla\partial_t\overline\bu:(\nabla\bfB_\eta^{-\intercal}\partial_t\bfB_\eta\overline\bu+\bfB_\eta^{-\intercal}\partial_t\nabla\bfB_\eta\overline\bu+\bfB_\eta^{-\intercal}\partial_t\bfB_\eta\nabla\overline\bu)\dxt.
\end{align*}
We estimate the above integrals as below:
\begin{align*}
&\int_I\int_\Omega\bfA_\eta\nabla\partial_t\overline\bu:\nabla\bfB_\eta^{-\intercal}\partial_t\bfB_\eta\overline\bu\dxt
\\&
\lesssim \int_I\|\nabla\partial_t\overline\bu\|_{L^2_\bx}\|\naby^2\eta\|_{L^6_\by}\|\partial_t\naby\eta\|_{L^6_\by}\|\overline\bu\|_{L^6_\bx}\dt
\\
&\lesssim \int_I\|\nabla\partial_t\overline\bu\|_{L^2_\bx}\|\naby^2\eta\|_{L^2_\by}^{\frac{2}{3}}\|\naby^2\eta\|_{W^{1,2}_\by}^{\frac{1}{3}}\|\partial_t\eta\|_{L^2_\by}^{\frac{1}{3}}\|\partial_t\Dely\eta\|_{L^2_\by}^{\frac{2}{3}}\|\overline\bu\|_{L^2_\bx}^{\frac{1}{3}}\|\nabla\overline\bu\|_{L^2_\bx}^{\frac{2}{3}}\dt
\\
&\leq\delta\int_I\|\nabla\partial_t\overline\bu\|_{L^2_\bx}^2\dt+c(\delta)\int_I\|\naby^2\eta\|_{W^{1,2}_\by}^{\frac{2}{3}}\|\partial_t\Dely\eta\|_{L^2_\by}^{\frac{4}{3}}\|\nabla\overline\bu\|_{L^2_\bx}^{\frac{4}{3}}\dt
\\
&\leq \delta\int_I\|\nabla\partial_t\overline\bu\|_{L^2_\bx}^2\dt+\mu\int_I\|\eta\|_{W^{3,2}_\by}^2\dt+c(\delta, \mu)\int_I\|\partial_t\Dely\eta\|_{L^2_\by}^2\|\nabla\overline\bu\|_{L^2_\bx}^2\dt,
\end{align*}
\begin{align*}
&\int_I\int_\Omega\bfA_\eta\nabla\partial_t\overline\bu:\bfB_\eta^{-\intercal}\partial_t\nabla\bfB_\eta\overline\bu\dxt\\
&\lesssim \int_I\int_{-L}^0\|\nabla\partial_t\overline\bu\circ\Lambda^\bfvarphi\|_{L^2_\by}\|\partial_t\nabla^2_\by\eta\|_{L^2_\by}\|\overline\bu\circ\Lambda^\bfvarphi\|_{L^\infty_\by}\dd z\dt\\
&\lesssim \int_I\|\nabla\partial_t\overline\bu\|_{L^2_\bx}\|\partial_t\nabla^2_\by\eta\|_{L^2_\by}\|\nabla\overline\bu\|_{L^2_\bx}\dt\\
&\leq \delta\int_I\|\nabla\partial_t\overline\bu\|_{L^2_\bx}^2\dt+c(\delta)\int_I\|\partial_t\eta\|_{W^{2,2}_\by}^2\|\nabla\overline\bu\|_{L^2_\bx}^2\dt,
\end{align*}
\begin{align*}
\int_I\int_\Omega\bfA_\eta\nabla\partial_t\overline\bu:\bfB_\eta^{-\intercal}\partial_t\bfB_\eta\nabla\overline\bu\dxt&\lesssim \int_I\|\nabla\partial_t\overline\bu\|_{L^2_\bx}\|\partial_t\naby\eta\|_{L^\infty_\by}\|\nabla\overline\bu\|_{L^2_\bx}\dt
\\
&\leq\delta\int_I\|\nabla\partial_t\overline\bu\|_{L^2_\bx}^2\dt+c(\delta)\int_I\|\partial_t\Dely\eta\|_{L^2_\by}^2\|\nabla\overline\bu\|_{L^2_\bx}^2\dt.
\end{align*}

\textbf{Estimates for the pressure term.}
Finally, we are concerned with the pressure term in \eqref{dt3.8a2d}. We use the weak formulation of the system \eqref{contEqAloneBar}--\eqref{momEqAloneBar} to represent the pressure integrals involved in \eqref{dt3.8a2d} as follows:
\begin{align}
&\int_I\int_\Omega\pi\Div(\partial_t \bfB_\eta^\intercal(\partial_t\overline{\bv}+\bfG))\dx\dt=\int_I\int_\Omega\pi\Div(\bfB_\eta^\intercal\boldsymbol{\Theta}^\eta)\dx\dt\nonumber
\\
&=\int_I\int_\Omega J_{\eta}\partial_t\overline{\bu}\cdot\boldsymbol{\Theta}^\eta\dx\dt+\int_I\int_\Omega\bfA_\eta\nabla\overline{\bu}:\nabla\boldsymbol{\Theta}^\eta\dx\dt-\int_I\int_\Omega \big(\bfh_\eta(\overline{\bu})-J_\eta\partial_t\overline\bu\big)\cdot\boldsymbol{\Theta}^\eta\dx\dt\nonumber
\\
&\quad +\int_I\int_\omega(\partial_t^2\eta+\Dely^2\eta) \bn^\intercal\boldsymbol{\Theta}^\eta\circ\bfvarphi\dy\dt-\int_I\int_\omega g\, \bn^\intercal\boldsymbol{\Theta}^\eta\circ\bfvarphi\dy\dt,\label{later}
\end{align}
where 
\begin{align*}
&\boldsymbol{\Theta}^\eta:=\bfB_\eta^{-\intercal}\partial_t\bfB_\eta^\intercal\partial_t\overline{\bu}+\bfB_\eta^{-\intercal}\partial_t\bfB_\eta^\intercal\bfG=:\boldsymbol{\Theta}^\eta_1+\boldsymbol{\Theta}^\eta_2.
\end{align*}
We note that $\int_I\int_\Omega J_{\eta}\partial_t\overline{\bu}\cdot\boldsymbol{\Theta}^\eta_2\dx\dt$ can be done similarly as \eqref{similar}, while (after an integration by parts) we can procced as for the terms $J_1',\dots, J_5'$ considered above to bound
$\int_I\int_\Omega J_{\eta}\partial_t\overline{\bu}\cdot\boldsymbol{\Theta}^\eta_1\dx\dt$.
In order to proceed we argue yet another time as in \eqref{eq:2405} obtaining
\begin{align*}
\int_I\int_\Omega\bfA_\eta\nabla\overline{\bu}:\nabla\boldsymbol{\Theta}_1^\eta\dx\dt&\lesssim \int_I\int_{-L}^0\|\nabla\overline\bu\circ\Lambda^{\bfvarphi}\|_{L^2_\by}\big(1+\|\naby^2\eta\|_{L^2_\by}\big)\|\partial_t\naby\eta\|_{L^\infty_\by}\|\partial_t\overline\bu\circ\Lambda^{\bfvarphi}\|_{L^\infty_\by}\dd z\dt
\\
&\quad + \int_I\int_{-L}^0\|\nabla\overline\bu\circ\Lambda^{\bfvarphi}\|_{L^2_\by}\|\partial_t\naby^2\eta\|_{L^{2}_\by}\|\partial_t\overline\bu\circ\Lambda^{\bfvarphi}\|_{L^{\infty}_\by}\dd z\dt
\\
&\quad +\int_I\int_{-L}^0\|\nabla\overline\bu\circ\Lambda^{\bfvarphi}\|_{L^2_\by}\|\partial_t\naby\eta\|_{L^\infty_\by}\|\partial_t\nabla\overline\bu\circ\Lambda^{\bfvarphi}\|_{L^2_\by}\,\dd z\dt
\\
&\lesssim \int_I\|\nabla\overline\bu\|_{L^2_\bx}\|\partial_t\Dely\eta\|_{L^2_\by}\|\partial_t\overline\bu\|_{W^{1,2}_\by}\dt
\\
&\lesssim \delta\int_I\|\partial_t\overline\bu\|^2_{W^{1,2}_\by}\dt+c(\delta)\int_I\|\nabla\overline\bu\|^2_{L^2_\bx}\|\partial_t\Dely\eta\|_{L^2_\by}^2,
\end{align*}
on account of \eqref{210and212}, \eqref{218} and \eqref{energyEstLinear}.
Similarly, we have
\begin{align*}
\int_I\int_\Omega\bfA_\eta\nabla\overline{\bu}:\nabla\boldsymbol{\Theta}_2^\eta\dx\dt&\lesssim \int_I\|\nabla\overline\bu\|_{L^2_\bx}^2\big(1+\|\naby\eta\|_{L^\infty_\by}\big)\|\partial_t\naby\eta\|_{L^\infty_\by}^2\dt
\\
&\quad +\int_I\int_{-L}^0\|\nabla\overline\bu\circ\Lambda^{\bfvarphi}\|_{L^{2}_\by}\|\overline\bu\circ\Lambda^{\bfvarphi}\|_{L^{\infty}_\by}\big(1+\|\naby^2\eta\|_{L^2_\by}\big)\|\partial_t\naby\eta\|_{L^\infty_\by}^2\,\dd z\dt
\\
&\quad +\int_I\int_{-L}^0\|\nabla\overline\bu\circ\Lambda^{\bfvarphi}\|_{L^{2}_\by}\|\overline\bu\circ\Lambda^{\bfvarphi}\|_{L^{\infty}_\by}\|\partial_t\naby\eta\|_{L^{\infty}_\by}\|\partial_t\naby^2\eta\|_{L^{2}_\by}\,\dd z\dt
\\
&\lesssim \int_I\|\nabla\overline\bu\|_{L^2_\bx}^2\big(\|\partial_t\Dely\eta\|_{L^2_\by}^2+1\big)\dt.
\end{align*}
Finally, 
\begin{align*}
&\int_I\int_\Omega\big(\bfh_\eta(\overline{\bu})-J_\eta\partial_t\overline\bu\big)\cdot\boldsymbol{\Theta}^\eta_1\dx\dt\\&=
-\int_I\int_\Omega J_\eta\partial_t \overline{\bu}\cdot\bfB_\eta^{-\intercal}\partial_t\bfB_\eta^\intercal
\partial_t\overline{\bu}\dxt-\int_I\int_\Omega
J_\eta \nabx\overline{\bu}\cdot \partial_t \bfPsi_\eta^{-1}\circ \bfPsi_\eta \cdot\bfB_\eta^{-\intercal}\partial_t\bfB_\eta^\intercal
\partial_t\overline{\bu}\dxt\\
&\quad -\int_I\int_\Omega\mathbf{B}_\eta\nabx\overline{\bfv}~\overline{\bu}\cdot\bfB_\eta^{-\intercal}\partial_t\bfB_\eta^\intercal
\partial_t\overline{\bu}\dxt+
\int_I\int_\Omega J_\eta  \bff \circ \bfPsi_\eta\cdot\bfB_\eta^{-\intercal}\partial_t\bfB_\eta^\intercal
\partial_t\overline{\bu}\dxt\\
&=:({\tt J})^1_1+({\tt J})^2_1+({\tt J})^3_1+({\tt J})^4_1,
\end{align*}
where the first term is the most critical one. However, it can be estimates as $J_1',\dots, J_5'$ considered above.
 Note that ${\tt(J)}_1^4$ is trivial such that we only estimate the remainig two integrals. 
It holds (using once more\eqref{210and212}, \eqref{218} and \eqref{energyEstLinear} and employ the argument from \eqref{eq:2405})
\begin{align*}
{\tt (J)}_1^2&\lesssim\int_I\int_{-L}^0
\|\nabx\overline{\bu}\circ\Lambda^{\bfvarphi}\|_{L^2_\by}\|\partial_t\eta\|_{L_\by^2} \|\partial_t\naby\eta\|_{L^\infty_\by}\|
\partial_t\overline{\bu}\circ\Lambda^{\bfvarphi}\|_{L^\infty_\by}\dd z\dt\\
&\lesssim\int_I
\|\nabx\overline{\bu}\|_{L^2_\bx} \|\partial_t\Dely\eta\|_{L^2_\by}\|
\partial_t\overline{\bu}\|_{W^{1,2}_\bx}\dt\\
&\leq\delta \int_I\|\partial_t\overline\bu\|_{W^{1,2}_\bx}^{2}\dt+c(\delta)\int_I\|\overline\bu\|_{W^{1,2}_\bx}^{2}\|\partial_t\Dely\eta\|_{L^2_\by}^2\dt,
\end{align*}
\begin{align*}
{\tt (J)}_1^3&\lesssim\int_I\int_{-L}^0
\|\nabx\overline{\bu}\circ\Lambda^{\bfvarphi}\|_{L^2_\by} \|\partial_t\naby\eta\|_{L^\infty_\by}\|
\overline{\bu}\circ\Lambda^{\bfvarphi}\|_{L^2_\by}\|
\partial_t\overline{\bu}\|_{L^\infty_\by}\,\dd z\dt\\
&\lesssim\int_I
\|\nabx\overline{\bu}\|_{L^2_\bx} \|\partial_t\Dely\eta\|_{L^2_\by}\|
\overline{\bu}\|_{L^2_\bx}\|
\partial_t\overline{\bu}\|_{W^{1,2}_\bx}\dt\\
&\leq\delta \int_I\|\partial_t\overline\bu\|_{W^{1,2}_\bx}^{2}\dt+c(\delta)\int_I\|\overline\bu\|_{W^{1,2}_\bx}^{2}\|\partial_t\Dely\eta\|_{L^2_\by}^2\dt.
\end{align*}
Similarly, \begin{align*}
&\int_I\int_\Omega\big(\bfh_\eta(\overline{\bu})-J_\eta\partial_t\overline\bu\big)\cdot\boldsymbol{\Theta}^\eta_2\dx\dt\\
&=
-\int_I\int_\Omega J_\eta\partial_t \overline{\bu}\cdot(\bfB_\eta^{-\intercal}\partial_t\bfB_\eta^\intercal)^2
\overline{\bu}\dxt-\int_I\int_\Omega
J_\eta \nabx\overline{\bu}\cdot \partial_t \bfPsi_\eta^{-1}\circ \bfPsi_\eta \cdot(\bfB_\eta^{-\intercal}\partial_t\bfB_\eta^\intercal)^2
\overline{\bu}\dxt\\
&\quad -\int_I\int_\Omega\mathbf{B}_\eta\nabx\overline{\bfv}~\overline{\bu}\cdot(\bfB_\eta^{-\intercal}\partial_t\bfB_\eta^\intercal)^2
\overline{\bu}\dxt+
\int_I\int_\Omega J_\eta  \bff \circ \bfPsi_\eta\cdot(\bfB_\eta^{-\intercal}\partial_t\bfB_\eta^\intercal)^2\overline{\bu}\dxt\\
&=:({\tt J})^1_2+({\tt J})^2_2+({\tt J})^3_2+({\tt J})^4_2,
\end{align*}
where
\begin{align*}
({\tt J})^1_2&\lesssim\int_I\int_{-L}^0\|\partial_t\overline\bu\circ\Lambda^{\bfvarphi}\|_{L^\infty_\by}\|\partial_t\naby\eta\|_{L^2_\by}^2\|\overline\bu\circ\Lambda^{\bfvarphi}\|_{L^\infty_\by}\,\dd z\dt\\
&\lesssim\int_I\|\partial_t\overline\bu\|_{W^{1,2}_\bx}\|\partial_t\eta\|_{L^2_\by}\|\partial_t\Dely\eta\|_{L^2_\by}\|\overline\bu\|_{W^{1,2}_\bx }\dt\\
&\lesssim \delta \int_I\|\partial_t\overline\bu\|_{W^{1,2}_\bx}^{2}\dt+c(\delta)\int_I\|\overline\bu\|_{W^{1,2}_\bx }\|\partial_t\Dely\eta\|_{L^2_\by}^2\dt,
\end{align*}
	\begin{align*}
({\tt J})^2_2&\lesssim\int_I\int_{-L}^0\|\nabx\overline{\bfv}\circ\Lambda^{\bfvarphi}\|_{L^2_\by}\|\partial_t\eta\|_{L^2_\by}\|\partial_t\naby\eta\|_{L^\infty_\by}^2\|
\overline{\bu}\circ\Lambda^{\bfvarphi}\|_{L^\infty_\by}\dd z\dt\\
&\lesssim\int_I\|\overline\bu\|_{W^{1,2}_x}^2\|\partial_t\Dely\eta\|_{L^2_\by}^{2}\dt,
\end{align*}
\begin{align*}
({\tt J})^3_2&\lesssim\int_I\int_{-L}^0\|\nabx\overline{\bfv}\circ\Lambda^{\bfvarphi}\|_{L^2_\by}\|\partial_t\naby\eta\|_{L^4_\by}^2\|
\overline{\bu}\circ\Lambda^{\bfvarphi}\|_{L^\infty_\by}^2\dd z\dt\\
&\lesssim\int_I\|\overline\bu\|_{W^{1,2}_\bx}^2\|\partial_t\Dely\eta\|_{L^2_\by}^2\|\overline\bu\|_{L^{2}_\bx }\dt\\
&\lesssim \int_I\|\overline\bu\|_{W^{1,2}_\bx }^2\|\partial_t\Dely\eta\|_{L^2_\by}^2\dt,
\end{align*}
\begin{align*}
({\tt J})^4_2
&\lesssim\int_I\|\overline\bff\|_{L^{2}_\bx}\|\partial_t\nabla_\by\eta\|_{L^\infty_\by}^2\|\overline\bu\|_{L^{2}_\bx }\dt\lesssim \int_I\big(\|\overline\bu\|_{W^{1,2}_\bx }^2+\|\overline\bff\|_{L^{2}_\bx}^2\big)\|\partial_t\Dely\eta\|_{L^2_\by}^2\dt.
\end{align*}

Finally we deal with the integrals in \eqref{later} coming from the boundary of $\bfG$. 
on account of $\nabla\bfPsi_\zeta\bfn=\bfn$ and \eqref{eq:detPsi}
we have
\begin{align*}
 \boldsymbol{\Theta}^\eta\circ\bfvarphi&=\bfB_\eta^{-\intercal}\circ\bfvarphi\,\partial_t\bfB_\eta^\intercal\circ\bfvarphi \,\partial_t^2\eta\bn+\bfB_\eta^{-\intercal}\bfvarphi\,\partial_t\bfB_\eta^\intercal\circ\bfvarphi \bfG\circ\bfvarphi\bn\nonumber\\
 &=J_\eta^{-1}\partial_t J_\eta \partial_t^2\eta \bn+ J_\eta^{-1} J_\eta^{-1} \partial_t J_\eta \partial_t J_\eta \partial_t\eta\bn\\
 &\approx \partial_t\eta\partial_t^2\eta(1+\eta)\bn+(1+\eta)^2|\partial_t\eta|^2\partial_t\eta\bn.
\end{align*}
Then we realize that $\int_I\int_\omega|\partial_t^2\eta|^2\partial_t\eta(1+\eta)\dy\dt$ can be estimated similarly as in \eqref{etaestimate}. We find that
\begin{align*}
\int_I\int_\omega\partial_t^2\eta(1+\eta)^2|\partial_t\eta|^3\dy\dt&\lesssim \int_I\|\partial_t^2\eta\|_{L^2_\by}\|\partial_t\eta\|_{L^6_\by}^3\dt
\\
&\lesssim \int_I\|\partial_t^2\eta\|_{L^2_\by}\|\partial_t\eta\|_{L^2_\by}^2\|\partial_t\naby\eta\|_{L^2_\by}\dt
\\
&\lesssim  \int_I\left(\|\partial_t^2\eta\|_{L^2_\by}^2+\|\partial_t\Dely\eta\|_{L^2_\by}^2\right)\dt,
\end{align*}
where we used the $1D$ interpolation inequality:
$$\|\partial_t\eta\|_{L^6_\by}\lesssim \|\partial_t\eta\|_{L^2_\by}^{\frac{2}{3}}\|\partial_t\eta\|_{W^{1,2}_\by}^{\frac{1}{3}}. $$
We obtain further by using \eqref{energyEstLinear} that
\begin{align}
\int_I\int_\omega\Dely^2\eta (1+\eta)^2|\partial_t\eta|^3\dy\dt& =-\int_I\int_{\omega}\naby\Dely\eta
\cdot\naby\big((1+\eta)^2|\partial_t\eta|^3\big)\dx\dt\nonumber\\
&\lesssim \int_I\|\eta\|_{W^{3,2}_\by}\big(\|\partial_t\eta\|_{L^2_\by}\|\partial_t\eta\|_{L^\infty_\by}^2
+\|\partial_t\naby\eta\|_{L^\infty_\by}\|\partial_t\eta\|_{L^2_y}\|\partial_t\eta\|_{L^\infty_\by}\big)\dt\nonumber
\\
&\lesssim \int_I\|\eta\|_{W^{3,2}_\by}\big(\|\partial_t\naby\eta\|_{L^2_\by}
+\|\partial_t\eta\|_{L^2_\by}^{1/4}\|\partial_t\Dely\eta\|_{L^2_\by}^{3/4}\|\partial_t\eta\|_{L^2_\by}^{3/4}\|\partial_t\Dely\eta\|_{L^2_\by}^{1/4}\big)\dt\nonumber
\\
&\lesssim\,\delta\int_I\|\eta\|_{W^{3,2}_\by}^2\dt +c(\delta)\int_I\|\partial_t\Dely\eta\|_{L^2_\by}^2
\dt.\label{needagain}
\end{align}
As for the remaining integral containing $ \Dely^2\eta$ we procced by a formal computation which can be made rigorous by projecting $\eta$ on a finite sum of Fourier modes. We obtain
\begin{align*}
&\int_I\int_\omega\Dely^2\eta\,\partial_t^2\eta\partial_t\eta(1+\eta)\dy\dt
=\tfrac{1}{2}\int_I\int_\omega\Dely^2\eta\,\partial_t|\partial_t\eta|^2(1+\eta)\dy\dt\\
&=-\tfrac{1}{2}\int_I\int_\omega\Dely^2\partial_t\eta\,|\partial_t\eta|^2(1+\eta)\dy\dt-\tfrac{1}{2}\int_I\int_\omega\Dely^2\eta\,|\partial_t\eta|^2\partial_t(1+\eta)\dy\dt\\&\quad +\frac{1}{2}\int_\omega\Dely^2\eta\,|\partial_t\eta|^2(1+\eta)\dy\bigg|_0^T\\
&=:(\tt{I})_1+(\tt{I})_2+(\tt{I})_3,
\end{align*}
where (using \eqref{energyEstLinear})
\begin{align*}
({\tt I})_1&=-\frac{1}{2} \int_I\int_\omega\Dely\partial_t\eta\,\Dely\big(|\partial_t\eta|^2(1+\eta)\big)\dy\dt
\\
&\lesssim \int_I\|\partial_t\Dely\eta\|_{L^2_\by}\|\partial_t\naby\eta\|_{L^4_\by}^2 \dt+\int_I\|\partial_t\Dely\eta\|_{L^2_\by}\|\partial_t\Dely\eta\|_{L^3_\by}\|\partial_t\eta\|_{L^6_\by}\dt\\&\quad +\int_I\|\Dely\eta\|_{L^2_\by}\|\partial_t\eta\|_{L^\infty_\by}^2\|\partial_t\Dely\eta\|_{L^2_\by}\dt\\
&\lesssim \int_I\|\partial_t\Dely\eta\|_{L^2_\by}\|\partial_t\eta\|_{W^{1/2,2}_\by}\|\partial_t\eta\|_{W^{7/4,4}_\by}\dt+\int_I\|\partial_t\Dely\eta\|_{L^2_\by}^{\frac{4}{3}}\|\partial_t\Dely\eta\|^{\frac{2}{3}}_{W^{1/4,2}_\by}\|\partial_t\eta\|_{W^{1/2,2}_\by}\dt\\&\quad +\int_I\|\partial_t\Dely\eta\|_{L^2_\by}^2\dt
\\
&\lesssim \delta \int_I\|\partial_t\Dely\eta\|_{W^{1/4,2}_\by}^2\dt+c(\delta)\int_I\|\partial_t\Dely\eta\|_{L^2_\by}^2(\|\overline{\bu}\|_{W^{1,2}_\bx}^2+1)\dt,
\end{align*}
Notice that we have
\begin{align*}
({\tt I})_2&=\int_I\int_\omega\naby\Dely\eta\cdot\naby\big(|\partial_t\eta|^2\partial_t(1+\eta)\big)\dy\dt=3\int_I\int_\omega\naby\Dely\eta|\partial_t\eta|^2\partial_t\naby\eta\dy\dt,
\end{align*}
which has been estimated in the second part of \eqref{needagain}.
For $({\tt I})_3$, we have
\begin{align*}
({\tt I})_3&=\int_\omega\Dely\eta\,\Dely \big(|\partial_t\eta|^2(1+\eta)\big)\dy\bigg|_0^T
\\
&\lesssim 
\sup_I\|\Dely\eta\|_{L^2_\by}\|\partial_t\Dely\eta\|_{L^2_\by}\|\partial_t\eta\|_{L^\infty_\by}+ \sup_I\|\Dely\eta\|_{L^2_\by}^2\|\partial_t\eta\|^2_{L^\infty_\by}+ \sup_I\|\Dely\eta\|_{L^2_\by}\|\partial_t\naby\eta\|_{L^4_\by}^2\\
&\lesssim \sup_I\|\partial_t\Dely\eta\|_{L^2_\by}\|\partial_t\eta\|_{L^2_\by}^{\frac{3}{4}}\|\partial_t\Dely\eta\|_{L^2_\by}^{\frac{1}{4}}+\sup_I\|\partial_t\eta\|_{L^2_\by}^{\frac{3}{2}}\|\partial_t\Dely\eta\|_{L^2_\by}^{\frac{1}{2}}+\sup_I\|\partial_t\eta\|_{L^2_\by}^{\frac{3}{4}}\|\partial_t\Dely\eta\|_{L^2_\by}^{\frac{5}{4}}\\
&\lesssim\delta \sup_I\|\partial_t\Dely\eta\|_{L^2_\by}^2+c(\delta).
\end{align*}
Moreover, we have
\begin{align*}
&\int_I\int_\omega g \left(\partial_t\eta\partial_t^2\eta(1+\eta)+(1+\eta)^2|\partial_t\eta|^3\right)\dy\dt
\\&\leq\int_I\|g\|_{L^2_\by}\big(\|\partial_t^2\eta\|_{L^2_\by}\|\partial_t\eta\|_{L^\infty_\by}+\|\partial_t\eta\|_{L^2_\by}\|\partial_t\eta\|_{L^\infty_\by}^2\big)\dt\\
&\lesssim \int_I\|g\|_{L^2_\by}\left(\|\partial_t^2\eta\|_{L^2_\by}\|\partial_t\Dely\eta\|_{L^2_\by}+\|\partial_t\eta\|_{L^2_\by}^2\|\partial_t\naby\eta\|_{L^2_\by}\right)\dt\\
&\lesssim \int_I\left(\|g\|_{L^2_\by}^2+1\right)\left(1+\|\partial_t\Dely\eta\|_{L^2_\by}^{2}+\|\partial_t^2\eta\|_{L^2_\by}^2\right)\dt,
\end{align*}
using that $g\in W^{1,2}_tL^2_\by$.

Now setting 
\begin{align*}
\mathcal E(t)&=\|\partial_t^2\eta(t)\|_{L^2_\by}^2+\|\partial_t\Dely\eta(t)\|_{L^2_\by}^2+\|\partial_t\bu(t)\|_{L^2_\bx}^2+\|\nabla\bu(t)\|_{L^2_\bx}^2
,\\
H(t)&=\|\bff(t)\|^2_{L^2_\bx}+\|\partial_t\bff(t)\|^2_{L^2_\bx}+\|g(t)\|^2_{L^2_\by}+\|\partial_tg(t)\|^2_{L^2_\by},
\end{align*}
collecting the previous estimates, choosing $\delta$ small enough and using \eqref{eq:MuSc'2D},
we have shown that
\begin{align*}
&\mathcal E(t)+\int_0^t\big(\|\eta\|_{W^{3,2}_\by}^2+\|\partial_t\Dely\eta\|_{W^{1/4,2}_\by}^2+\|\partial_t\nabla\overline\bu\|^2_{L^2_\bx}\big)\dt\\&\lesssim \mathcal E(0)+1+\int_0^tH(s)\ds+\int_0^t(1+\|\overline{\bu}\|_{W^{1,2}_\bx}^2)\mathcal E(s)\ds.
\end{align*}
We now apply Gr\"{o}nwall's lemma to complete the proof.
\end{proof}

\section*{Acknowledgments}
D. Breit and P. R. Mensah are supported by Grant BR 4302/5-1 (543675748) and ME 6391/1-1 (543675748), respectively, by the German Research Foundation (DFG).

S. Schwarzacher and P. Su are supported by the ERC-CZ Grant CONTACT LL2105 funded by the Ministry of Education, Youth and Sport of the Czech Republic. S. Schwarzacher also acknowledges the support of the VR
Grant 2022-03862 of the Swedish Research Council. P. Su is also supported by PEPS ``Jeunes Chercheuses et Jeunes Chercheurs
de l'Insmi" from CNRS and the Sophie Germain program of the Fondation Math\'ematique Jacques Hadamard (FMJH).

\section*{Compliance with Ethical Standards}
\smallskip
\par\noindent
{\bf Conflict of Interest}. The authors declare that they have no conflict of interest.

\smallskip
\par\noindent
{\bf Data Availability}. Data sharing is not applicable to this article as no datasets were generated or analysed during the current study.

\end{document}